\documentclass[review]{elsarticle}

\usepackage{lineno,hyperref}
\modulolinenumbers[5]

\journal{Journal of \LaTeX\ Templates}

\usepackage{amsmath}
\usepackage{amssymb}
\usepackage{amsfonts}
\usepackage{amsthm}
\usepackage{mathtools}
\mathtoolsset{
above-intertext-sep = -1ex
}
 \usepackage{booktabs}
  \usepackage{pdflscape}
\usepackage[utf8]{inputenc}
\usepackage[T1]{fontenc}

\usepackage[sans]{dsfont}

\usepackage[left=2.5cm, right=2.5cm, top=2.5cm, bottom=2.5cm]{geometry}

\usepackage[%
cal=euler,
bb=fourier,
scr=zapfc
]
{mathalfa}

\usepackage{caption}
\usepackage{subcaption}
\usepackage{graphicx}
\graphicspath{{./Figures/}} 
\usepackage{acronym}
\usepackage{latexsym}
\usepackage{paralist}
\usepackage{wasysym}
\usepackage{framed}

\usepackage[dvipsnames,svgnames]{xcolor}
\colorlet{MyBlue}{DodgerBlue!75!Black}
\colorlet{MyGreen}{DarkGreen!95!Black}

\usepackage{hyperref}
\hypersetup{
colorlinks=true,
linktocpage=true,
pdfstartview=FitH,
breaklinks=true,
pdfpagemode=UseNone,
pageanchor=true,
pdfpagemode=UseOutlines,
plainpages=false,
bookmarksnumbered,
bookmarksopen=false,
bookmarksopenlevel=1,
hypertexnames=true,
pdfhighlight=/O,
urlcolor=MyBlue!60!black,linkcolor=MyBlue!70!black,citecolor=DarkGreen!70!black, 
pdftitle={},
pdfauthor={},
pdfsubject={},
pdfkeywords={},
pdfcreator={pdfLaTeX},
pdfproducer={LaTeX with hyperref}
}

\numberwithin{equation}{section}  
\usepackage[sort&compress,capitalize,nameinlink]{cleveref}
\crefname{example}{Ex.}{Exs.}

\crefrangeformat{equation}{\upshape(#3#1#4)\textendash(#5#2#6)}

\usepackage{tikz} 
\usepackage{wrapfig}

\newcommand{\dd}{\:d}

\newcommand{\eps}{\varepsilon}

\newcommand{\dif}{\dd}

\DeclareMathOperator*{\argmin}{argmin}
\DeclareMathOperator*{\argmax}{argmax}

\DeclareMathOperator{\bd}{bd}
\DeclareMathOperator{\cl}{cl}

\DeclareMathOperator{\diag}{diag}
\DeclareMathOperator*{\Diag}{Diag}

\DeclareMathOperator{\dist}{dist}

\DeclareMathOperator{\dom}{dom}

\DeclareMathOperator{\gr}{graph}
\DeclareMathOperator{\Int}{int}
\DeclareMathOperator{\rint}{relint}

\DeclareMathOperator{\tr}{tr}

\DeclareMathOperator{\Id}{Id}

\newcommand{\ce}{\mathtt{e}}

\DeclareMathOperator{\prox}{Prox}
\newcommand{\bA}{{\mathbf A}}
\newcommand{\BB}{{\mathbf B}}

\newcommand{\DD}{{\mathbf D}}
\newcommand{\bJ}{{\mathbf J}}
\newcommand{\XX}{\mathbf{X}}
\newcommand{\YY}{\mathbf{Y}}

\newcommand{\UU}{\mathbf{U}}

\newcommand{\bI}{\mathbf{I}}
\newcommand{\bM}{\mathbf{M}}

\newcommand{\bK}{\mathbf{K}}

\renewcommand{\iff}{\Leftrightarrow}

\renewcommand{\emptyset}{\varnothing}

\newcommand{\scrH}{\mathcal{H}}

\newcommand{\scrL}{\mathcal{L}}

\newcommand{\scrP}{\mathcal{P}}

\newcommand{\setC}{\mathsf{C}}

\newcommand{\setE}{\mathsf{E}}

\newcommand{\setS}{\mathsf{S}}

\newcommand{\setV}{\mathsf{V}}
\newcommand{\setW}{\mathsf{W}}
\newcommand{\setX}{\mathsf{X}}

\newcommand{\setZ}{\mathsf{Z}}

\newcommand{\Rn}{\R^n}

\newcommand{\R}{\mathbb{R}}

\newcommand{\N}{\mathbb{N}}

\DeclareMathOperator{\NC}{\mathsf{NC}}
\DeclareMathOperator{\TC}{\mathsf{TC}}

\newcommand{\BL}{\mathsf{BL}}

\usepackage{algorithmic}
\usepackage[ruled,vlined]{algorithm2e}

\theoremstyle{plain}
\newtheorem{theorem}{Theorem}

\newtheorem{lemma}[theorem]{Lemma}
\newtheorem{proposition}[theorem]{Proposition}

\theoremstyle{definition}
\newtheorem{definition}[theorem]{Definition}
\newtheorem{assumption}{Assumption}

\renewcommand\qed{\hfill\small$\blacksquare$}

\theoremstyle{remark}
\newtheorem{remark}{Remark}
\newtheorem*{remark*}{Remark}
\newtheorem*{notation*}{Notational remark}
\newtheorem{example}{Example}

\numberwithin{theorem}{section}
\numberwithin{remark}{section}
\numberwithin{example}{section}

\DeclarePairedDelimiter{\inner}{\langle}{\rangle}
\DeclarePairedDelimiter{\norm}{\lVert}{\rVert}

\newacro{HBA}[HBA]{Hessian-barrier algorithm}
\newacro{AHBA}[AHBA]{Adaptive-Hessian-barrier algorithm}
\newacroplural{NE}[NE]{Nash equilibria}
\newacro{VI}{variational inequality}
\newacroplural{VI}[VIs]{variational inequalities}
\newacro{iid}[i.i.d.]{independent and identically distributed}
\newacro{FOM}{First-order method}
\newacro{LMO}{Linear Minimization Oracle}
\newacro{DGF}{Distance Generating Function}
\newacro{LLOO}{Local Linear Optimization Oracle}
\newacro{FW}{Frank-Wolfe}
\newacro{CG}{Conditional Gradient}
\newacro{SC}{self-concordant}
\newacro{GSC}{generalized self-concordant}
\newacro{DGF}{distance generating function}

\begin{document}

\begin{frontmatter}

\title{First-Order Methods for Convex Optimization
}

\author[WIAS]{Pavel Dvurechensky}
\address[WIAS]{Weierstrass Institute for Applied Analysis and Stochastics, Mohrenstr. 39, 10117 Berlin, Germany}
\ead{pavel.dvurechensky@wias-berlin.de}

\author[Technion]{Shimrit Shtern}
\address[Technion]{Faculty of Industrial Engineering and Management, Technion - Israel Institute of Technology, Haifa, Israel}
\ead{shimrits@technion.ac.il}

\author[UM]{Mathias Staudigl}
\address[UM]{Maastricht University, Department of Data Science and Knowledge Engineering (DKE) and Mathematics Centre Maastricht (MCM), Paul-Henri Spaaklaan 1, 6229 EN Maastricht, The Netherlands \corref{mycorrespondingauthor}}
\cortext[mycorrespondingauthor]{Mathias Staudigl}
\ead{m.staudigl@maastrichtuniversity.nl}

\begin{abstract}
First-order methods for solving convex optimization problems have been at the forefront of mathematical optimization in the last 20 years. The rapid development of this important class of algorithms is motivated by the success stories reported in various applications, including most importantly machine learning, signal processing, imaging and control theory. First-order methods have the potential to provide low accuracy solutions at low computational complexity which makes them an attractive set of tools in large-scale optimization problems. In this survey we cover a number of key developments in gradient-based optimization methods. This includes non-Euclidean extensions of the classical proximal gradient method, and its accelerated versions. Additionally we survey recent developments within the class of projection-free methods, and proximal versions of primal-dual schemes. We give complete proofs for various key results, and highlight the unifying aspects of several optimization algorithms. 
\end{abstract}


\end{frontmatter}

\linenumbers

\section{Introduction}
\label{sec:introduction}
The traditional standard in convex optimization was to translate a problem into a conic program and solve it using a primal-dual interior point method (IPM). The monograph \cite{NesNem94} was instrumental in setting this standard. The primal-dual formulation is a mathematically elegant and powerful approach as these conic problems can then be solved to high accuracy when the dimension of the problem is of moderate size. This philosophy culminated into the development of a robust technology for solving convex optimization problems which is nowadays the computational backbone of many specialized solution packages like MOSEK \cite{Mosek00}, or SeDuMi \cite{Stu99}. However, in general, the iteration costs of interior point methods grow non-linearly with the problem's dimension. As a result, as the dimension $n$ of optimization problems grows, off-the shelve interior point methods eventually become impractical. As an illustration, the computational complexity of a single step of many standardized IPMs scales like $n^{3}$, corresponding roughly to the complexity of inverting an $n\times n$ matrix. This means that for already quite small problems of size like $n=10^{2}$, we would need roughly $10^{6}$ arithmetic operations just to compute a single iterate. From a practical viewpoint, such a scaling is not acceptable. An alternative solution approach, particularly attractive for such "large-scale" problems, are \emph{first-order methods} (FOMs). These are iterative schemes with computationally cheap iterations usually known to yield low-precision solutions within reasonable computation time. The success-story of FOMs went hand-in-hand with the fast progresses made in data science, analytics and machine learning. In such data-driven optimization problems, the trade-off between fast iterations and low accuracy is particularly pronounced, as these problems usually feature high-dimensional decision variables. In these application domains precision is usually considered to be a subordinate goal because of the inherent randomness of the problem data, which makes it unreasonable to minimize with accuracy below the statistical error. 

The development of first-order methods for convex optimization problems is still a very vibrant field, with a lot of stimulus from the already mentioned applications in machine learning, statistics, optimal control, signal processing, imaging, and many more, see e.g. review papers on optimization for machine learning \cite{jain2017non-convex,curtis2017optimization,wright2018optimization}. Naturally, any attempt to try to survey this lively scientific field is already doomed from the beginning to be a failure, if one is not willing to make restrictions on the topics covered. Hence, in this survey we tried to give a largely self-contained and concise summary of some important families of FOMs, which we believe have had an ever-lasting impact on the modern perspective of continuous optimization. Before we give an outline what is covered in this survey, it is therefore maybe fair to mention explicitly, what is NOT covered in the pages to come. One major restriction we imposed on ourselves is the concentration on \emph{deterministic} optimization algorithms. This is indeed a significant cut in terms of topics, since the field of stochastic optimization and randomized algorithms has particularly been at the forefront of recent progresses made. Nonetheless, we made this cut by purpose, since most of the developments within stochastic optimization algorithms are based on deterministic counterparts, and actually in many cases one can think of deterministic algorithms as the mean-field equivalent of a stochastic optimization technique. As well-known example, we can mention the celebrated stochastic approximation theory initiated by Robbins and Monro \cite{RobMon51}, with its deep connection to deterministic gradient descent. See \cite{Kus84,BenMetPri90,LjuPflWal12}, for classical references from the point of view of systems theory and optimization, and \cite{Ben98} for its deep connection with deterministic dynamical systems. This link has gained significant relevance in various stochastic optimization models recently \cite{MerSta18,MerStaJOTA18,DucRua18,DavDrusKakLee20}. An excellent reference on stochastic optimization is \cite{ShaDenRus09} and \cite{PicPfl14}. Furthermore, we excluded the very important class of alternating minimization methods, such as block-coordinate descent, and variations of the same idea. These methods are fundamental in distributed optimization, and lay the foundations for the now heavily investigated randomized algorithms, exploiting the block-structure of the model to achieve acceleration and reduce the overall computational complexity.  Section 14 in the beautiful book by Amir Beck \cite{Beck17} gives a thorough account of these methods and we urge the interested reader to start reading there. \\

So, what is it that we actually do in this article? Four seemingly different optimization algorithms are surveyed, all of which belong now to the standard toolkit of mathematical programmers. After introducing the (standard) notation that will be used in this survey, we give a precise formulation of the model problem for which modern convex optimization algorithms are developed. In particular, we focus on the general composite convex optimization model, including a smooth and one non-smooth term. This model is rich enough to capture a significant class of convex optimization problems. Non-smoothness is an important feature of the model, as it allows us to incorporate constraints via penalty and barrier functions. 
An efficient way to deal with non-smoothness is provided by the use of \emph{proximal operators}, a key methodological contribution born within convex analysis (see \cite{RW98} for an historical overview). Section \ref{sec:MD} introduces the general non-Euclidean proximal setup, which describes the mathematical framework within which the celebrated \emph{Mirror Descent} and \emph{Bregman proximal gradient methods} are analyzed nowadays. This set of tools has been extremely popular in online learning and convex optimization \cite{Bub15,JudNemML1,JudNemML2}. The main idea behind this technology is to exploit favorable structure in the problem's geometry to boost the practical performance of gradient-based methods. The proximal revolution has also influenced the further development of classical primal-dual optimization methods based on augmented Lagrangians. We review proximal variants of the celebrated Alternating Direction Method of Multipliers (ADMM) in Section \ref{sec:ADMM}. We then move on to give in-depth presentation of projection-free optimization methods based on linear minimization oracles, the classical \emph{Conditional Gradient} (CG) (a.k.a Frank-Wolfe) method and its recent variants. CG gained extreme popularity in large-scale optimization, mainly because of its good scalability properties and small iteration costs. Conceptually, it is an interesting optimization method, as it allows us to solve convex programming problems with complicated geometry on which proximal operators are not easy to evaluate. This, in fact, applies to many important domains, like the Spectrahedron, or domains defined via intersections of several half spaces. CG is also relevant when the iterates should preserve structural features of the desired solution, like sparsity. Section \ref{sec:CG} gives a comprehensive account of this versatile tool. All the methods we discussed so far generally provide sublinear convergence guarantees in terms of function values with iteration complexity of $O(1/{\eps})$ In his influential paper \cite{Nes83}, Nesterov published an optimal method with iteration complexity of $O(1/\sqrt{\eps})$ to reach an $\eps$-optimal solution. This was the starting point for the development of acceleration techniques for given FOMs. Section \ref{sec:accelerated} summarizes the recent developments in this field. While writing this survey, we tried to give a holistic presentation of the main methods in use. At various stages in the survey, we establish connections, if not equivalences, between various methods. For many of the key results we provide self-contained proofs to illustrate the main lines of thought in developing FOMs for convex optimization problems.

\paragraph{Notation} We use standard notation and concepts from convex and variational analysis, which, unless otherwise specified, can all be found in the monograph \cite{RW98,HirLem01,BauCom16}. Throughout this article, we let $\setV$ represent a finite-dimensional vector space of dimension $n$ with norm $\norm{\cdot}$. We will write $\setV^{\ast}$ for the (algebraic) dual space of $\setV$ with duality pairing $\inner{y,x}$ between $y\in\setV^{\ast}$ and $x\in\setV$. 
The dual norm of $y\in\setV^{\ast}$ is $\norm{y}_{\ast}=\sup\{\inner{y,x}\vert\quad \norm{x}\leq 1\}$. The set of proper lower semi-continuous functions $f:\setV\to(-\infty,\infty]$ is denoted as $\Gamma_{0}(\setV)$. The (effective) domain of a function $f\in\Gamma_{0}(\setV)$ is defined as $\dom f=\{x\in\setV\vert f(x)<\infty\}$. For a given continuously differentiable function $f:\setC\subseteq \setV\to\R$ we denote its gradient vector 
\[
\nabla f(x_{1},\ldots,x_{n})=\left(\frac{\partial f}{\partial x_{1}},\ldots,\frac{\partial f}{\partial x_{n}}\right)^{\top}.
\]
The subdifferential at a point $x\in\setC\subseteq\setV$ of a convex function $f:\setV\to\R\cup\{+\infty\}$ is denoted as 
\begin{align*}
\partial f(x)=\{p\in\setV^{\ast}\vert f(y)\geq f(x)+\inner{p,y-x}\quad\forall y\in\setV\}.
\end{align*}
The elements of $\partial f(x)$ are called subgradients. The subdifferential is the set-valued mapping $\partial f:\setV\to 2^{\setV^{\ast}}$. 

As a notational convention, we write matrices in bold capital fonts. Given some set $\setX\subseteq\setV$, denote its relative interior as $\rint(\setX)$. Recall that, if the dimension of the set $\setX$ agrees with the dimension of the ground space $\setV$, then the relative interior coincides with the topological interior, which we denote as $\Int(\setX)$. Hence, the two notions differ only in situations where $\setX$ is contained in a lower-dimensional submanifold. We denote the closure as $\cl(\setX)$. The boundary of $\setX$ is defined in the usual way $\bd(\setC)=\cl(\setC)\setminus\Int(\setC)$. 

\section{Composite convex optimization}
\label{sec:Model}

In this survey we focus on the generic optimization problem 
\begin{equation}\label{eq:Opt}\tag{P}
\min_{x\in\setX}\{\Psi(x)=f(x)+r(x)\},
\end{equation}
where 
\begin{itemize}
\item $\setX\subseteq\setV$ is a nonempty closed convex set embedded in a finite-dimensional real vector space $\setV$;
\item $f(\cdot)$ is $L_{f}$-smooth meaning that it is differentiable on $\setV$ with a $L_{f}$-Lipschitz continuous gradient on $\setX$:
\begin{equation}
\label{eq:fsmooth}
\norm{\nabla f(x)-\nabla f(x')}_{\ast}\leq L_{f}\norm{x-x'}\qquad\forall x,x'\in\setX.
\end{equation}

\item $r\in\Gamma_{0}(\setV)$ and $\mu$-strongly convex on $\setV$ for some $\mu\geq 0$ with respect to a norm $\norm{\cdot}$ on $\setV$. This means that for all $x,y\in\dom r$, and any selection $r'(x)\in\partial r(x)$, we have 
$$
r(y)\geq r(x)+\inner{r'(x),y-x}+\frac{\mu}{2}\norm{x-y}^{2}.
$$
\end{itemize}

Finally, we are interested in problems with a well-posed problem formulation.
\begin{assumption}
$\dom r\cap\setX\neq\emptyset$. 
\end{assumption}
The most important examples of function $r$ are as follows:
\begin{itemize}
\item $r$ is an indicator function of a closed convex set $\setC$ with $\setC\cap\setX\neq\emptyset$:
\begin{equation}\label{eq:indicator}
r(x)=\delta_{\setC}(x):=\left\{\begin{array}{cc} 
0 & \text{ if }x\in\setC,\\
+\infty & \text{if }x\notin\setC.
\end{array}\right.
\end{equation}
\item $r$ is a self-concordant barrier \cite{NesNem94,Nes18} for a closed convex set $\setC\subset\setV$ with $\setC\cap\setX\neq\emptyset$. 
\item $r$ is a nonsmooth convex function with relatively simple structure. For example, it could be a norm regularization like the celebrated $\ell_{1}$-regularizer 
\begin{equation}
r(x)=\left\{\begin{array}{cc} 
\norm{x}_{1} & \text{ if }\norm{x}_{1}\leq R,\\
+\infty & \text{ else}
\end{array}\right. .
\end{equation}
This regularizer plays a fundamental role in high-dimensional statistics \cite{BuhvdG11} and signal processing \cite{DauDefDeM04,BruDonEla09}.
\end{itemize}

For characterizing solutions to our problem \eqref{eq:Opt}, define the \emph{tangent cone} associated with the closed convex set $\setX\subseteq\setV$ as
\[
\TC_{\setX}(x):=\left\{\begin{array}{cc}
\{v=t(x' -x)\vert x'\in\setX,t\geq 0\}\subseteq\setV & \text{if }x\in\setX,\\
\emptyset & \text{else.}
\end{array}\right.
\]
and the \emph{normal cone} associated to the closed convex set $\setX$ at $x\in\setX$ as the polar cone of $\TC_{\setX}(x)$:
\[
\NC_{\setX}(x):=\left\{\begin{array}{cc} 
\{p\in\setV^{\ast}\vert \sup_{v\in\TC(x)} \inner{p,v}\leq 0\} & \text{if } x\in\setX\\
\emptyset& \text{else.}
\end{array}
\right.
\]
We remark that $\partial \delta_{\setX}(x)=\NC_{\setX}(x)$ for all $x\in\setX$.

Given the feasible set $\setX\subseteq\setV$, we denote the minimal function value 
\begin{equation}\label{eq:Psimin}
\Psi_{\min}(\setX)=\inf_{x\in\setX}\Psi(x).
\end{equation}
We are focussing in this survey on problems which are solvable. This justifies the next assumption. 
\begin{assumption}\label{ass:S}
$\setX^{\ast}=\{x\in\setX\vert \Psi(x)=\Psi_{\min}(\setX)\}\neq\emptyset.$
\end{assumption}
Given the standing hypothesis on the functions $f$ and $r$, it is easy to see that $\setX^{\ast}$ is always a closed convex set. Moreover, if $\mu>0$, then problem \eqref{eq:Opt} is \emph{strongly convex}, and so $\setX^{\ast}$ is a singleton.\\

Given the structural assumptions of the model problem \eqref{eq:Opt}, the sum rule of subgradients implies that all points in the solution set $\setX^{\ast}$ satisfy the monotone inclusion (Fermat's rule)
\begin{equation}\label{eq:Fermat}
0\in\nabla f(x^{\ast})+\partial r(x^{\ast})+\NC_{\setX}(x^{\ast}).
\end{equation}
This means that there exists $\xi\in\partial r(x^{\ast})$ such that 
\begin{equation}\label{eq:VI-f}
\inner{\nabla f(x^{\ast})+\xi,v}\geq 0\qquad\forall v\in\TC_{\setX}(x^{\ast}).
\end{equation}

The structured composite optimization problem \eqref{eq:Opt} has attracted a lot of interest in convex programming over the last 20 years motivated by a number of important applications. This led to a rich interplay between convex programming on the one hand and machine learning and signal/image processing on the other hand. 
Indeed, several work-horse models in these applied fields are of the composite type
\begin{equation}\label{eq:structured}
\Psi(x)=g(\bA x)+r(x) 
\end{equation}
where $g:\setE\to\R$ is a smooth function defined on a finite-dimensional set $\setE$ (usually of lower dimension than $\setV$), and $\bA\in\BL(\setV,\setE)$ is bounded linear operator mapping points $x\in\setV$ to elements $\bA x\in\setE$. Convexity allows us to switch between primal and dual formulations freely, so that the above problem can be equivalently considered as a convex-concave minimax problem 
\begin{equation}\label{eq:saddle}
\min_{x\in\setX}\max_{y\in\setE}\{r(x)+\inner{\bA x,y}-g^{\ast}(y)\}
\end{equation}
Such minimax problems have been of key importance in signal processing and machine learning \cite{JudKilNem13,JudNemML1,JudNemML2}, game theory \cite{Sor00}, decomposition methods \cite{Tse91} and its very recent innovation around generative adversarial networks \cite{Good-GANS}.

Another canonical class of optimization problems in machine learning is the finite-sum model 
\begin{equation}
\Psi(x)=\frac{1}{N}\sum_{i=1}^{N}f_{i}(x)+r(x),
\end{equation}
which comes from supervised learning, where $f_{i}(x)$ corresponds to the loss incurred on the $i$-th data sample using a hypothesis parameterized by the decision variable $x$. Hence, in practice, $N$ is an extremely large number as it corresponds to the size of the data set. The recent literature on variance reduction techniques and distributed optimization is very active in making such large scale optimization problems tractable. Surveys on the latest developments in these fields can be found in \cite{GowSchBacRic20} and the comprehensive textbook by Lan \cite{lan2020first}.

\section{The Proximal Gradient Method}
\label{sec:MD}
%
\subsection{Motivation}
We are starting our survey on first-order methods for solving convex optimization problems with perhaps the most basic optimization method known to every student who took a course in mathematical programming: the gradient projection scheme. In the context of the composite optimization problem \eqref{eq:Opt}, a classical and very powerful idea is to construct numerical optimization methods by exploiting problem structure. Following this philosophy, we determine the position of the next iterate by minimizing the sum of the linearization of the smooth part, the non-smooth part $r\in\Gamma_{0}(\setV)$, and a quadratic regularization term with weight $\gamma >0$: 
\begin{equation}\label{eq:PG1}
x^{+}(\gamma)=\argmin_{u\in\setX}\{f(x)+\inner{\nabla f(x),u-x} +r(u)+\frac{1}{\gamma}\norm{u-x}^{2}_{2}\}.
\end{equation}
Disregarding terms which do not influence the computation of the solution of this strongly convex minimization problem, and absorbing the set constraint into the non-smooth part by defining $\phi(x)=r(x)+\delta_{\setX}(x)$, we see that \eqref{eq:PG1} can be equivalently written as 
\begin{equation}\label{eq:PG}
x^{+}(\gamma)=\argmin_{u\in\setV}\left\{\gamma \phi(u)+\frac{1}{2}\norm{u-(x-\gamma\nabla f(x))}^{2}_{2}\right\}.
\end{equation}
The trained reader will immediately see some geometric principles involved in this minimization routine; Indeed, if $r$ would be constant on $\setX$ (say $0$ for concreteness), then the rule \eqref{eq:PG} is nothing else than the Euclidean projection of the directional vector $x-\gamma\nabla f(x)$ onto the set $\setX$. In this case, the minimization routine returns the classical projected gradient step $x^{+}(\gamma)=P_{\setX}(x-\gamma\nabla f(x)).$ Iterating the map $T_{\gamma}:=P_{\setX}\circ(\Id-\gamma\nabla f)$ generates the Gradient projection method, which can be traced back to the 1960s (see \cite{Ber99} for the history of this method). A new obstacle arises in cases where the non-smooth function $r$ is non-trivial over the relevant domain $\setX$. A fundamental idea, going back to Moreau \cite{Mor65}, is to define the \emph{proximity operator} $\prox_{\phi}:\setV\to\setV$ associated with a function $\phi\in\Gamma_{0}(\setV)$ as\footnote{The repository \url{http://proximity-operator.net/index.html} provides codes and explicit expressions for proximity operators of many standard functions. A useful MATLAB implementation of proximal methods is described in \cite{BeckFOM}.} 
\begin{equation}\label{eq:Mprox}
\prox_{\phi}(x):=\argmin_{u\in\setV}\left\{\phi(u)+\frac{1}{2}\norm{u-x}^{2}_{2}\right\}.
\end{equation} 
In terms of the proximity-operator, the minimization step \eqref{eq:PG} becomes
\begin{equation}
x^{+}(\gamma)=T_{\gamma}(x):=\prox_{\gamma \phi}(x-\gamma \nabla f(x)).
\end{equation}
Iterating the map $T_{\gamma}$ yields a new and more general method, known in the literature as the proximal gradient method (PGM).

 \begin{algorithm}{\bf{The Proximal Gradient Method} (PGM)}
\\
	{\bf Input:} pick $x^{0}\in\setX.$\\
	{\bf General step:} For $k=0,1,\ldots$ do:\\
	\qquad pick $\gamma_{k}>0$.\\
	\qquad set $x^{k+1}=\prox_{\gamma_{k}\phi}\left(x^{k}-\gamma_{k} \nabla f(x^{k})\right)$.
\end{algorithm}

PGM is a very powerful method which received enormous interest in optimization and its applications. For a survey in the context of signal processing we refer the reader to \cite{ComPes11}. A general survey on proximal operators has been given by Parikh and Boyd \cite{ParBoy14} and Beck \cite{Beck17}, and many more references can be found in these references.

The special case when $f=0$ is known as the \emph{proximal point method}, which reads explicitly as 
\begin{equation}
x^{k+1}=\prox_{\gamma_{k} \phi}(x^{k})=\argmin_{u\in\setV}\{\phi(u)+\frac{1}{2\gamma_{k}}\norm{u-x^{k}}^{2}\}.
\end{equation}
The value function 
$$
\phi_{\gamma}(x)=\inf_{u}\{\phi(u)+\frac{1}{2\gamma}\norm{u-x}^{2}\}
$$
is called the \emph{Moreau envelope} of the function $\phi$, and is an important smoothing and regularization tool, frequently employed in numerical analysis. Indeed, for a function $\phi\in\Gamma_{0}(\setV)$, its Moreau envelope is finite everywhere, convex and has $\gamma^{-1}$-Lipschitz continuous gradient on $\setV$ given by 
$\nabla \phi_{\gamma}(x)=\frac{1}{\gamma}(x-\prox_{\gamma\phi}(x)).$


\subsection{Bregman Proximal Setup}
\label{S:Bregman_setup}
The basic idea behind non-Euclidean extensions of PGM is to replace the $\ell_{2}$-norm $\frac{1}{2}\norm{u-x}^{2}$ by a different distance-like function which is tailored to the geometry of the feasible set $\setX\subseteq\setV$. These non-Euclidean distance-like functions that will be used are \emph{Bregman divergences}. The transition from Euclidean to non-Euclidean distance measures is motivated by the usefulness and flexibility of the latter in computational perspectives and potentials for improving convergence properties for specific application domains. In particular, the move from Euclidean to non-Euclidean distance measures allows to adapt the algorithm to the underlying geometry, typically explicitly embodied in the set constraint $\setX$, see e.g. \cite{AusTeb09}. This can not only positively affect the per-iteration complexity, but also will have a footprint on the overall iteration complexity of the method, as we will demonstrate in this section.


In the rest of this section, we assume that the set constraint $\setX$ is a closed convex set with nonempty relative interior $\rint(\setX)$. 
The point of departure of Bregman Proximal algorithms is to introduce a distance generating function $h:\setV\to(-\infty,\infty]$, which is a barrier-type of mapping suitably chosen to capture geometric features of the set $\setX$. 
\begin{definition}
\label{Def:DGF}
Let $\setX$ be a compact convex subset of $\setV$. We say that $h\in\Gamma_{0}(\setV)$ is a \emph{\ac{DGF}} with modulus $\alpha>0$ with respect to $\norm{\cdot}$ on $\setX$ if 
\begin{enumerate}
\item Either $\dom h=\rint(\setX)$ or $\dom h=\setX$; 
\item $h$ is differentiable over $\setX^{\circ}=\{x\in\setX\vert \partial h(x)\neq\emptyset\}$.  
\item $h$ is $\alpha$-strongly convex on $\setX$ relative to $\norm{\cdot}$
\begin{equation}
h(x')\geq h(x)+\inner{\nabla h(x),x'-x}+\frac{\alpha}{2}\norm{x'-x}^{2}
\end{equation}
for all $x\in\setX^{\circ}$ and all $x'\in\setX$.
\end{enumerate}
We denote by $\scrH_{\alpha}(\setX)$ the set of DGFs on $\setX$.
\end{definition}
Note that $\setX^{\circ}$ contains the relative interior of $\setX$ and, restricted to $\setX^{\circ}$, $h$ is continuously differentiable with $\partial h(x)=\{\nabla h(x)\}$. In many proximal settings we are interested in \ac{DGF}s which act as barriers on the feasible set $\setX$. Such DGFs are included in the case $\dom h=\rint(\setX)$. Naturally, the barrier properties of the function $h$ are captured by its scaling near $\bd(\setX)$, usually encoded in terms of the notion of \emph{essential smoothness} \cite{Roc70}.

\begin{definition}[Essential smoothness]
\label{def:Legendre}
$h\in\scrH_{\alpha}(\setX)$ is \emph{essentially smooth} if for all sequences $(x_{j})_{j\in\N}\subseteq\rint(\setX)$ with $\lim_{j\to\infty}\dist(\{x_{j}\},\bd(\setX))=0$, we have $\lim_{j\to\infty}\norm{\nabla h(x_{j})}_{\ast}=\infty$.
\end{definition}

Given a \ac{DGF} $h\in\scrH_{\alpha}(\setX)$, we define the \emph{Bregman divergence} $D_{h}:\dom h \times \setX^{\circ} \to\R$ induced by $h$ as 
\begin{equation}\label{eq:Breg}
D_{h}(u,x)=h(u)-h(x)-\inner{\nabla h(x),u-x}.
\end{equation}
Since $h\in\scrH_{\alpha}(\setX)$, it follows immediately that  
\begin{equation}\label{eq:Dlower}
D_{h}(u,x)\geq\frac{\alpha}{2}\norm{u-x}^{2}\qquad \forall x\in\setX^{\circ},u\in\dom h.
\end{equation}
Hence, Bregman divergences are zero on the main diagonal of $\setX^{\circ}\times\setX^{\circ}$, but in general they are not symmetric and they do not satisfy a triangle inequality. This disqualifies them to carry the label of a metric, but still they can be interpreted as distance measures on $\setX^{\circ}$.

The convex conjugate $h^{\ast}(y)=\sup_{x\in\setV}\{\inner{x,y}-h(x)\}$ for a function $h\in\scrH_{\alpha}(\setX)$ is known to be differentiable on $\setV^{\ast}$ and $\frac{1}{\alpha}$-Lipschitz smooth, i.e. 
  \begin{equation}
 h^{\ast}(y_{2})\leq h^{\ast}(y_{1})+\inner{\nabla h^{\ast}(y_{1}),y_{2}-y_{1}}+\frac{1}{2\alpha}\norm{y_{2}-y_{1}}^{2}_{\ast}.
 \end{equation}
 for all $y_{1},y_{2}\in\setV^{\ast}$. In fact, Section 12H in \cite{RW98} gives us the following general result which is of fundamental importance for the following approaches. 
 \begin{proposition}
 \label{prop:Rock}
 Let $\omega:\Rn\to\R\cup\{+\infty\}$ be a proper convex and lower semi-continuous function. Consider the following statements:
 \begin{enumerate}
 \item[(a)] $\omega$ is strongly convex with parameter $\alpha>0$
 \item[(b)] The subdifferential mapping $\partial\omega:\Rn\to 2^{\Rn}$ is strongly monotone with parameter $\alpha>0$:
 \begin{equation}
 \inner{u-v,x-y}\geq\alpha\norm{x-y}^{2}\qquad\forall (x,u),(v,y)\in\gr(\partial\omega)
 \end{equation}
 \item[(c)] The inverse map $(\partial\omega)^{-1}$ is single-valued and Lipschitz continuous with modulus $\frac{1}{\alpha}$;
 \item[(d)] $\omega^{\ast}$ is finite and differentiable everywhere.
 \end{enumerate}
 Then $(a)\iff (b)\Rightarrow (c)\iff (d)$.
 \end{proposition}

Once we endow our set $\setX$ with a Bregman divergence, the technology generating a gradient method in this non-Euclidean setting is the \emph{prox-mapping}.

 \begin{definition}[Prox-Mapping]
 Given $h\in\scrH_{\alpha}(\setX)$ and $\phi\in\Gamma_{0}(\setV)$, define the \emph{prox-mapping} as

 \begin{equation}\label{eq:prox}
 \scrP^{h}_{\phi}(x,y):=\argmin_{u\in\setX}\{\phi(u)+\inner{y,u-x}+D_{h}(u,x)\}.
\end{equation}
\end{definition}
The prox-mapping takes as inputs a "primal-dual" pair $(x,y)\in \setX^{\circ}\times\setV^{\ast}$ where $x$ is the current iterate, and $y$ is a dual variable representing the signal we obtain on the smooth part of the minimization problem \eqref{eq:Opt}. Various conditions on the well-posedness of the prox-mapping have been stated in the literature. We will not repeat them here, but rather refer to the recent survey \cite{Teb18}. 

It will be instructive to go over some standard examples of the Bregman proximal setup. See also \cite{ComWaj05}, \cite{BauCom16}, and \cite{ben-tal2020lectures}. 
 
 \begin{example}[Proximity Operator]
 We begin by revisiting the Euclidean projection on some convex closed subset $\setX$ of the vector space $\setV$. Letting $h(x)=\frac{1}{2}\norm{x}^{2}_{2}+\delta_{\setX}(x)$ for $x\in\setV$, we readily see that $\setX^{\circ}=\setX=\dom(h)$. Moreover, for $x\in\setX$, the vector field $\nabla h(x)=x$ is a continuous selection of $\partial h(x)$ for all $x\in\setX$. Hence, the associated Bregman divergence is $D_{h}(u,x)=\frac{1}{2}\norm{u-x}^{2}_{2}$ for all $u,x\in\setX$. Given a function $\phi\in\Gamma_{0}(\setV)$, the resulting prox-mapping reads as 
 $$
 \scrP^{h}_{\phi}(x,y)=\argmin_{u\in\setX}\{\phi(u)+\inner{y,u-x}+\frac{1}{2}\norm{u-x}^{2}_{2}\}=\prox_{\phi+\delta_{\setX}}(x-y)
 $$
where $\prox_{\phi}$ is the proximity operator defined in \eqref{eq:Mprox}. 
 \end{example}
 
 \begin{example}[Entropic Regularization]
 Let $\setX=\{x\in\Rn_{+}\vert \sum_{i=1}^{n}x_{i}=1\}$
 denote the unit simplex in $\setV=\Rn$. Define the function $\psi:\R\to[0,\infty]$ as 
 \[
 \psi(t):=\left\{\begin{array}{ll} 
 t\ln(t)-t & \text{if }t>0,\\
 0 &\text{if } t=0,\\
 +\infty & \text{else}
 \end{array}
 \right.
 \]
As \ac{DGF} consider the \emph{Boltzmann-Shannon entropy} $h(x):=\sum_{i=1}^{n}\psi(x_{i})+\delta_{\{x\in\setV\vert\sum_{i=1}^{n}x_{i}=1\}}$. Endowing the ground space $\setV$ with the $\ell_{1}$ norm, it can be shown that $h\in\scrH_{1}(\setX)$ with $\dom h=\setX$ and $\setX^{\circ}=\{x\in\R^{n}_{++}\vert\sum_{i=1}^{n}x_{i}=1\}=\rint(\setX)$. The resulting Bregman divergence is the \emph{Kullback-Leibler divergence}
 \[
 D_{h}(u,x)=\sum_{i=1}^{n}u_{i}\ln\left(\frac{u_{i}}{x_{i}}\right)+\sum_{i=1}^{n}(x_{i}-u_{i}).
 \]
For $\phi=0$ a standard calculation gives rise to the prox-mapping 
\begin{equation}
[\scrP^{h}_{\delta_{\setX}}(x,y)]_{i}=\frac{x_{i}e^{y_{i}}}{\sum_{j=1}^{n}x_{j}e^{y_{j}}}\qquad 1\leq i\leq n,x\in\setX,y\in\setV^{\ast}.
\end{equation}
This mapping plays a key role in optimization, where it is known as exponentiated gradient descent \cite{BecTeb03,JudNazTsyVay05}. 
 \end{example}
 
 \begin{example}[Box Constraints]
Assume that $\setV=\Rn$ and $\setX=\prod_{i=1}^{n}[a_{i},b_{i}]$ where $0\leq a_{i}\leq b_{i}$. Given parameters $0\leq a\leq b$, define the \emph{Fermi-Dirac entropy} 
 \[
 \psi_{a,b}(t):=\left\{\begin{array}{ll} 
 (t-a)\ln(t-a)+(b-t)\ln(b-t) & \text{if }t\in (a,b),\\
 0 &\text{if } t\in\{a,b\},\\
 +\infty & \text{else}
 \end{array}
 \right.
 \]
 Then $h(x)=\sum_{i=1}^{n}\psi_{a_{i},b_{i}}(x_{i})$ is a \ac{DGF} on $\setX=\dom h$ with $\setX^{\circ}=\prod_{i=1}^{n}(a_{i},b_{i})$. 
 \end{example}
 
\begin{example}[Semidefinite Constraints]
Let $\setV$ be the set of real symmetric matrices and $\setX=\setS^{n}_{+}$ be the cone of real symmetric positive semi-definite matrices equipped with the inner product $\inner{\bA,\BB}=\tr(\bA\BB)$. Define $h(\XX)=\tr[\XX\log(\XX)]$ as the matrix-equivalent of the negative Boltzmann-Shannon entropy. It can be verified that $\dom h=\setX$ and $\nabla h(\XX)=\log(\XX)+\bI$. Hence, one sees that $\dom h=\setX$, and $\setX^{\circ}=\setS^{n}_{++}$, the cone of positive definite matrices. For $\XX'\in\setS^{n}_{++}$, the corresponding Bregman divergence is given by
\[
D_{h}(\XX',\XX)=\tr[\XX\log(\XX)-\XX\log(\XX')+\XX'-\XX]
\]
See \cite{DolTeb98} for further examples on matrix domains.  
 \end{example}
 
 \begin{example}[Spectrahedron]
 \label{ex:spectrahedron}
 Let $\setX=\{\XX\in\setS^{n}_{+}\vert \tr(\XX)\leq 1\}$ the unit spectrahedron of positive semi-definite matrices with the nuclear norm $\norm{\XX}_{1}=\tr(\XX)$.
 For this geometry, a widely used regularizer is the von Neumann entropy 
 \begin{equation}
 h(\XX)=\tr(\XX\log\XX)+(1-\tr(\XX))\log(1-\tr(\XX))
 \end{equation}
 It can be shown that this function is $\frac{1}{2}$-strongly convex with respect to the nuclear norm and $\dom h=\setX$, as well as $\setX^{\circ}=\{\XX\in\setS^{n}_{++}\vert \tr(\XX)< 1\}$.
 \end{example}
 
\begin{example}[2nd order cone constraints]
Let $\setV=\Rn$ and $L^{n}_{++}:=\{x\in\setV\vert x_{n}>(x_{1}^{2}+\ldots+x_{n-1}^{2})^{1/2}\}$ the interior of the second-order cone with closure denoted by $\setX$. Let $\bJ_{n}$ be the $n\times n$ diagonal matrix with $-1$ in its first $n-1$ diagonal entries and $1$ in the last one. Define $h(x)=-\ln(\inner{\bJ_{n}x,x})+\frac{\alpha}{2}\norm{x}^{2}_{2}$. Then $h\in\scrH_{\alpha}(\setX)$ with $\dom h=\setX^{\circ}=L^{n}_{++}\subset\setX$. The associated Bregman divergence is 
$$
D_{h}(x,u)=-\ln\left(\frac{\inner{\bJ_{n}x,x}}{\inner{\bJ_{n}u,u}}\right)+2\frac{\inner{\bJ_{n}x,u}}{\inner{\bJ_{n}u,u}}-2+\frac{\alpha}{2}\norm{x-u}^{2}_{2}.
$$
The proximal framework for general conic constraints has been developed in \cite{AusTeb06}. 
\end{example}

 If $\gamma>0$ is a step-size parameter and $y=\gamma\nabla f(x)$, then we obtain the \emph{Bregman proximal map} $T^{h}_{\gamma}(x)=\scrP^{h}_{\gamma r}(x,\gamma\nabla f(x))$ for all $x\in\setX$. Iterating this map generates a discrete-time dynamical system known as the \emph{Bregman proximal gradient method} (BPGM). 
 \begin{algorithm}{\bf{The Bregman Proximal Gradient Method} (BPGM)}
\\
	{\bf Input:} $h\in\scrH_{\alpha}(\setX)$. Pick $x^{0}\in\dom(r)\cap\setX^{\circ}.$\\
	{\bf General step:} For $k=0,1,\ldots$ do:\\
	\qquad pick $\gamma_{k}>0$.\\
	\qquad set $x^{k+1}=\scrP^{h}_{\gamma_{k} r}(x^{k},\gamma_{k}\nabla f(x^{k}))$.
\end{algorithm}

The BPGM approach consists of linearizing the differentiable part $f$ around $x$, adding the composite term $r$, and regularizing the sum with a proximal distance from the point $x$. When $h$ is the squared Euclidean norm, BPGM reduces to the classical proximal gradient method. For simple implementation, BPGM relies on the structural assumption that the prox-mapping $\scrP^{h}_{r}(x,y)$ can be evaluated efficiently on the trajectory $\{(x^{k},\gamma_{k}\nabla f(x^{k}))\vert 0\leq k\leq K\in\N^{\ast}\}$. This, often somewhat hidden, assumption is known in the literature as the "prox-friendliness" assumption, a terminology apparently coined by \cite{CoxJudNem14}).

\subsection{Basic Complexity Properties}
To analyze the iteration complexity of BPGM, let us define the convex lower semi-continuous and proper function 
\begin{equation}\label{eq:phi}
\varphi(u)=r(u)+\inner{y,u-x}+\delta_{\setX}(x),
\end{equation}
where $y\in\setV^{\ast}$ and $x \in \setX$ are treated as parameters. Under this terminology, we readily see that the basic iterate of BPGM is determined by the evaluation of the \emph{Bregman proximal operator} \cite{Teb92} applied to the function $\varphi\in\Gamma_{0}(\setV)$:
\begin{align*}
\prox_{\varphi}^{h}(x):=\argmin_{u \in \setV}\{\varphi(u)+D_{h}(u,x)\}.
\end{align*}
Writing the first-order optimality condition satisfied by the point $x^{+}=\scrP^{h}_{r}(x,y)$ in terms of the function $\varphi$ in \eqref{eq:phi}, we get
\begin{align*}
0\in\partial \varphi(x^{+})+\nabla h(x^{+})- \nabla h(x).
\end{align*}
Whence, there exists $\xi\in\partial \varphi(x^{+})$ such that, for all $u \in \setX$, 
\begin{align}
\label{eq:sub_prob_opt_cond}
\inner{\xi+\nabla h(x^{+})-\nabla h(x),x^{+}-u}\leq 0.
\end{align}
Via the subgradient inequality for the convex function $u\mapsto\varphi(u)$, we obtain for all $u \in \setX$:
\begin{equation}
\varphi(x^{+})-\varphi(u)\leq \inner{\xi,x^{+}-u} \leq \inner{\nabla h(x)-\nabla h(x^{+}),x^{+}-u}.
\end{equation}

For further analysis, we need the celebrated three-point identity, due to \cite{CT93}.
\begin{lemma}[3-point lemma]
\label{Lm:3-point}
For all $x,y\in\setX^{\circ}$ and $z\in\dom h$ we have
\begin{align*}
D_{h}(z,x)-D_{h}(z,y)-D_{h}(y,x)=\inner{\nabla h(x)-\nabla h(y),y-z}.
\end{align*}
\qed
\end{lemma}

This yields immediately, 
\begin{align*}
\varphi(x^{+})-\varphi(u)\leq D_{h}(u,x)-D_{h}(u,x^{+})-D_{h}(x^{+},x).
\end{align*}

Performing the formal substitution $r\leftarrow \gamma r$ and $y\leftarrow \gamma \nabla f(x)$ in the definition of the function $\varphi$ in \eqref{eq:phi}, this delivers the inequality 
 \begin{equation}\label{eq:r}
 \gamma(r(x^{+})-r(u))\leq \gamma\inner{\nabla f(x),u-x^{+}}+D_{h}(u,x)-D_{h}(u,x^{+})-D_{h}(x^{+},x).
 \end{equation}
 
Note that if $x^{+}$ is calculated inexactly in the sense that instead of \eqref{eq:sub_prob_opt_cond} it holds that
\begin{align}
\label{eq:sub_prob_opt_cond_inexact}
\inner{\xi+\nabla h(x^{+})-\nabla h(x),x^{+}-u}\leq \Delta
\end{align}
for some $\Delta \geq 0$, then instead of \eqref{eq:r} we have
 \begin{equation}\label{eq:r_inexact}
 \gamma(r(x^{+})-r(u))\leq \gamma\inner{\nabla f(x),u-x^{+}}+D_{h}(u,x)-D_{h}(u,x^{+})-D_{h}(x^{+},x) + \Delta.
\end{equation}
See \cite{AusTeb06} for an explicit analysis of the error-prone implementation.

Since $f$ is assumed to possess a Lipschitz continuous gradient on $\setV$, the classical "descent Lemma" \cite{Nes18} tells us that
 \begin{equation}\label{eq:DLclassical}
 f(x^{+})\leq f(x)+\inner{\nabla f(x),x^{+}-x}+\frac{L_{f}}{2}\norm{x^{+}-x}^{2}.
 \end{equation}
 Additionally, for all $u\in\dom h$, convexity of $f$ on $\setV$ implies
 \[
 f(u)\geq f(x)+\inner{\nabla f(x),u-x}.
 \]
 Therefore, combining this with \eqref{eq:DLclassical} and using \eqref{eq:Dlower}, we obtain for any $u\in\dom h$,
$$
 f(x^{+})-f(u)\leq\inner{\nabla f(x),x^{+}-u}+\frac{L_{f}}{2}\norm{x^{+}-x}^{2}\leq \inner{\nabla f(x),x^{+}-u}+\frac{L_{f}}{\alpha}D_{h}(x^{+},x).
$$
 Multiplying this by $\gamma$ and adding the result to \eqref{eq:r}, we obtain, for any $u\in\dom h$, 
\begin{equation}\label{eq:DesIneq}
 \gamma(\Psi(x^{+})-\Psi(u))\leq D_{h}(u,x)-D_{h}(u,x^{+})-\left(1-\frac{\gamma L_{f}}{\alpha}\right)D_{h}(x^{+},x).
 \end{equation}
 If $\gamma\in(0,\frac{\alpha}{L_f}]$, then the above yields 
 \begin{align*}
 \gamma(\Psi(x^{+})-\Psi(u))\leq D_{h}(u,x)-D_{h}(u,x^{+}),\quad u\in\dom h.
 \end{align*}
Setting $x=x^{k},x^{+}=x^{k+1},\gamma=\gamma_{k}$ and $u\in\dom h$, one can reformulate the previous display as 
 \begin{align*}
 \gamma_{k}\left(\Psi(x^{k+1})-\Psi(u)\right)\leq D_{h}(u,x^{k})-D_{h}(u,x^{k+1}),\quad u\in\dom h.
\end{align*}
If $u=x^{k}$, we readily see $\gamma_{k}(\Psi(x^{k+1})-\Psi(x^{k}))\leq -D_{h}(x^{k},x^{k+1})\leq 0$, i.e. the sequence of function values $\{\Psi(x^{k})\}_{k\in\N}$ is non-increasing. On the other hand, for a general reference point $u\in\dom h$, we also see that 
\begin{align*}
\sum_{k=0}^{N-1}\left(\Psi(x^{k+1})-\Psi(u)\right)&\leq \sum_{k=0}^{N-1}\frac{1}{\gamma_{k}}\left[D_{h}(u,x^{k})-D_{h}(u,x^{k+1})\right]\\
&=\frac{1}{\gamma_{0}}D_{h}(u,x^{0})-\frac{1}{\gamma_{N-1}}D_{h}(u,x^{N})+\sum_{k=0}^{N-2}\left(\frac{1}{\gamma_{k+1}}-\frac{1}{\gamma_{k}}\right)D_{h}(u,x^{k+1}).
\end{align*}
Assuming a constant step size policy $\gamma_{k}=\gamma$, this gives us 
\begin{align*}
\sum_{k=0}^{N-1}\left(\Psi(x^{k+1})-\Psi(u)\right)&\leq \frac{1}{\gamma}D_{h}(u,x^{0}).
\end{align*}
Define the function gap $s^{k}:=\Psi(x^{k})-\Psi(u)$, then $s^{k+1}-s^{k}=\Psi(x^{k+1})-\Psi(x^{k})\leq 0$, and therefore 
\begin{align*}
s^{N}&\leq\frac{1}{N}\sum_{k=0}^{N-1}s^{k+1}= \frac{1}{N}\sum_{k=0}^{N-1}[\Psi(x^{k+1})-\Psi(u)]\leq\frac{1}{N\gamma}D_{h}(u,x^{0})
\end{align*}
 for all $u\in\dom h$. As an attractive step size choice, we may take the greedy choice $\gamma=\frac{\alpha}{L_{f}}$. However, we need to know the Lipschitz constant of the gradient map of the smooth part $f$ of the minimization problem \eqref{eq:Opt} to make this an implementable solution strategy. Assuming that $\dom h$ is closed we get immediately from the estimate above the basic complexity result on the BPGM.
 \begin{proposition}
 If BPGM is run with the constant step size $\gamma_{k}=\frac{\alpha}{L_{f}}$ and $\dom h=\cl(\dom h)=\setX$, then for any $x^{\ast}\in\setX^{\ast}$, we have 
 \begin{equation}\label{eq:LCBPGM}
 \Psi(x^{k})-\Psi_{\min}(\setX)\leq\frac{L_{f}}{\alpha k}D_{h}(x^{\ast},x^{0}).
 \end{equation}
 \end{proposition}
 This global sublinear rate of convergence for the Euclidean setting has been established in \cite{BecTebGradient09,Nes13}. Under additional assumption that the objective $\Psi$ is $\mu$-relatively strongly convex \cite{LuFreNes18} it is possible to obtain linear convergence rate of BPGM, i.e. $\Psi(x^{k})-\Psi_{\min}(\setX)\leq 2L_f\exp(-k\mu/L_f)D_{h}(x^{\ast},x^{0})$, see e.g. \cite{LuFreNes18,stonyakin2019gradient,stonyakin2020inexact}, where the authors of the latter two papers also analyze this kind of methods under inexact oracle and inexact Bregman proximal step.
 
 \subsubsection{Subgradient and Mirror Descent}
 
 In the previous subsections we focused on the setting of problem \eqref{eq:Opt} with smooth part $f$ and obtained for BPGM a convergence rate $O(1/k)$. The same method actually works for non-smooth convex optimization problems when $f$ has bounded subgradients. In this setting BPGM with a different choice of the step-size $\gamma$ is known as the \emph{Mirror Descent} (MD) method \cite{NY83}. A version of this method for convex composite non-smooth optimization was proposed in \cite{duchi2010composite}, and an overview of Subgradient/Mirror Descent type of methods for non-smooth problems can be found in \cite{Beck17,dvurechensky2020advances,lan2020first}. 
 
 The main difference between BPGM and MD is that one replaces the assumption that $\nabla f$ is Lipschitz continuous with the assumption that $f$ is subdifferentiable with bounded subgradients, i.e. $\| f'(x)\|_{\ast} \leq M_f$ for all $x \in \setX$ and $f'(x)\in\partial f(x)$. For a given sequence of step-sizes $(\gamma_{k})_{k}$ one defines the next test point as 
 \[
 x^{k+1}=\argmin_{u}\left\{\inner{\gamma_{k} f'(x^{k}),u-x^{k}}+\gamma_{k}r(u)+D_{h}(u,x^{k})\right\}=\scrP^{h}_{\gamma_{k}r}(x^{k},\gamma_{k}f'(x^{k})).
 \]
 A typical choice for the step size sequence is a monotonically decreasing policy like $\gamma_{k}\sim k^{-1/2}$. Under such a specification, the MD sequence $(x^{k})_{k}$ can be shown to converge with rate $O(1/\sqrt{k})$ to the solution, which is optimal in this setting. A proof of this result can be patterned via a suitable adaption of the arguments employed in our analysis of the Dual Averaging Method in Section \ref{sec:DA}. 

 \subsubsection{Potential Improvements due to relative smoothness}
\label{sec:NoLips} 
A key pillar of the complexity analysis of BPGM was the descent lemma \eqref{eq:DLclassical}, which in turn is a consequence of the assumed Lipschitz continuity of the gradient $\nabla f$. The very influential recent work by \cite{BauBolTeb16} introduced a very clever construction which allows one to relax this restrictive assumption.\footnote{Variations on the same theme can be found in \cite{LuFreNes18}.} The elegant observation made in \cite{BauBolTeb16} is that the Lipschitz-gradient-based descent lemma has the equivalent, but insightful, expression 
\begin{align*}
\left(\frac{L_{f}}{2}\norm{x}^{2}-f(x)\right)-\left(\frac{L_{f}}{2}\norm{u}^{2}-f(u)\right)\geq \inner{L_{f}u-\nabla f(u),x-u}\qquad\forall x,u\in\setV.
\end{align*}
This is just the gradient inequality for the convex function $x\mapsto \frac{L_{f}}{2}\norm{x}^{2}-f(x)$. Based on the general intuition we have gained while working with a general proximal setup, a very tempting and natural generalization is the following.
\begin{definition}[Relative Smoothness, \cite{BauBolTeb16}]
\label{def:RelSmo}
The function $f$ is smooth relative to the essentially smooth DGF $h\in\scrH_{0}(\setX)$ with $\setX=\cl(\dom h)$, if for any $x,u\in\setX^{\circ}$, there is a scalar $L_f^h\geq 0$ for which 
\begin{equation}\label{eq:DLNoLips}
f(u)\leq f(x)+\inner{\nabla f(x),u-x}+L_f^h D_{h}(u,x).
\end{equation}
\end{definition}
Structurally, relative smoothness implies a descent lemma where the squared Euclidean norm is replaced with a general Bregman divergence induced by an essentially smooth function $h\in\scrH_{0}(\setX)$. Rearranging terms, a very concise and elegant way of writing relative smoothness is that $D_{L_f^h h-f}(u,x)\geq 0$ on $\setX^{\circ}$, or that $L_f^h h- f$ is convex on $\setX^{\circ}$ if the latter is a convex set. Clearly, if $f$ and $h$ are twice continuously differentiable on $\setX^{\circ}$, the relative smoothness condition can be stated in terms of a positive semi-definitness condition on the set $\setX^{\circ}$ as 
\begin{equation}
L_f^h \nabla^{2}h(x)-\nabla^{2}f(x)\succeq 0\qquad\forall x\in\setX^{\circ}.
\end{equation}
Beside providing a non-Euclidean version of the descent lemma, the notion of relative smoothness allows us to rigorously apply gradient methods to problems whose smooth part admits no global Lipschitz continuous gradient. This gains relevance in solving various classes of inverse problems (see Section 5.2 in \cite{BauBolTeb16}), and optimal experimental design \cite{LuFreNes18}, a class of problems structurally equivalent to finding the minimum volume ellipsoid containing a list of vectors \cite{BoyVan04,Tod16}.\\

The complexity analysis of BPGM under a relative smoothness assumption on the pair $(f,h)$ proceeds analogous to the previous analysis. This NoLips algorithm, using the terminology coined by \cite{BauBolTeb16}, however involves a different condition number than the ratio $\frac{\alpha}{L_{f}}$ as in \eqref{eq:LCBPGM}. This is an important fact which makes this method potentially interesting even if the problem at hand admits a Lipschitz continuous gradient. The first important result is an extended version of the fundamental inequality \eqref{eq:DesIneq}, which reads as 
\begin{equation}
\gamma(\Psi(x^{+})-\Psi(u))\leq D_{h}(u,x)-D_{h}(u,x^{+})-(1-\gamma L_f^h)D_{h}(x^{+},x)\quad\forall u\in\dom h.
\end{equation}
The derivation of this inequality is analogous to inequality \eqref{eq:DesIneq}, replacing the Lipschitz-gradient-based descent inequality \eqref{eq:DLclassical} by the relative smoothness inequality \eqref{eq:DLNoLips} with parameter $L_f^h$. The continuation of the proof differs then in an important aspect. It relies on the introduction of the \emph{symmetry coefficient} of the \ac{DGF} $h$ as 
\begin{equation}
\nu(h):=\inf\left\{\frac{D_{h}(x,u)}{D_{h}(u,x)}\vert (x,u)\in\setX^{\circ}\times\setX^{\circ},x\neq u\right\}.
\end{equation}
The symmetry coefficient $\nu(h)$ is confined to the interval $[0,1]$, and $\nu(h)=1$ applies essentially only to the energy function $h(x)=\frac{1}{2}\norm{x}^{2}$. Choosing $\gamma=\frac{1+\nu}{2L_f^h},x^{+}=x^{k+1},x=x^{k}$ gives
\[
\gamma(\Psi(x^{k+1})-\Psi(u))\leq D_{h}(u,x^{k})-D_{h}(u,x^{k+1})-\frac{1-\nu}{2}D_{h}(x^{k+1},x^{k})\\
\]
Setting $u=x^{k}$ gives descent of the function value sequence $(\Psi(x^{k}))_{k\geq 0}$. Moreover, it immediately follows that 
\[
\Psi(x^{k})-\Psi(u)\leq\frac{2L_f^h}{1+\nu}\left(D_{h}(u,x^{k-1})-D_{h}(u,x^{k})\right)
\]
Summing from $k=1,2,\ldots,N$, the same argument as for the BPGM give sublinear convergence of NoLips 
\begin{equation}
\Psi(x^{N})-\Psi(u)\leq\frac{2L_f^h}{N(1+\nu)}D_{h}(u,x^{0}).
 \end{equation}
Comparing the constants in the complexity estimates of NoLips and BPGM we see that the relative efficiency of the two methods depends on the condition number ratio $\frac{2L_f^h/(1+\nu)}{L_{f}/\alpha}$. Hence, even if the objective function is globally Lipschitz smooth (i.e. admits a Lipschitz continuous gradient), exploiting the idea of relative smoothness might lead to superior performance of NoLips. 

To establish global convergence of the trajectory $(x^{k})_{k\in\N}$, additional "reciprocity" conditions on the Bregman divergence must be imposed. 
\begin{assumption}\label{ass:H}
The essentially smooth function $h\in\scrH_{0}(\setX)$ satisfies the Bregman reciprocity condition if the level sets $\{u\in\setX^{\circ}\vert D_{h}(u,x)\leq\beta\}$ are bounded for all $\beta\in\R$, and 
\begin{align*}
x^{k}\to x\in\setX^{\circ}\iff \lim_{k\to\infty}D_{h}(x,x^{k})=0. 
\end{align*}
\end{assumption}
This assumption is necessary, as in some settings Bregman reciprocity is violated. See Example 4.1 in \cite{DolTeb98} as a simple illustration. Under Bregman reciprocity, one can prove global convergence in the spirit of Opial's lemma \cite{Opi67}: 

\begin{theorem}[\cite{BauBolTeb16}, Theorem 2]
Let $(x^{k})_{k\in\N}$ be the sequence generated by BPGM with $\gamma\in(0,\frac{1+\nu(h)}{L_f^h})$. Assume $\setX =\cl(\dom h)=\dom h$ and that Assumption \ref{ass:H} holds additional to the standing hypothesis of this survey. Then, the sequence $(x^{k})_{k\in\N}$ converges to some solution $x^{\ast}\in\setX^{\ast}$.
\end{theorem}

 \subsection{Dual Averaging}
 \label{sec:DA}
An alternative method called Dual Averaging (DA) was proposed in \cite{Nes09} and, on the contrary, is a primal-dual method making alternating updates in the space of gradients and in the space of iterates. The extension to the convex non-smooth composite problem \eqref{eq:Opt} is due to \cite{Lin10}. Below we give a self-contained complexity analysis of this scheme for non-smooth optimization, i.e. under an assumption that $\norm{\nabla f(x)}_{\ast} \leq L_f$ for all $x\in \setX$ instead of the $L_f$-smoothness assumption in the previous subsections. 
 
We start the description and analysis of the Dual Averaging method with some preliminaries and assumptions. First, we change in this section the Lipschitz-smoothness assumption on $f$ to the following.
\begin{assumption}\label{ass:BoundedSubgrad}
The part $f$ in the problem \eqref{eq:Opt} has bounded subgradients, i.e. $\| f'(x)\|_{\ast} \leq M_f$ for all $x \in \setX$ and all $f'(x) \in \partial f(x)$.
\end{assumption}

\begin{assumption}\label{ass:Xcompact}
$\setX$ is a nonempty convex compact set. 
\end{assumption}
Let $h\in\scrH_{\alpha}(\setX)$ be a given \ac{DGF} for the feasible set $\setX\subseteq\setV$. 
 \begin{assumption}\label{ass:hbound}
 The \ac{DGF} $h\in\scrH_{\alpha}(\setX)$ is nonnegative on $\setX$ and upper bounded. Denote by 
 \begin{equation}
 \Omega_{h}(\setX) :=\max_{p\in\setX}h(p)-\min_{p\in\setX}h(p)
 \end{equation}
 the $h$-diameter of the set $\setX$.
 \end{assumption}
We emphasize that Assumption \ref{ass:hbound} implies that $\Omega_{h}(\setX)<\infty$.
 \begin{assumption}
 For all $x\in\setX$ we have $r(x)\geq 0$. 
 \end{assumption}
 
 Define the \emph{mirror map} 
 \begin{equation}
 Q_{\beta,\gamma}(y):=\argmax_{x\in\setX}\left\{\inner{y,x}-\beta h(x)-\gamma r(x)\right\}. 
 \end{equation}
 
Our terminology is motivated by the working of the Dual Averaging method. Given the current primal-dual pair $(x,y)$, DA performs a gradient step in the dual space $\setV^{\ast}$ to produce a new gradient feedback point $y^{+}=y-\lambda\nabla f(x)$, where $\lambda>0$ is a step size parameter. Taking this as a new signal, we update the primal state by applying the mirror map
$x^{+}=Q_{\beta,\gamma}(y^{+}).$ 
 \begin{algorithm}{\bf{The Dual Averaging method} (DA)}
\\
	{\bf Input:} pick $y^{0}=0,x^{0}=Q_{\beta_{0},\gamma_{0}}(0)$, nondecreasing learning sequence $(\beta_{k})_{k\in\N_{0}},(\gamma_{k})_{k\in\N_{0}}$ and non-increasing step-size sequence $(\lambda_{k})_{k\in\N_{0}}$\\
	{\bf General step:} For $k=0,1,\ldots$ do:\\
	\qquad dual update $y^{k+1}=y^{k}-\lambda_{k} f'(x^{k})$,\\
	\qquad set $x^{k+1}=Q_{\beta_{k+1},\gamma_{k+1}}(y^{k+1})$.
\end{algorithm}

\begin{remark}
If $r=0$ on $\setX$ we recover the dual averaging scheme of \cite{Nes09}. The definition of this method can be simplified to the following primal-dual updating scheme:
\begin{equation}\label{eq:DA}
\left\{\begin{array}{ll}
y^{k+1}=y^{k}-\lambda_{k} f'(x^{k}),y^{0}\text{ given,}\\
x^{k+1}=Q_{\beta_{k+1}}(y^{k+1})=\argmax_{x\in\setX}\{\inner{y^{k+1},x}-\beta_{k+1}h(x)\}.
\end{array}\right.
\end{equation}
We will revisit this scheme more thoroughly in Sections \ref{sec:DA-MD} and \ref{sec:dynamics}.
\end{remark}

We now assess the iteration complexity of DA, showing that it features the same order convergence rate $O(1/\sqrt{k})$ as BPGM and MD. 
 
 Fix an arbitrary anchor point $p\in\setX$ and define for given parameters $\beta,\gamma\in[0,\infty)$ the function 
 \begin{equation}
\label{eq:DA_proof_1}
 H_{\beta,\gamma}(y):=\max_{x\in\setX}\left\{\inner{y,x-p}-\beta h(x)-\gamma r(x)\right\}
 \end{equation}
 The mapping $x\mapsto \beta h(x)+\gamma r(x)$ is $\alpha_{\omega}:=(\alpha\beta+\gamma\mu)$-strongly convex. Applying Proposition \ref{prop:Rock}, the function $y\mapsto H_{\beta,\gamma}(y)$ is convex and continuously differentiable with 
 \begin{equation}\label{eq:H1}
 H_{\beta,\gamma}(y+g)\leq H_{\beta,\gamma}(y)+\inner{\nabla H_{\beta,\gamma}(y),g}+\frac{1}{2\alpha_{\omega}}\norm{g}^{2}_{\ast}\quad\forall y,g\in\setV^{\ast}.
 \end{equation}
 Moreover, if $\beta_{1}\geq\beta_{2}$ and $\gamma_{1}\geq \gamma_{2}$, then it is easy to see that 
 $$
 H_{\beta_{1},\gamma_{1}}(y)\leq H_{\beta_{2},\gamma_{2}}(y)\qquad\forall y\in\setV^{\ast}.
 $$
An important consequence of strong convexity is the following relation (see e.g. \cite{Nes05})
$$
Q_{\beta,\gamma}(y)-p=\nabla H_{\beta,\gamma}(y)\qquad \forall y\in\setV^{\ast}. 
$$
 To simplify the notation, let $H_{k}(y)\equiv H_{\beta_{k},\gamma_{k}}(y)$ for all $y\in\setV^{\ast}$. Thanks to the monotonicity in the parameters, we get through some elementary manipulations the relation
\begin{align*}
H_{k}(y^{k+1})\geq H_{k+1}(y^{k+1})+(\beta_{k+1}-\beta_{k})h(x^{k+1})+(\gamma_{k+1}-\gamma_{k})r(x^{k+1}).
\end{align*}
By Assumption \ref{ass:hbound}, we know that $h\geq 0$ on $\setX$. Indeed, this can be achieved by a simple shift of the graph of the function, if it is not satisfied from the beginning. Continuing under this nonnegativity assumption, and using $\beta_{k+1}\geq\beta_{k}$, we arrive at the estimate 
$$
H_{k}(y^{k+1})\geq H_{k+1}(y^{k+1})+(\gamma_{k+1}-\gamma_{k})r(x^{k+1}).
$$
Now, we impose some further structure on the choice of the step-sizes $\gamma_{k},\lambda_{k}$. Specifically, assume that 
\[
\gamma_{k+1}=\gamma_{k}+\lambda_{k},\gamma_{0}=0, 
\]
so that $\Lambda_{k}:=\sum_{i=0}^{k}\lambda_{i}\equiv \gamma_{k}$. From this, via equation \eqref{eq:H1}, we arrive at the upper bound
\begin{align*}
H_{k+1}(y^{k+1})+\lambda_{k}r(x^{k+1})\leq H_{k}(y^{k})-\lambda_{k}\inner{f'(x^{k}),x^{k}-p}+\frac{\lambda^{2}_{k}}{2\alpha_{k}}\norm{f'(x^{k})}_{\ast}^{2},
\end{align*}
where $\alpha_{k}=\beta_{k}\alpha+\gamma_{k}\mu$ is the strong concavity parameter of the maximization problem in \eqref{eq:DA_proof_1} at iteration $k$. Rearranging and using the convexity of the part $f$ gives 
$$
\lambda_{k}[f(x^{k})-f(p)+r(x^{k+1})]\leq H_{k}(y^{k})-H_{k+1}(y^{k+1})+\frac{\lambda^{2}_{k}}{2\alpha_{k}}\norm{f'(x^{k})}_{\ast}^{2}. 
$$
Summing over $k=0,1,\ldots,N$, this gives 
$$
\sum_{k=0}^{N}\lambda_{k}[f(x^{k})-f(p)+r(x^{k+1})]\leq H_{0}(y^{0})-H_{N+1}(y^{N+1})+\sum_{k=0}^{N}\frac{\lambda^{2}_{k}}{2\alpha_{k}}\norm{f'(x^{k})}_{\ast}^{2}. 
$$
Since $h,r\geq 0$ on $\setX$, it is clear that $H_{0}(0)\leq 0$. Moreover, $H_{N+1}(y^{N+1})\geq -\beta_{N+1} h(p)-\gamma_{N+1}r(p)$, where $p\in\setX$ is the chosen anchor point. This, using the bounded subgradients assumption, leads to the weaker estimate 
\begin{align*}
\sum_{k=0}^{N}\lambda_{k}[\Psi(x^{k})-\Psi(p)+r(x^{k+1})+r(p)-r(x^{k})]\leq \beta_{N+1}h(p)+\Lambda_{N}r(p)+\sum_{k=0}^{N}\frac{\lambda^{2}_{k}M^{2}_{f}}{2\alpha_{k}}. 
\end{align*}
Jensen's inequality applied to the ergodic average 
$$
\bar{x}_{N}:=\frac{1}{\Lambda_{N}}\sum_{k=0}^{N}\lambda_{k}x^{k}
$$
gives us further 
$$
\Lambda_{N}[\Psi(\bar{x}^{N})-\Psi(p)]\leq \beta_{N+1}\Omega_{h}(\setX)+\lambda_{0}r(x^{0})+\sum_{k=0}^{N}\frac{\lambda^{2}_{k}M^{2}_{f}}{2\alpha_{k}}.
$$
Let us now make the concrete choice of parameters 
\begin{align*}
\beta_{k}=\beta>0 \text{ and }\lambda_{k}=\frac{1}{\sqrt{k+1}}\qquad\forall k\geq 0.
\end{align*} 
Then, for all $k\geq 1$, classical Calculus arguments (see e.g. \cite{Beck17}, Lemma 8.26) yield the bounds 
\begin{align*}
&\gamma_{k}=\Lambda_{k}\geq\sqrt{k+1},\alpha_{k}=\alpha+\Lambda_{k}\mu\geq\alpha+\frac{\mu}{\sqrt{k+1}},\\
&\frac{\lambda^{2}_{k}}{\alpha_{k}}\leq \frac{1}{\alpha(k+1)},\text{ and }\sum_{k=0}^{N}\frac{\lambda^{2}_{k}}{2\alpha_{k}}\leq \frac{1+\log(N+1)}{2\alpha}.
\end{align*}
to get the $O(\log(N)/\sqrt{N})$ bound
\begin{equation}\label{eq:complexityDA}
\Psi(\bar{x}_{N})-\Psi_{\min}(\setX)\leq\frac{\beta\Omega_{h}(\setX)+r(x^{0})+\frac{M_{f}^{2}}{2\alpha}(1+\log(N+1))}{\sqrt{N+1}}.
\end{equation}

\paragraph{The effectiveness of non-Euclidean setups}
With the help of the explicit rate estimate \eqref{eq:complexityDA} we are now in the position to evaluate the potential efficiency gains we can make by adopting the non-Euclidean framework. To do so, assume that we are interested in estimating the accuracy obtained when running DA over an a-priori fixed window $\{0,\ldots,N\}$. If the optimizer commits at the beginning to this decision, then a more efficient step size strategy can be constructed by setting 
\begin{equation}
\beta_{k}=\beta,\text{ and }\lambda_{k}=\frac{\sqrt{2\alpha\beta\Omega_{h}(\setX)}}{\sqrt{N+1}M_{f}}. 
\end{equation}
By doing this, we obtain $\Lambda_{N}=\sqrt{2\alpha\beta\Omega_{h}(\setX)(N+1)}$, and therefore \eqref{eq:complexityDA} reads as 
\begin{align*}
\Psi(\bar{x}^{N})-\Psi(p)\leq\frac{(\beta+1)\Omega_{h}(\setX)+r(x^{0})}{\sqrt{2\alpha\beta(N+1)}\sqrt{\Omega_{h}(\setX)}}.
\end{align*}
Assuming that $r(x^{0})=0$, the complexity estimate becomes $O(\sqrt{\Omega_{h}(\setX)/N}M_{f})$. We now illustrate how the factor $\sqrt{\Omega_{h}(\setX)}M_{f}$ depends on the choice of the Bregman setup. 

\begin{example}
Assume that $\setX=\{x\in\Rn_{+}\vert \sum_{i=1}^{n}x_{i}=1\}$. We investigate the complexity of DA under two different potentially interesting Bregman proximal setups. 
\begin{enumerate}
\item Endow the set $\setX$ with the $\ell_{2}$ norm $\norm{\cdot}=\norm{\cdot}_{2}$. Then $\norm{\cdot}_{\ast}=\norm{\cdot}_{2}$ and $M_{f}\equiv M_{f}^{(2)}=\sup_{x\in\setX}\norm{f'(x)}_{2}$. It can be easily computed that $\Omega_{h}(\setX)=\frac{n-1}{2n}\approx 1/2$ for $n\to\infty$. 
\item Enow the set $\setX$ with the $\ell_{1}$ norm $\norm{\cdot}=\norm{\cdot}_{1}$. We then have $\norm{\cdot}_{\ast}=\norm{\cdot}_{\infty}$. Set $M_{f}\equiv M_{f}^{\infty}=\sup_{x\in\setX}\norm{f'(x)}_{\infty}$. As \ac{DGF} let us consider $h(x)=\sum_{i=1}^{n}x_{i}\ln(x_{i})$. Then, $\Omega_{h}(\setX)=\ln(n)$. 
\end{enumerate}
Since $\norm{a}_{\infty}\leq \norm{a}_{2}\leq \sqrt{n}\norm{a}_{\infty}$, we see that $1\leq \frac{M_{f}^{2}}{M_{f}^{\infty}}\leq \sqrt{n}$, and hence 
\begin{align*}
\frac{\sqrt{(n-1)/(2n)}}{\sqrt{\ln(n)}}\leq \frac{\sqrt{(n-1)/(2n)}}{\sqrt{\ln(n)}}\frac{M_{f}^{(2)}}{M_{f}^{\infty}}\leq \frac{\sqrt{(n-1)/(2n)}}{\sqrt{\ln(n)}}\sqrt{n}.
\end{align*}
Thus, in particular for $n$ large, it can be seen that the $\ell_{1}$-setup is never worse than the $\ell_{2}$-setup, and there can be strong reasons to prefer the non-Euclidean $\ell_{1}$ setup over the $\ell_{2}$ setup. 
\end{example}

 \subsubsection{On the connection between Dual Averaging and Mirror Descent}
 \label{sec:DA-MD}
 A deep and important connection between the Dual Averaging and Mirror Descent algorithms for convex non-smooth optimization has been observed in \cite{BecTeb03}. To illustrate this link, let us particularize our model problem \eqref{eq:Opt} to the constrained convex programming case where $r(x)=0$ on $\setX$. In this case, the dual averaging scheme produces primal-dual iterates via the updates \eqref{eq:DA}. To relate these iterates to BPGM, we assume that $\dom h\subset\setX$ and $h$ is essentially smooth in the sense of Definition \ref{def:Legendre}.

Let us recall that $h$ is essentially smooth if and only if its Fenchel conjugate $h^{\ast}$ is essentially smooth. Moreover, $\nabla h:\Int(\dom h)\to\Int(\dom h^{\ast})$ is a bijection with 
 \begin{equation}\label{eq:hLegendre}
 (\nabla h)^{-1}=\nabla h^{\ast}\text{ and }\nabla h^{\ast}(\nabla h(x))=\inner{x,\nabla h(x)}-h(x)
 \end{equation}
Taking $\setX=\cl(\dom h)$, it follows 
\begin{align*}
\dom \partial h=\rint(\dom h)=\rint(\setX)\text{ with }\partial h(x)=\{\nabla h(x)\}\quad\forall x\in\rint(\setX). 
\end{align*}

Assuming that the penalty function $h$ is of Legendre type, the primal projection step is seen to be the regularized maximization step
\begin{align*}
x^{k}=\argmax_{u\in\setX}\{\inner{y^{k},u}-\beta_{k}h(u)\}\iff y^{k}=\beta_{k}\nabla h(x^{k}).
\end{align*}
Using the definition of the dual trajectory, we see that for all $k\geq 0$ the primal-dual relation obeys: 
\begin{align*}
0=\lambda_{k}\nabla f(x^{k})+\beta_{k+1}\nabla h(x^{k+1})-\beta_{k}\nabla h(x^{k}).
\end{align*}
Assuming that $\beta_{k}\equiv 1$, this implies  
\begin{align*}
x^{k+1}\in\argmin_{u\in\setX}\{\inner{\lambda_{k}\nabla f(x^{k}),u-x^{k}}+D_{h}(u,x^{k})\}=\scrP_{0}(x^{k},\lambda_{k}\nabla f(x^{k})). 
\end{align*}
We have thus shown that DA and BPGM/MD agree if all parameters and initial conditions are chosen in the same way. 
 
 \subsubsection{Links to continuous-time dynamical systems}
 \label{sec:dynamics}
The connection between numerical algorithms and continuous-time dynamical systems for optimization is classical and well-documented in the literature (see e.g.  \cite{HelMoo96} for a textbook reference). Here we describe an interesting link between dual averaging and a class or Riemannian gradient flows originally introduced in \cite{ABB04,AttBolRedTeb04,AttTeb04} and further studied in \cite{BolTeb03}. A complexity analysis of discretized versions of these gradient flows has recently been obtained in \cite{HBA-Linear}. Our point of departure is the following continuous-time dynamical system based on dual averaging, which has been introduced in \cite{MerSta18} in the context of convex programming and in \cite{MerStaJOTA18} for general monotone variational inequality problems. The main ingredient of this dynamical system is a pair of primal-dual trajectories $(x(t),y(t))_{t\geq 0}$ evolving in continuous time according to the differential-projection system
\begin{equation}\label{eq:SDA}
\left\{\begin{array}{l}
y'(t):=\frac{\dif y(t)}{\dif t}=-\lambda(t)\nabla f(x(t)),\\
x(t)=Q_{1}(\eta(t)y(t))=: Q(\eta(t)y(t)).
\end{array}\right.
\end{equation}
To relate this scheme formally to its discrete-time counterpart \eqref{eq:DA}, let us perform an Euler discretization of the dual trajectory by $y^{k}-y^{k-1}=-\lambda_{k}\nabla f(x^{k})$, and project the resulting point to the primal space by applying the mirror map $Q(\frac{1}{\beta_{k+1}}y^{k+1})$, where $\beta_{k+1}^{-1}$ is the discrete-time learning rate appropriately sampled from the function $\eta(t)$. As in Section \ref{sec:DA-MD}, let us assume that the mirror map is generated by a Legendre function $h$, so that 
$$
x(t)=\nabla h^{\ast}(\eta(t)y(t)).
$$
Let us further assume that $h$ is twice continuously differentiable and $\eta(t)\equiv 1$. Differentiating the previous equation with respect to time $t$ gives 
$$
x'(t)=\nabla^{2}h^{\ast}(y(t))y'(t)=-\lambda(t)\nabla^{2}h^{\ast}(y(t))\nabla f(x(t)).
$$
To make headway, recall the basic properties of Legendre function saying $\nabla h^{\ast}(\nabla h(x))=x$ for all $x\in\Int\dom h$ (cf. \eqref{eq:hLegendre}). Differentiating implicitly this identity, we obtain $\nabla^{2}h^{\ast}(\nabla h(x)))\equiv\Id$, or
\begin{equation}\label{eq:Hinverse}
\nabla^{2}h^{\ast}(\nabla h(x))=[\nabla^{2}h(x)]^{-1}=: H(x)^{-1}. 
\end{equation}
As in Section \ref{sec:DA-MD}, it holds true that $y(t)=\nabla h(x(t))$ for all $t\geq 0$, we therefore obtain the interesting characterization of the primal trajectory as 
$$
x'(t)=-\lambda(t)H(x(t))^{-1}\nabla f(x(t)).
$$
If $\setX$ is a smooth manifold, we can define a Riemannian metric 
$$
g_{x}(u,v):=\inner{H(x)u,v}\qquad \forall (x,u,v)\in\setX^{\circ}\times\setV\times\setV.
$$
The gradient of a smooth function $\phi$ with respect to the metric $g$ is then given by $\nabla_{g}\phi(x)=H(x)^{-1}\nabla \phi(x)$. Hence, the continuous-time version of the dual averaging method gives rise the class of primal \emph{Riemannian-Hessian gradient flows} 
\begin{equation}\label{eq:HG}
x'(t)+\lambda(t)\nabla_{g}f(x(t))=0,\quad x(0)\in\setX^{\circ}.
\end{equation}
This class of continuous-time dynamical systems gave rise to a vigorous literature in connection with Nesterov's optimal method, which we will thoroughly discuss in Section \ref{sec:accelerated}. As an appetizer, consider the system of differential equations 
\begin{equation}\label{eq:AMD}
y'(t)=-\lambda(t)\nabla f(x(t)),\quad x'(t)=\gamma(t)[Q(\eta(t)y(t))-x(t)].
\end{equation}
Suppose that in \eqref{eq:AMD} we take $Q(y)=y,\eta(t)=1$. This corresponds 

to the Legendre function $h(x)=\frac{1}{2}\norm{x}^{2}_{2}+\delta_{\setX}(x)$ for a given closed convex set $\setX$. Under this specification, the dynamical system \eqref{eq:AMD} becomes 
$$
y'(t)=-\lambda(t)\nabla f(x(t)),\quad x'(t)=\gamma(t)[y(t)-x(t)]. 
$$
Combining the primal and the dual trajectory, we easily derive a purely primal second-order in time dynamical system given by
\begin{align*}
x''(t) -x'(t)\left(\frac{\gamma(t)^{2}-\gamma'(t)}{\gamma(t)}\right)+\lambda(t)\nabla f(x(t))=0.
\end{align*}
Setting $\gamma(t)=\beta/t$ and $\lambda(t)=1/\gamma(t)$ and rearranging gives 
\begin{align*}
x''(t)+\frac{\beta+1}{t}x'(t)+\nabla f(x(t))=0, 
\end{align*}
which corresponds to the continuous-time version of the Heavy-ball method of Polyak \cite{Pol64}. For $\beta=2$ this gives the continuous-time formulation of Nesterov's accelerated scheme, as shown by \cite{SuBoyCan16}.  

More generally, suppose that $h$ is a twice continuously differentiable Legendre function and $\eta(t)\equiv 1$. Then a direct calculation shows that 
\begin{align*}
x''(t)+\left(\gamma(t)-\frac{\gamma'(t)}{\gamma(t)}\right)x'(t)+\gamma(t)\lambda(t)(\nabla^{2}h)^{-1}(Q(y(t)))\nabla f(x(t))=0. 
\end{align*}
Using the identity \eqref{eq:Hinverse}, as well as $\frac{x'(t)}{\gamma(t)}+x(t)=\nabla h^{\ast}(y(t))$, it follows that 
\begin{align*}
\nabla^{2}h\left(x(t)+\frac{x'(t)}{\gamma(t)}\right)\left(\frac{x''(t)}{\gamma(t)}+\left(1-\frac{\gamma'(t)}{\gamma(t)^2}\right)x'(t)\right)=-\lambda(t)\nabla f(x(t))
\iff  \frac{\dif}{\dif t}\nabla h\left(x(t)+\frac{x'(t)}{\gamma(t)}\right)=-\lambda(t)\nabla f(x(t)).
\end{align*}
This shows that for $\eta\equiv 1$, the dynamic coincides with the Lagrangian family of second-order systems constructed in \cite{Wib16}. These ideas are now investigated heavily when combined with numerical discretization schemes for dynamical system with the hope to get insights how to construct new and more efficient algorithmic formulation of gradient-methods. This literature grew quite fastly over the last years, and we mention \cite{AttChbPeyRed18,Bah:2019aa,Shi:2019aa,Attouch:2020aa}.

\section{The Proximal Method of Multipliers and ADMM}
\label{sec:ADMM}
%
In this section we turn our attention to a classical toolbox for solving linearly constrained optimization problems building on the classical idea of the celebrated \emph{method of multipliers}. An extremely powerful proponent of this class of algorithms is the \emph{Alternating Direction Method of Multipliers} (ADMM), which has received enormous interest from different directions, including PDEs \cite{AttRedSou07,AttCabFraPey11}, mixed-integer programming \cite{AhmFeiSun17}, optimal control \cite{FarJovLin12} and signal processing \cite{Yua12,YanZha11}. The very influential monograph \cite{BoyChuEckADMM11} contains over 180 references, reflecting the deep impact of alternating methods on optimization theory and its applications. Following the general spirit of this survey, we introduce alternating direction methods in a proximal framework, as pioneered by Rockafellar \cite{Roc76, Roc76AL}, and due to \cite{SheTeb14}. See also \cite{BanBotCse16} for some further important elaborations. 

To set the stage, consider the composite convex optimization problem \eqref{eq:Opt}, in its special form \eqref{eq:structured}. Hence, we are interested in minimizing the composite convex function 
$$
\Psi(x)=g(\bA x)+r(x),
$$
for a given bounded linear operator $\bA$. To streamline the presentation, we directly assume in this section that $\setV=\setV^{\ast}=\Rn$, and the underlying metric structure is generated by the Euclidean norm $\norm{a}\equiv \norm{a}_{2}=\inner{a,a}^{1/2}=\left(\sum_{i=1}^{n}a_{i}\right)^{1/2}$. 
Introducing the auxiliary variable $z=\bA x$, this problem can be equivalently written as 
\begin{equation}\label{eq:P-ADMM}
\inf\{\Phi(x,z)=g(z)+r(x)\vert \bA x-z=0,x\in\setX,z\in\setZ\},
\end{equation}
where $\setX=\Rn$ and $\setZ=\R^{m}$. We will call this the \emph{primal} problem. By Fenchel-Rockafellar duality \cite{BauCom16}, the \emph{dual problem} to \eqref{eq:P-ADMM} is 
\begin{equation}\label{eq:D-ADMM}
\min_{y} g^{\ast}(y)+r^{\ast}(-\bA^{\top}y).
\end{equation}

The Lagrangian associated to \eqref{eq:P-ADMM} is
\begin{equation}
L(x,z,y)=g(z)+r(x)+\inner{y,\bA x-z},
\end{equation}
where $y\in\R^{m}$ is the Lagrange multiplier associated with the linear constraint. 

\begin{assumption}
The Lagrangian $L$ associated to problem \eqref{eq:P-ADMM} has a saddle point, i.e. there exists $(x^{\ast},z^{\ast},y^{\ast})$ such that 
\begin{equation}
L(x^{\ast},z^{\ast},y)\leq L(x^{\ast},z^{\ast},y^{\ast})\leq L(x,z,y^{\ast})\qquad\forall (x,z,y)\in\setX\times\setZ\times\R^{m}.
\end{equation}
\end{assumption}

A key actor in alternating direction methods is the \emph{augmented Lagrangian} defined for some $c>0$ as 
\begin{equation}\label{eq:AL}
L_{c}(x,z,y)=r(x)+g(z)+\inner{y,Ax-z}+\frac{c}{2}\norm{Ax-z}_{2}^{2}.
\end{equation}

 \begin{algorithm}{\bf{The Alternating Direction of Method of Multipliers} (ADMM)}
\\
	{\bf Input:} pick $(z^{0},y^{0})\in\setZ\times\R^{m}$ and penalty parameter $c>0$;\\
	{\bf General step:} For $k=0,1,\ldots$ do:
	\begin{align}  
  x^{k+1}&= \argmin_{x\in\setX}\{r(x)+\frac{c}{2}\norm{\bA x-z^{k}+\frac{1}{c}y^{k}}_{2}^{2}\}\\
  z^{k+1}&= \argmin_{z\in\setZ}\{g(z)+\frac{c}{2}\norm{\bA x^{k+1}-z+\frac{1}{c}y^{k}}^{2}_{2}\}\\
  y^{k+1}&=y^{k}+c(\bA x^{k+1}-z^{k+1}).
  \end{align}	
\end{algorithm}

ADMM updates the decision variables in a sequential manner, and thus is not capable of featuring parallel updates which are often required in large-scale distributed optimization problems. In the context of the AC optimal power flow problem in electric power grid optimization \cite{Sun13} provide such a modification of ADMM. Furthermore, the ADMM can be extended to consider formulations with general linear constraints of the form $\bA_1 x+\bA_2 z=b$. For ease of exposition we stick to the simplified problem formulation above.    

\subsection{The Douglas-Rachford algorithm and ADMM}
\label{sec:DR}
The Douglas-Rachford (DR) algorithm is a fundamental method to solve general monotone inclusion problems where the task is to find zeros of the sum of two maximally monotone operators (see \cite{BauCom16} and \cite{AusTebBook06}). To keep the focus on convex programming, we introduce this method for solving the dual problem \eqref{eq:D-ADMM}. To that end, let us define the matrix $\bK=-\bA^{\top}$, so that our aim is to solve the convex programming problem
\begin{equation}\label{eq:DR}
\min_{z}g^{\ast}(z)+r^{\ast}(\bK z).
\end{equation}
Any solution $\bar{z}\in \dom(r^{\ast})$ satisfies the monotone inclusion 
\begin{equation}
0\in \bK^{\top}\partial r^{\ast}(\bK \bar{z})+\partial g^{\ast}(\bar{z}).
\end{equation}
The DR algorithm aims to determine such a point $\bar{z}$ by iteratively constructing a sequence $\{(u^{k},v^{k},y^{k}),k\geq 0\}$ determined by 
\begin{align*}
v^{k+1}&=(\Id+c\bK^{\top}\circ\partial r^{\ast}\circ\bK)^{-1}(2y^{k}-u^{k}),\\
u^{k+1}&=v^{k+1}+u^{k}-y^{k},\\
y^{k+1}&=(\Id+c \partial g^{\ast})^{-1}(u^{k+1}).
\end{align*}
To bring this into an equivalent form, let us focus on the definition of the $y^{k+1}$ update, which reads as the inclusion 
$$
0\in\frac{1}{c}(y^{k+1}-u^{k+1})+\partial g^{\ast}(y^{k+1}).
$$
This is clearly recognizable as the first-order optimality condition of the $\min_{y}\{g^{\ast}(y)+\frac{1}{2c}\norm{y-u^{k+1}}^{2}_{2}\}$. 
Therefore, we can rewrite the above iteration in terms of convex optimization subroutines as: 
\begin{align}
v^{k+1}&=\argmin_{v}\{r^{\ast}(\bK v)+\frac{1}{2c}\norm{v-(2 y^{k}-u^{k})}_{2}^{2}\}\label{eq:DRv},\\
u^{k+1}&=v^{k+1}+u^{k}-w^{k}\label{eq:DRu},\\
y^{k+1}&=\argmin_{y}\{g^{\ast}(y)+\frac{1}{2c}\norm{y-u^{k+1}}_{2}^{2}\}.\label{eq:DRy}
\end{align}
Via Fenchel-Rockafellar duality, the dual problem to \eqref{eq:DRv} reads as 
\begin{align*}
x^{k+1}=\argmin_{x}\{r(x)+\frac{c}{2}\norm{\bA x+\frac{1}{c}(2 y^{k}-u^{k})}_{2}^{2}\},
\end{align*}
where the coupling between the primal and the dual variables is 
\[
u^{k+1}=y^{k}+c\bA x^{k+1}.
\]
The dual to step \eqref{eq:DRy} reads as
\begin{align*}
z^{k+1}=\argmin_{z}\{g(z)+\frac{c}{2}\norm{z-\frac{1}{c}u^{k+1}}_{2}^{2}\}.
\end{align*}
The coupling between primal and dual variables reads as 
\[
y^{k+1}=u^{k+1}-c z^{k+1}.
\]
Combining all these relations, we can write the dual minimization problem as 
\begin{align*}
x^{k+1}&=\argmin_{x}\{r(x)+\frac{c}{2}\norm{\bA x-z^{k}+\frac{1}{c}y^{k}}_{2}^{2}\},\\
z^{k+1}&=\argmin_{z}\{g(z)+\frac{c}{2}\norm{\bA x^{k+1}-z+\frac{1}{c}y^{k}}_{2}^{2}\},\\
y^{k+1}&=y^{k}+c(\bA x^{k+1}-z^{k+1})
\end{align*}
which is just the standard ADMM. By this we have recovered a classical result on connection between the DR and ADMM algorithms due to \cite{Gab83} and \cite{EckBer92}.


\subsection{Proximal Variant of ADMM}
\label{sec:ProxADMM}

One of the limitations of the ADMM comes from the presence of the term $\bA x$ in the update of $x^{k+1}$. The presence of this factor makes it impossible to implement the algorithm in parallel, which makes it slightly unattractive for large-scale problems in distributed optimization. Moreover, due to the result of \cite{chen2016direct} the convergence of ADMM for general linear constraints does not generalize to more than two blocks. Leaving parallelization issues aside, Shefi and Teboulle \cite{SheTeb14} proposed an interesting extension of the ADMM by adding further quadratic penalty terms, which adds stability to the algorithm, and as well allows us to give a unified perspective of Lagrangian methods and prove global convergence results. 

Given some point $(x^{k},z^{k},y^{k})\in\setX\times\setZ\times\R^{m}$ and two positive definite matrices $\bM_{1},\bM_{2}$, we are ready to define the new ingredient of the method. 
\begin{definition}
The proximal augmented Lagrangian of \eqref{eq:P-ADMM} is 
\begin{equation}
P_{k}(x,z,y)=L_{c}(x,z,y)+q_{k}(x,z)
\end{equation}
where 
\begin{equation}
q_{k}(x,z)=\frac{1}{2}\norm{x-x^{k}}^{2}_{\bM_{1}}+\frac{1}{2}\norm{z-z^{k}}_{\bM_{2}}^{2}.
\end{equation}
\end{definition}
Here, $\norm{u}_{\bM}^{2}=\inner{u,\bM u}$ is the semi-norm induced by $\bM$, which is a norm if $\bM$ is positive definite. 

 \begin{algorithm}{\bf{The Alternating Direction proximal Method of Multipliers} (AD-PMM)}
\\
	{\bf Input:} pick $(x^{0},z^{0},y^{0})\in\setX\times \setZ\times\R^{m}$ and penalty parameter $c>0$;\\
	{\bf General step:} For $k=0,1,\ldots$ do:
	\begin{align}  
  x^{k+1}&= \argmin_{x\in\setX}\{r(x)+\frac{c}{2}\norm{\bA x-z^{k}+\frac{1}{c}y^{k}}_{2}^{2}+\frac{1}{2}\norm{x-x^{k}}^{2}_{\bM_{1}}\}
\label{eq:ADPMM-x}  \\
  z^{k+1}&= \argmin_{z\in\setZ}\{g(z)+\frac{c}{2}\norm{\bA x^{k+1}-z+\frac{1}{c}y^{k}}_{2}^{2}+\frac{1}{2}\norm{z-z^{k}}_{\bM_{2}}^{2}\} \label{eq:ADPMM-z}\\
  y^{k+1}&=y^{k}+c(\bA x^{k+1}-z^{k+1})\label{eq:ADPMM-y}.
  \end{align}	
\end{algorithm}

We give a brief analysis of the complexity of AD-PMM in the special case of problem \eqref{eq:P-ADMM}. Recall that a standing hypothesis in this survey is that the smooth part $f$ of the composite convex programming problem \eqref{eq:Opt} admits a Lipschitz continuous gradient. Since $f(x)=g(\bA x)$, the Lipschitz constant of $\nabla f$ is determined by a corresponding Lipschitz assumption on $\nabla g$, with the constant henceforth denoted as $L_{g}$, and a bound on spectrum of the matrix $\bA$. To highlight the primal-dual nature of the algorithm, a key element in the complexity analysis is the bifunction 
\[
S(x,y)=r(x)-g^{\ast}(y)+\inner{y,\bA x}=L(x,0,y).
\]
Our derivation of an iteration complexity estimate of AD-PMM proceeds in two steps. First, we present an interesting ``Meta-Theorem'', due to \cite{SheTeb14}, and reprinted here as Proposition \ref{prop:meta}. It gives a general convergence guarantees for any primal-dual algorithms satisfying a specific per-iteration bound. We then apply this general result to AD-PMM, by verifying that this scheme actually satisfies these mentioned per-iteration bounds.

We start with an auxiliary technical fact.
\begin{lemma}\label{lem:bound}
Let $h:\Rn\to\R$ be a proper convex and $L_{h}$-Lipschitz continuous. Then, for any $\xi\in\Rn$ we have 
\begin{equation}
h(\xi)\leq \max\{\inner{\xi,u}-h^{\ast}(u):\norm{u}_{2}\leq L_{h}\}.
\end{equation}
\end{lemma}
\begin{proof}
Since $h$ is convex and continuous, it agrees with its biconjugate: $h^{\ast\ast}=h$. By Corollary 13.3.3 in \cite{Roc70}, $\dom h^{\ast}$ is bounded with $\dom h^{\ast}\subseteq\{u:\norm{u}_{2}\leq L_{h}\}$. Hence, the definition of the conjugate gives 
$$
h(\xi)=\sup_{u\in\dom h^{\ast}}\{\inner{u,\xi}-h^{\ast}(u)\}\leq\max_{u:\norm{u}_{2}\leq L_{h}}\{\inner{\xi,u}-h^{\ast}(u)\}.
$$
\end{proof}

\begin{proposition}\label{prop:meta}
Let $(x^{\ast},y^{\ast},z^{\ast})$ be a saddle point for $L$. Let $\{(x^{k},y^{k},z^{k});k\geq 0\}$ be a sequence generated by some algorithm for which the following estimate holds for any $y\in\R^{m}$:
\begin{equation}\label{eq:iteration}
L(x^{k},z^{k},y)-\Psi(x^{\ast})\leq \frac{1}{2k}\left[ C(x^{\ast},z^{\ast})+\frac{1}{c}\norm{y-y^{0}}_{2}^{2}\right]
\end{equation}
for some constant $C(x^{\ast},z^{\ast})>0$. Then 
\[
\Psi(x^{k})-\Psi(x^{\ast})\leq \frac{C_{1}(x^{\ast},z^{\ast},L_{g})}{2k}.
\]
where $C_{1}(x^{\ast},z^{\ast},L_{g})=C(x^{\ast},z^{\ast})+\frac{2}{c}(L^{2}_{g}+\norm {y^{0}}_{2}^{2}).$
\end{proposition}

\begin{proof}
Thanks to the Fenchel inequality 
\begin{align*}
L(x,z,y)-S(x,y)=g(z)+g^{\ast}(y)-\inner{y,z}\geq 0.
\end{align*}
By the definition of the convex conjugate 
\begin{align*}
\Psi(x)&=g(\bA x)+r(x)=\sup_{y}\{r(x)+\inner{y,\bA x}-g^{\ast}(y)\}=\sup_{y}S(x,y).
\end{align*}
Now, since $g$ is convex and continuous on $\R^{m}$, we know $g=g^{\ast\ast}$, and we can apply Lemma \ref{lem:bound} to obtain the string of inequalities: 
\begin{align*}
\Psi(x^{k})-\Psi(x^{\ast})&=\sup_{y}\{S(x^{k},y)-\Psi(x^{\ast})\}\leq \sup_{y:\norm{y}_{2}\leq L_{g}}\{S(x^{k},y)-\Psi(x^{\ast})\}
\leq \sup_{y:\norm{y}_{2}\leq L_{g}}\{L(x^{k},z^{k},y)-\Psi(x^{\ast})\}\\
&\leq \sup_{y:\norm{y}_{2}\leq L_{g}}\left\{\frac{1}{2k}\left(C(x^{\ast},z^{\ast})+\frac{1}{c}\norm{y-y^{0}}_{2}^{2}\right)\right\}
\leq \frac{1}{2k}\left[C(x^{\ast},z^{\ast})+\frac{2}{c}(L_{g}+\norm{y^{0}}_{2}^{2})\right].
\end{align*}
\end{proof}
To apply this Meta-Theorem, we need to verify that AD-PMM satisfies the condition \eqref{eq:iteration}. To make progress towards that end, Lemma 4.2 in \cite{SheTeb14} proves that 
\begin{equation}\label{eq:bound1}
L(x^{k+1},z^{k+1},y)-L(x,z,y^{k+1})\leq T_{k}(x,z,x^{k+1})+R_{k}(x,y,z)
\end{equation}
for all $(x,z,y)\in\setX\times\setZ\times\R^{m}$ and some explicitly given functions $T_{k}$ and $R_{k}$. Furthermore, it is shown that 
\begin{align*}
&T_{k}(x,z,x^{k+1})\leq\frac{c}{2}\left(\norm{\bA x-z^{k}}_{2}^{2}-\norm{\bA x-z^{k+1}}_{2}^{2}+\frac{c}{2}\norm{\bA x^{k+1}-z^{k+1}}_{2}^{2}\right),\text{ and }\\
&R_{k}(x,z,y)\leq \frac{1}{2}\left(\Delta_{k}(x,\bM_{1})+\Delta_{k}(z,\bM_{2})+\frac{1}{c}\Delta_{k}(y,\Id)\right)-\frac{c}{2}\norm{\bA x^{k+1}-z^{k+1}}_{2}^{2},
\end{align*}
where for any point $z$ and positive semi-definite matrix $\bM$, 
\begin{align*}
\Delta_{k}(z,\bM)=\frac{1}{2}\norm{z-z^{k}}_{\bM}^{2}-\frac{1}{2}\norm{z-z^{k+1}}^{2}_{\bM}.
\end{align*}
Using these bounds and summing inequality \eqref{eq:bound1} over $k=0,1,\ldots,N-1$, we get 
\[
\sum_{k=0}^{N-1}[L(x^{k+1},z^{k+1},y)-L(x,z,y^{k+1})]\leq \frac{1}{2}\left(c\norm{\bA x-z^{0}}_{2}^{2}+\norm{x-x^{0}}^{2}_{\bM_{1}}+\norm{z-z^{0}}^{2}_{\bM_{2}}+\frac{1}{c}\norm{y-y^{0}}_{2}^{2}\right)
\]
Dividing both sides by $N$ and using the convexity of the Lagrangian with respect to $(x,z)$ and the linearity in $y$, we easily get 
\[
L(\bar{x}_{N},\bar{z}_{N},y)-L(x,z,\bar{y}_{N})\leq \frac{1}{2N}\left(C(x,z)+\frac{1}{c}\norm{y-y^{0}}_{2}^{2}\right)
\]
in terms of the ergodic average
\begin{align*}
\bar{x}_{N}=\frac{1}{N}\sum_{k=0}^{N-1}x^{k},\; \bar{y}_{N}=\frac{1}{N}\sum_{k=0}^{N-1}y^{k},\; \bar{z}_{N}=\frac{1}{N}\sum_{k=0}^{N-1}z^{k},
\end{align*}
and the constant $C(x,z)=c\norm{\bA x-z^{0}}^{2}+\norm{x-x^{0}}^{2}_{\bM_{1}}+\norm{z-z^{0}}^{2}_{\bM_{2}}$. Therefore, we can apply Proposition \ref{prop:meta} to the sequence of ergodic averages $(\bar{x}_{k},\bar{z}_{k},\bar{y}_{k})$ generated by AD-PMM, and derive a $O(1/N)$ convergence rate in terms of the function value.

\subsection{Relation to the Chambolle-Pock primal-dual splitting}
\label{sec:PD}

In this subsection we discuss the relation between ADMM and the celebrated Chambolle-Pock (a.k.a Primal-Dual Hybrid Gradient) method \cite{ChaPoc11}, designed for problems in the form \eqref{eq:saddle}.

\begin{algorithm}{\bf{The Chambolle-Pock primal-dual algorithm} (CP)}
\\
	{\bf Input:} pick $(x^{0},y^{0},p^{0})\in\Rn\times\R^{m}\times\R^{m}$ and $c,\tau>0,\theta\in[0,1]$;\\
	{\bf General step:} For $k=0,1,\ldots$ do:
	\begin{align}  
  x^{k+1}&= \argmin_{x}\{r(x)+\frac{1}{2\tau}\norm{x-(x^{k}-\tau\bA^{\top}p^{k})}_{2}^{2}\label{eq:CP-x}\\
  y^{k+1}&=\argmin_{y}\{g^{\ast}(y)+\frac{1}{2c}\norm{y-(y^{k}+c\bA x^{k+1})}_{2}^{2}\} \label{eq:CP-y}\\ 
  p^{k+1}&=y^{k+1}+\theta(y^{k+1}-y^{k}) \label{eq:CP-p}.
  \end{align}	
\end{algorithm}

For later references it is instructive to write this algorithm slightly differently in operator-theoretic notation. From the optimality condition of the step $x^{k+1}$, we see 
\begin{align*}
0\in \partial r(x^{k+1})+\frac{1}{\tau}(x^{k+1}-w^{k})=(\Id+\tau \partial r)(x^{k+1})-w^{k}
\end{align*}
where $w^{k}=x^{k}-\tau\bA^{\top} p^{k}$. Hence, we can give an explicit expression of the update as 
\begin{align*}
x^{k+1}=(\Id+\tau\partial r)^{-1}(w^{k})=(\Id+\tau\partial r)^{-1}(x^{k}-\tau\bA^{\top}p^{k}).
\end{align*}
Similarly, we can write the update $y^{k+1}$ explicitly as 
\begin{align*}
y^{k+1}=(\Id+c\partial g^{\ast})^{-1}(y^{k}+c\bA x^{k+1}).
\end{align*}
When $\theta=0$ we obtain the classical Arrow-Hurwicz primal-dual algorithm \cite{ArrHurUza58}. For $\theta=1$ the last line in CP becomes $p^{k+1}=2y^{k+1}-y^{k}$, which corresponds to a simple linear extrapolation based on the current and previous iterates. In this case, \cite{ChaPoc11} provide a $O(1/N)$ non-asymptotic convergence guarantees in terms of the primal-dual gap function of the corresponding saddle-point problem. The CP primal-dual splitting method has been of immense importance in imaging and signal processing and constitutes nowadays a standard method for tackling large-scale instances in these application domains. Interestingly, if $\theta=1$, CP is a special case of the proximal version of  ADMM (AD-PMM). To establish this connection, let us set $\bM_{1}=\frac{1}{c}\Id-c\bA^{\top}\bA$ and $\bM_{2}=0$. After some elementary manipulations, we arrive at the update formula for $x^{k+1}$ in AD-PMM \eqref{eq:ADPMM-x} as  
\begin{align*}
x^{k+1}=\argmin_{x}\{r(x)+\frac{1}{2\tau}\norm{x-(x^{k}-\tau\bA^{\top}(y^{k}+c(\bA x^{k}-z^{k})))}_{2}^{2}\}.
\end{align*}
Introducing the variable $p^{k}=y^{k}+c(\bA x^{k}-z^{k})$, the above reads equivalently as 
\begin{align*}
x^{k+1}=\argmin_{x}\{r(x)+\frac{1}{2\tau}\norm{x-(x^{k}-\tau\bA^{\top}p^{k})}_{2}^{2}\}=\prox_{\tau r}(x^{k}-\tau\bA^{\top}p^{k}).
\end{align*}
For $\bM_{2}=0$, the second update step in AD-PMM \eqref{eq:ADPMM-z} reads as 
\[
z^{k+1}=(\Id+\frac{1}{c}\partial g)^{-1}\left(\bA x^{k+1}+\frac{1}{c}y^{k}\right)=\prox_{\frac{1}{c}g}\left(\frac{1}{c}(c\bA x^{k+1}+y^{k})\right).
\]
Moreau's identity \cite[][Proposition 23.18]{BauCom16} states that  
\begin{equation}\label{eq:Moreau}
c\prox_{\frac{1}{c}g}(u/c)+\prox_{c g^{\ast}}(u)=u\quad\forall u\in\setV. 
\end{equation}
Applying this fundamental identity, we see 
\begin{align*}
c z^{k+1}+\prox_{c g^{\ast}}(y^{k}+c\bA x^{k+1})=y^{k}+c\bA x^{k+1}.
\end{align*}
The second summand is just the $y^{k+1}$-update in the CP algorithm, so that we deduce 
\begin{align*}
c z^{k+1}+y^{k+1}=y^{k}+c\bA x^{k+1}\iff y^{k+1}=y^{k}+c(\bA x^{k+1}-z^{k+1}).
\end{align*}
Consequently, 
\begin{align*}
p^{k+1}=y^{k+1}+c(\bA x^{k+1}-z^{k+1})=2y^{k+1}-y^{k},
\end{align*}
and hence we recover the three-step iteration defining CP:
\begin{align*}  
  x^{k+1}&= \argmin_{x}\{r(x)+\frac{1}{2\tau}\norm{x-(x^{k}-\tau\bA^{\top}p^{k})}_{2}^{2}\}\\
  y^{k+1}&= \argmin_{y}\{g^{\ast}(y)+\frac{1}{2c}\norm{y-(y^{k}+c\bA x^{k+1})}_{2}^{2}\}\\ 
  p^{k+1}&=2y^{k+1}-y^{k}.
\end{align*}

Given the above derivations, we can summarize this subsection by the following interesting observation.
\begin{proposition}[Proposition 3.1, \cite{SheTeb14}]
Let $(x^{k},y^{k},p^{k})$ be a sequence generated by CP with $\theta=1$. Then, the $y^{k+1}$-update \eqref{eq:CP-y} is equivalent to
\begin{align*}
&z^{k+1}=\argmin_{z}\{g(z)+\frac{c}{2}\norm{\bA x^{k+1}-z+\frac{1}{c}y^{k}}_{2}^{2}\},\\
& y^{k+1}=y^{k}+c(\bA x^{k+1}-z^{k+1})
\end{align*}
which corresponds to the primal $z^{k+1}$-minimization step \eqref{eq:ADPMM-z} with $\bM_{2}=0$, and to the dual multiplier update for $y^{k+1}$ \eqref{eq:ADPMM-y} of AD-PMM, respectively. Moreover, the minimization step with respect to $x$ in the CP algorithm given in \eqref{eq:CP-x} together with \eqref{eq:ADPMM-y} reduces to \eqref{eq:ADPMM-x} of AD-PMM with $\bM_{1}= \tau\Id-c\bA^{\top}\bA$. 
\end{proposition}

\section{The Conditional Gradient Method}
\label{sec:CG}

The Bregman proximal gradient method is an efficient first-order method whenever the prox-mapping can be evaluated efficiently. In this section, we present a class of first-order methods for convex programming problems which gain relevance in large-scale problems for which the computation of the prox-mapping is a significant computational bottleneck. We describe \emph{conditional gradient} (CG) methods, a family of methods which, originating in the 1960's, have received much attention in both machine learning and optimization in the last 10 years. CG is designed to be a method which solves convex programming problems over compact convex sets. Therefore, we assume in this section that the feasible set $\setX$ is a compact convex set. 

\begin{assumption}\label{ass:CG-X}
The set $\setX$ is a compact convex subset in a finite-dimensional real vector space $\setV$.
\end{assumption}

\subsection{Classical Conditional gradient}
To set the stage for the material presented in this section, we give a quick summary on the main developments of the classical CG method. CG, also known as the \emph{Frank-Wolfe method}, was independently suggested by Frank and Wolfe \cite{frank1956algorithm} for linearly constrained quadratic problems and  by Levitin and Polyak \cite{levitin1966constrained} for solving problem \eqref{eq:structured} with $r(x)\equiv0$ and a general compact set $\setX$, i.e.,
\begin{align}\label{prob:CG}
\Psi_{\min}(\setX):=\min\{f(x)\vert x\in\setX\}.
\end{align}
CG attempts to solve problem \eqref{prob:CG} by sequentially calling a \emph{linear oracle} (LO).

\begin{definition}
The Operator $\scrL_{\setX}:\setV^{\ast}\rightarrow\setX$ is a \emph{linear oracle} (LO) over set $\setX$ if for any vector $y\in\setV^{\ast}$ we have that
\begin{equation}\label{eq:CG-LO}
\scrL_{\setX}(y)\in \argmin_{s\in\setX}\inner{y,s}.
\end{equation}
\end{definition}
The practical application of an LO requires to make a selection from the set of solutions of the defining linear minimization problem. The precise definition of such a selection mechanism is not of any importance, and thus we are just concerned with any answer $\scrL_{\setX}(y)$ revealed by the oracle.

The information-theoretic assumption that the optimizer can only query a linear minimization oracle is clearly the main difference between CG and other gradient-based methods discussed in Section \ref{sec:MD}. For instance, the dual averaging algorithm solves at each iteration a strongly convex subproblem of the form 
\begin{equation}\label{eq:SPDA}
\min_{u\in\setX}\{\inner{y,u}+h(u)\},
\end{equation}
where $h\in\scrH_{\alpha}(\setX)$, whereas CG solves a single linear minimization problem at each iteration. This difference in the updating mechanism yields  the following potential advantages of the CG method. 
\begin{enumerate}
\item \underline{Low iteration costs:} In many cases it is much easier to construct an LO rather than solving the non-linear subproblem \eqref{eq:SPDA}. We emphasize that this potential benefit of CG does not depend on the structure of the objective function $f$, but rather on the geometry of the feasible set $\setX$. To illustrate this point, consider the set $\setX=\{\XX\in\R^{n\times n}_{\text{sym}}\vert \XX\succeq 0,\tr(\XX)\leq 1\}$, known as the spectrahedron (cf. Example \ref{ex:spectrahedron}). Computing the orthogonal projection of some symmetric matrix $\YY$ onto the spectrahedron requires first to compute the full spectral decomposition $\YY=\UU \DD \UU^\top$, and then for the diagonal matrix $\DD$ computing the projection of its diagonal elements onto the simplex. The resulting projection is therefore given by
	\[
	P_\setX(\YY)=\UU\Diag(P_{\Delta_n}(\diag(\DD)))\UU^\top.
	\]
	In contrast, computing a linear oracle over $\setX$ for the symmetric matrix $\YY$ involves finding the eigenvector of $\YY$ corresponding to the minimal   
	eigenvalue, that is $\scrL_{\setX}(\YY)=uu^\top$, where $u^{\top}\YY u=\lambda_{\min}(\YY)$. This operation can be typically done using such methods as Power, Lanczos or Kaczmarz, and randomized versions thereof - see \cite{KucWoz92} for general complexity results. For large-scale problems, computing such a leading eigenvector to a predefined accuracy is much more efficient than a full spectral decomposition. 
\item \underline{Simplicity:} The definition of an LO does not rely on a specific DGF $h\in\scrH_{\alpha}(\setX)$ and makes the update affine invariant.
\item \underline{Structural properties of the updates:}  When the feasible set $\setX$ can be represented as the convex hull of a countable set of atoms ("generators"), then CG often leads to simple updates, activating only few atoms at each iteration. In particular, in the case of the spectrahedron, the LO returns a matrix of rank one, which allows for sparsity preserving iterates. 
\end{enumerate}

The classical form of CG takes the answer obtained from querying the LO at a given gradient feedback $y=\nabla f(x)$, and returns the target vector 
\begin{equation}
p(x)=\scrL_{\setX}(\nabla f(x))\qquad\forall x\in\setX.
\end{equation}
It proposes then to move in the direction $p(x)-x$. As in every optimization routine, a key question is how to design efficient step-size rules to guarantee reasonable numerical performance. Letting $x^{k-1}$ and $p^{k}=p(x^{k-1})$ be a current position of the method together with its implied target vector, the following policies are standard choices: 
\begin{align}
\text{Standard:}&\quad \gamma_k=\frac{1}{2+k},\label{eq:standard}\\
\text{Exact line search:}&\quad  \gamma_k\in \argmin_{t\in(0,1]} f(x^{k-1}+t(p^k-x^{k-1})),\label{eq:LS} \\
\text{Adaptive:} &\quad \gamma_k=\min\left\{\frac{\inner{\nabla f(x^{k-1}),x^{k-1}-p^k}}{L_{f}\norm{x^{k-1}-p^k}^2},1\right\}.
\label{eq:adaptive}
\end{align}
Exact line search is conceptually attractive, but can be costly in large-scale applications when computing the function value is computationally expensive. To understand the construction of the adaptive step-size scheme, it is instructive to introduce a primal gap (merit) function to the problem, which is the fundamental performance measure of CG methods. The primal gap (merit) function is defined as
\begin{equation}\label{eq:CG-e}
\ce(x):=\sup_{u\in\setX}\inner{\nabla f(x),x-u}.
\end{equation}
This merit function is just the gap program (see e.g. \cite{FacPan03}) associated to the monotone variational inequality \eqref{eq:VI-f} in which the non-smooth part is trivial. In terms of this merit function, the celebrated descent lemma \eqref{eq:DLclassical} yields immediately 
\begin{align*}
f(x+t(p(x)-x))&\leq f(x)+t\inner{\nabla f(x),p(x)-x}+\frac{L_{f}t^{2}}{2}\norm{p(x)-x}^{2}\\
&=f(x)-t\ce(x)+\frac{L_{f}t^{2}}{2}\norm{p(x)-x}^{2}=f(x)-\eta_{x}(t),
\end{align*}
where $\eta_{x}(t):=t\ce(x)-\frac{L_{f}t^{2}}{2}\norm{p(x)-x}^{2}$. Optimizing this function with respect to $t\in[0,1]$ yields the largest-possible per-iteration decrease and returns the adaptive step-size rule in \eqref{eq:adaptive}. Once the optimizer decided upon the specific step-size policy, the classical CG picks one of the step sizes \eqref{eq:standard}, \eqref{eq:LS}, or \eqref{eq:adaptive}, and performs the update 
$$
x^{k}=x^{k-1}+\gamma_k(p(x^k)-x^{k-1}).
$$
\begin{algorithm}[h!]{\bf{The classical conditional gradient} (CG)}\\
	{\bf Input:} A linear oracle $\mathcal{L}_{\setX}$, a starting point $x^0\in\setX$.\\
	{\bf Output:} A solution $x$ such that $\Psi(x)-\Psi_{\min}(\setX)<\varepsilon$.\\
	{\bf General step:} For $k=1,2,\ldots$\\
	\qquad Compute $p^k=\mathcal{L}_X(\nabla f(x^{k-1}))$;\\
	\qquad Choose a step-size $\gamma_{k}$ either by \eqref{eq:standard}, \eqref{eq:LS}, \eqref{eq:adaptive};\\
	\qquad Update $x^k=x^{k-1}+\gamma_k(p^k-x^{k-1})$;\\
	\qquad Compute $\ce^k=\ce(x^{k-1})$.\\
	\qquad If $\ce^k<\varepsilon$ return $x^k$.
\end{algorithm}
The convergence properties of classical CG under either of the step-size variants above is well documented in the literature (see e.g. the recent text by \cite{lan2020first}, or \cite{Jag13}). We will obtain a full convergence and complexity theory under our more general analysis of the generalized CG scheme. 

\subsubsection{Relative smoothness}

The basic ingredient in proving convergence and complexity results on the classical CG is the fundamental inequality 
\[
f(x+t(p(x)-x))\leq f(x)-t\ce(x)+\frac{L_{f}t^{2}}{2}\norm{p(x)-x}^{2}.
\]
Based on the relative smoothness analysis in Section \ref{sec:NoLips}, it seems to be intuitively clear that we could easily prove also convergence of CG when instead of the restrictive Lipschitz gradient assumption we make a relative smoothness assumption in terms of the pair $(f,h)$ for some DGF $h\in\scrH_{\alpha}(\setX)$. Indeed, if we are able to estimate a scalar $L_f^h>0$ such that $L_f^hh(x)-f(x)$ is convex on $\setX$, then the modified descent lemma \eqref{eq:DLNoLips} yields the overestimation 
\begin{equation}
f(x+(tp-x))\leq f(x)-t\ce(x)+L_f^h D_{h}(x+t(p-x),x).
\end{equation}
Instead of requiring that $f$ has a Lipschitz continuous gradient over the convex compact set $\setX$, let us alternatively require the following:
\begin{assumption}\label{ass:alt_Lipschitz}
	There exists a DGF $h\in\scrH_{\alpha}(\setX)$ and a constant $L_f^h>0$, such that $L_f^h h-f$ is convex on $\setX$, and $h$ has a finite curvature on 
	$\setX$, that is,
	\begin{equation}\label{eq:FCh}
	\Omega^{2}_h(\setX)\coloneqq\sup_{x,u\in\setX, t\in[0,1]}  \frac{2 D_h(t u+(1-t)x,x)}{t^2}<\infty.
	\end{equation}
\end{assumption}
Note that when choosing $h$ to be the squared Euclidean norm $h(x)=\frac{1}{2}\norm{x}^2$ and $L_f^h=L_f$, then Assumption~\ref{ass:alt_Lipschitz} is equivalent to the Lipschitz gradient assumption, where $\Omega_{h}(\setX)$ is the diameter of set $\setX$. On the other hand, choosing $h(x)=f(x)$ and $L_f^h=L_{f}$, we essentially retrieve the finite curvature assumption used by Jaggi \cite{Jag13}. 

\begin{remark}
It is clear that the finite curvature assumption \eqref{eq:FCh} is not compatible with the DGF to be essentially smooth on $\setX$. We are therefore forced to work with non-steep distance-generating functions.
\end{remark}

The analysis of CG under a relative smoothness condition and Assumption \ref{ass:alt_Lipschitz} runs in the same way as for the classical CG. However, the adaptive step-size is reformulated as
\[
\gamma_k=\min\left\{\frac{\inner{\nabla f(x^{k-1}),x^{k-1}-p^k}}{L_f^h \Omega^{2}_h(\setX)},1\right\}.
\]
This can be easily seen by replacing the upper model function $f(x)-t\ce(x)+L_f^h D_{h}(x+t(p-x),x)$, with its more conservative bound $f(x)-t\ce(x)+\frac{L_f^h t^{2}}{2}\Omega^{2}_{h}(\setX)$.
Of course, in the case of the Euclidean norm this results in a smaller step-size than the adaptive step, which hints towards a deterioration of performance. Nevertheless, this trick allows us to handle convex programming problems outside the Lipschitz smooth case, which is not uncommon in various applications \cite{BiaChe15,BiaCheYe15,HaeLiuYe18}.

\subsection{Generalized Conditional Gradient}

Introduced by Bach \cite{Bac15} and \cite{nesterov2018complexity}, the generalized conditional gradient (GCG) method, is targeted to solve our master problem \eqref{eq:Opt} over a compact set $\setX$. 
 To handle the composite case, we need to modify our definition of a linear oracle accordingly. 
\begin{definition}
\label{Def:gen_lin_or}
	Operator $\scrL_{\setX,r}:\setV^{\ast}\rightarrow\setX$ is a \emph{generalized linear oracle} (GLO) over set $\setX$ with respect to function $r$ if for any vector $y\in\setV^{\ast}$ we have that
	\[
	\scrL_{\setX,r}(y)\in \argmin_{x\in\setX} \inner{y,x}+r(x).
	\]
\end{definition}
Besides this more demanding oracle assumption, the resulting generalized conditional gradient method is formally identical to the classical CG. In particular, we can consider the target vector 
\begin{equation}
p(x)=\scrL_{\setX,r}(\nabla f(x))\qquad\forall x\in\setX
\end{equation}
and the same three step size policies as in the classical CG, with the obvious modifications:
\begin{align}
\text{Exact line search: } \quad & \gamma_k\in \argmin_{t\in[0,1]} \Psi(x^{k-1}+t(p^k-x^{k-1})),\label{eq:GSC-LS} \\
\text{Adaptive: } \quad & \gamma_k=\min\left\{\frac{r(x^{k-1})-r(p^k)+\inner{\nabla f(x^{k-1}),x^{k-1}-p^k}}{L_f\norm{x^{k-1}-p^k}^2},1\right\}.
\label{eq:GSC-adaptive}
\end{align}
The adaptive step size variant is derived from an augmented merit function, taking into consideration the non-smooth composite nature of the underlying optimization problem. Indeed, as again can be learned from the basic theory of variational inequalities (see \cite{Nesterov:2007aa}), the natural merit function for the composite model problem \eqref{eq:Opt} is the non-smooth function 
\begin{equation}
\ce(x)=\sup_{u\in\setX}\Gamma(x,u),\text{ where }\Gamma(x,u):=r(x)-r(u)+\inner{\nabla f(x),x-u}.
\end{equation}
 By definition, we see that $\ce(x)\geq 0$ for all $x\in\setX$, with equality if and only if $x\in\setX^{\ast}.$ These basic properties justify our terminology, calling $\ce(x)$ a merit function. Of course, $\ce(\cdot)$ is also easily seen to be convex. Furthermore, using the convexity of $f$, one first sees that 
 $$
 \inner{\nabla f(x),x-u}\geq f(x)-f(u),
 $$
 so that for all $x,u\in\dom(r)$,
 \begin{align*}
 \Gamma(x,u)&\geq r(x)-r(u)+f(x)-f(u)=\Psi(x)-\Psi(u).
 \end{align*}
 From here, one immediately arrives at the relation 
 \begin{equation}\label{eq:LB-CG}
 \ce(x)\geq \Psi(x)-\Psi_{\min}(\setX).
 \end{equation}
 Clearly, with $r=0$, the above specification yields the classical CG. 
 
\subsubsection{Basic Complexity Properties of GCG}
We now turn to prove that the GCG method with one of the above mentioned step-sizes converges at a rate of $O(\frac{1}{k})$. We will derive this rate under the standard assumption Lipschitz smoothness assumption on $f$. This gives us access to the classical descent lemma \eqref{eq:DLclassical}. Combining this with the assumed convexity of the non-smooth function $r(\cdot)$, we readily obtain 
\begin{align*}
\Psi(x^{k-1}+t(p^{k}-x^{k-1}))&\leq f(x^{k-1})-t\inner{\nabla f(x^{k-1}),p^{k}-x^{k-1}}+\frac{t^{2}L_{f}}{2}\norm{p^{k}-x^{k-1}}^{2}+(1-t)r(x^{k-1})+t r(p^{k})\\
&=\Psi(x^{k-1})-t\ce(x^{k-1})+\frac{t^{2}L_{f}}{2}\norm{p^{k}-x^{k-1}}^{2}.
\end{align*}
Based on this fundamental inequality of the per-iteration decrease, we can deduce the iteration complexity via an induction argument. First, one observes that for each of the three introduced step-size rules (standard, line search and adaptive), one obtains a recursion of the form
\begin{align*}
\Psi(x^{k-1}+\gamma_{k}(p^{k}-x^{k-1}))\leq \Psi(x^{k-1})-\gamma_{k}\ce(x^{k-1})+\frac{L_{f}\gamma^{2}_{k}}{2}\norm{p^{k}-x^{k}}^{2}.
\end{align*}
When denoting $s^{k}:=\Psi(x^{k})-\Psi_{\min}(\setX)$, $\ce^{k}=\ce(x^{k-1})$ and $\Omega^{2}\equiv \Omega_{\frac{1}{2}\norm{\cdot}^{2}}^{2}(\setX)=\max_{x,u\in\setX}\norm{x-u}^{2}$, this gives us 
\[
s^{k}\leq s^{k-1}-\gamma_{k}\ce^{k}+\frac{L_{f}\gamma^{2}_{k}}{2}\Omega^{2}.
\]
Applying to this recursion Lemma 13.13 in \cite{Beck17}, we deduce the next iteration complexity result for GCG.
\begin{theorem}
Consider algorithm GCG with one of the step size rules: standard \eqref{eq:standard}, line search \eqref{eq:GSC-LS}, or adaptive \eqref{eq:GSC-adaptive}. Then
\begin{align*}
\Psi(x^{k})-\Psi_{\min}(\setX)\leq \frac{2\max\{\Psi(x^{0})-\Psi_{\min}(\setX), L_{f}\Omega^{2}\}}{k}\quad \forall k\geq 1.
\end{align*}
\end{theorem}

\begin{proof}
We give a self-contained proof of this result for the adaptive step-size policy \eqref{eq:GSC-adaptive}.

If $\gamma_{k}=1$, the per-iteration progress is easily seen to be
\begin{align*}
s^{k}&\leq s^{k-1}-\ce^{k}-\frac{L_{f}}{2}\Omega^{2}\leq s^{k-1}-\ce^{k}\leq s^{k-1}-\frac{1}{2}\ce^{k}
\leq s^{k-1}-\frac{1}{2}s^{k-1}=\frac{1}{2}s^{k-1}
\end{align*}
where we have used $\ce^{k}\geq 0$, and \eqref{eq:LB-CG}. For $\gamma_{k}=\frac{r(x^k)-r(p^k)+\inner{\nabla f(x^{k-1}),x^{k-1}-p^k}}{L_f\norm{x^{k-1}-p^k}^2}=\frac{\ce^{k}}{L_{f}\norm{p^{k}-x^{k-1}}^{2}}$, a simple computation reveals 
\begin{align*}
s^{k}\leq s^{k-1}-\frac{\ce^{k}}{2L_{f}\norm{p^{k}-x^{k-1}}^{2}}\leq s^{k-1}-\frac{(\ce^{k})^{2}}{2L_{f}\Omega^{2}}\leq s^{k-1}-\frac{(s^{k-1})^{2}}{2L_{f}\Omega^{2}}.
\end{align*}
Summarizing these two cases, we see 
\[
s^{k}\leq\max\left\{\frac{1}{2}s^{k-1},s^{k-1}-\frac{(s^{k-1})^{2}}{2L_{f}\Omega^{2}}\right\}.
\]
Thus, the convergence is split into two periods, which are defined by $K:=\log_2\left(\lfloor \frac{s^0}{\min\{L_f\Omega^{2},s^0\}} \rfloor\right)+1$.
If $k\leq K$ then $s^{k-1}\geq L_f\Omega^{2}$ and thus $s^{k}\leq \frac{1}{2}s^{k-1}$, which implies 
\begin{align*}
s^k\leq 2^{-k}s^{0},\;k\in\{0,1,\ldots,K\}.
\end{align*} 
However, if $k> K$ then $s^{k-1}< \min\{L_f \Omega^{2},s^0\}$ and $s^k\leq s^{k-1}-\frac{1}{2L_f\Omega^{2}}(s^{k-1})^2$, which by induction (see for example  \cite[Lemma 5.1]{dunn1979rates}) implies that
\begin{align*}
s^{k}\leq  \frac{s^K}{1+\frac{s^K}{2L_f \Omega^{2}} (k-K)}\leq \frac{2L_f \Omega^{2}}{2+ (k-K)} \leq \frac{\max\{K,2\}L_f \Omega^{2}}{k}\leq \frac{2\max\{s^0, L_{f}\Omega^{2}\}}{k},\; k\geq K+1,
\end{align*}
where the second inequality follows from $s^K< \min\{L_f \Omega^{2},s^0\}$, the third inequality follows from $\frac{a}{a+(k-K)}$ being a monotonic function in $a\geq 0$ for any  $k\geq K+1$, and the last inequality follows from $K\leq \max\left\{2,\frac{s^0}{L_f \Omega^{2}}\right\}$.
Combining these two results, we have that
\begin{align*}
s^{k}\leq  \frac{2\max\{s^0,L_f \Omega^{2}\}}{k}.
\end{align*}

\end{proof}

\subsubsection{Alternative assumptions and step-sizes}\label{sec:assumption_stepsize}
A key takeaway from the analysis of the generalized conditional gradient is that one needs to have a bound on the quadratic term of the upper model
$$
t\mapsto Q(x,p(x),t,L_{f}):=\Psi(x)-t\ce(x)+\frac{L_{f}t^{2}}{2}\norm{p(x)-x}^{2}.
$$
Such a bound was given to us essentially for free under the compactness assumption of the domain $\setX$, and the Lipschitz-smoothness assumption on the smooth part $f$. The resulting complexity constant is then determined by $L_{f}\Omega^{2}$.  Moreover, this constant will be involved in lower bounds of the adaptive step-size rule \eqref{eq:GSC-adaptive}. However, such a constant may not be known, or may be expensive to compute. Moreover, a global estimate of this constant is not actually needed for obtaining an upper bound. To see this, we proceed formally as follows. Consider an alternative quadratic function of the form 
\begin{align*}
Q(x,p,t,M):=\Psi(x)-t \ce(x) + \frac{t^{2}M}{2}q(p,x),
\end{align*}
where $q(p,x)$ is a positive function bounded by some constant $C$, and choose $\gamma(x,M):=\min\{1,\frac{\ce(x)}{M q(p(x),x)}\}$, for $p(x)=\scrL_{\setX,r}(\nabla f(x))$. Let $M>0$ be a constant such that the point obtained by using this step-size
is upper bounded by the corresponding quadratic function, {\it i.e.},
\begin{align}\label{eq:CG_key_ineq}
\Psi\left((1-\gamma(x,M))x+\gamma(x,M)p(x)\right)\leq Q(x,p(x),\gamma(x,M),M)< \Psi(x).
\end{align}
Thus applying the update $x^{+}:=(1-\gamma(x,M))x+\gamma(x,M)p(x)$, we obtain 
$$
\Psi(x^{+})-\Psi_{\min}(\setX)\leq \Psi(x)-\Psi_{\min}(\setX)-\frac{1}{2}\ce(x)\leq \frac{1}{2}(\Psi(x)-\Psi_{\min}(\setX))
$$
if $\gamma(x,M)=1$, and 
\begin{align*}
\Psi(x^{+})-\Psi_{\min}(\setX)&\leq \Psi(x)-\Psi_{\min}(\setX)-\frac{1}{2M q(p(x),x)}\ce(x)^{2}\leq \Psi(x)-\Psi_{\min}(\setX)-\frac{1}{2M C}(\Psi(x)-\Psi_{\min}(\setX))^{2}
\end{align*}
if $\gamma(x,M)=\frac{\ce(x)}{M q(p(x),x)}$. If $(x^{k})_{k\geq 0}$ is the trajectory defined in this specific way, we get the familiar recursion
\begin{align*}
s^{k}\leq \min\{\frac{1}{2}s^{k-1},s^{k-1}-\frac{1}{2 M_{k}C}(s^{k-1})^2\}
\end{align*}
in terms of the approximation error $s^{k}:=\Psi(x^{k})-\Psi_{\min}(\setX)$, and the local estimates $(M_{k})_{k\geq 0}$. Thus, as we are able to bound $M_{k}$ from above for all iterations of the algorithm, the same convergence as for GCG can be achieved.

Based on this observation, and knowing that $M_k$ must be bounded for Lipschitz smooth objective functions, we can try to determine $M_k$ via a backtracking procedure, as suggested in \cite{pedregosa2020linearly}. By construction, the resulting iterates $x^{k}$ will induce monotonically decreasing function values so that the whole trajectory $x^{k}$ will be contained in the level set $\{x\in\setX\vert\Psi(x)\leq\Psi(x^{0})\}$. Hence, it is sufficient for $Q(x^k,p^k,t,M)$ to be an upper bound on $\Psi(x_{t})$ for any point $x_{t}=(1-t)x^{k-1}+t p^k$ such that $\Psi(x_{t})\leq \Psi(x^0)$. Thus, the Lipschitz continuity (or curvature) can be assumed only on the appropriate level set and there is no need to insist on global Lipschitz smoothness on the entire set $\setX$. This insight enabled, for example, proving the $O(1/k)$ convergence rate of CG with adaptive and exact step-size rules when applied to self-concordant functions, which are not necessarily Lipschitz smooth on the predefined set $\setX$ \cite{dvurechensky20a,dvurechensky2020generalized}. However, this observation need not apply to the standard step size rule \eqref{eq:standard}, since the standard step-size choice does not guarantee that all the iterates remain in the appropriate level set.

To conclude, we reiterate that the step-size choices analyzed here are the most common, but there may be many more choices of step-size which provide similar guarantees. For example, \cite{freund2016new} suggests new step-size rules based on an alternative analysis of the CG method that utilizes an updated duality gap. \cite{nesterov2018complexity} discusses recursive step-size rules, and in \cite{odor2016frank,dvurechensky20a} new step-size rules are suggested based on additional assumptions on the problem structure.

\subsection{Variants of CG}
One of the main drawbacks of CG method is that, in general, it comes with worse complexity bounds than BPGM for strongly convex functions. Indeed, it was shown as early as in 1968 by Cannon and Cullum \cite{canon1968tight} (see also \cite{Lan13, lan2020first}) that the rate of $O(\frac{1}{k})$ is in fact tight, even when the function $f$ is strongly convex.  This slow convergence is due to the well-documented zig-zagging effect between different extreme points in $\setX$. In the smooth case, where $r=0$, and the objective function $f$ and the feasible set $\setX$ are both strongly convex, only a rate of $O(\frac{1}{k^2})$ can be shown \cite{garber2015faster}, whereas \cite{nesterov2018complexity} showed an accelerated $O(\frac{1}{k^2})$ rate of convergence for GCG with strongly convex $r$ ($\mu>0$). Linear convergence of the CG method can only be proved under additional assumptions regarding the problem structure or location of the optimal solution (see e.g. \cite{levitin1966constrained,dunn1979rates,guelat1986some, epelman2000condition, beck2004conditional}).

Departing from these somewhat negative results, variants of the classical CG were suggested in order to obtain the desired linear convergence in the case of strongly convex function $f$. We will discuss four of these variants: Away-step CG, Fully-corrective CG, CG based on a local linear optimization oracle (LLOO), and CG with sliding. 

\subsubsection{Away-step CG}
The away-step variation of CG (AW-CG), first suggested by Wolfe \cite{wolfe1970}, treats the case where $\setX$ is a polyhedron. It requires two calls of the LO at each iteration. The first call generates $p^k=\mathcal{L}_\setX(\nabla f(x^k))$, defined in the original CG algorithm, while the second call generates an additional vector $u^k=\mathcal{L}_\setX(-\nabla f(x^k))$. The two vectors $p^k$ and $u^k$ define the \emph{forward direction} $d^{k}_{FW}=p^k-x^{k-1}$ and the \emph{away direction} $d^{k}_{A}=x^{k-1}-u^k$, respectively. By construction, both of this directions are descent directions. The effectively chosen direction at iteration $k$ is obtained by
\begin{align*}
d^k=\argmax_{d\in \{d^{k}_{FW},d^{k}_{A}\}} \inner{-\nabla{f}(x^k),d},
\end{align*}
with a corresponding updating step
\begin{align*} 
x^k=x^{k-1}+\gamma_{k}d^k.
\end{align*}  
Here, the choice of the step-size $\eta_k$ will also depend on the direction chosen.
The first analysis of this algorithm by Gu{\'e}lat and Marcotte \cite{guelat1986some} assumes that the step-size is chosen using exact line search over $\gamma_k\in [0,\gamma_{\max}]$, where $\gamma_{\max}:=\max\{t\geq 0: x^{k-1}+t d^k\in \setX\}$.  Under this step-size choice, they prove linear convergence of CG for strongly convex $f$. However, this rate estimate  depends on the distance between the optimal solution and the boundary of set $T\subset\setX$, which is the minimal face of $\setX$ containing the optimal solution. This result was later extended in \cite{lacoste2015global}, with a slight variation on the original algorithm. In this variation, the set $\setX$ is represented as the convex hull of a finite set of atoms $\mathcal{A}$ (not necessarily containing only its vertices), and a representation of the current iterate as a convex combination of these atoms is maintained throughout the algorithm, {\it i.e.}, $x^k=\sum_{S^k} \lambda^k_a a$ where $S^k=\{a\in \mathcal{A}:\lambda^k_a>0\}$ is defined as the set of active atoms. Thus, the AW-CG produces $p^k\in\mathcal{A}$ and $u^k\in S^k$, and the away step maximal step size is respecified as $\gamma_{\max}=\frac{\lambda_{u^k}}{1-\lambda_{u^k}}$. This implies, that using the maximal away-step step-size will not necessarily result on a point on the boundary of $\setX$. Thus, when $f$ is strongly convex, Jaggi and Lacoste-Julian \cite{lacoste2015global} show a linear convergence of AW-CG with a rate which only depends on the geometry of set $\setX$, which is captured by the \emph{pyramidal width} parameter. 
The Pairwise variant of AW-CG, which is also presented and analyzed in \cite{lacoste2015global}, takes $d^k=u^k-p^k$ and $\gamma_{\max}=\lambda_{u^k}$, and has similar analysis.

In \cite{beck2017linearly}, Beck and Shtern extend the linear convergence results of AS-CG to functions of the form $f(x)=g(\bA x)+\inner{b,x}$ where $g$ is a strongly convex function. The linear rate depends on a parameter based on the Hoffman constant, which captures both on the geometry of $\setX$ as well as matrix $\bA$. It is also worth mentioning, a stream of work which shows linear convergence of AS-CG where the strong convexity assumption is replaced by the assumption that sufficient second order optimality conditions, known as Robinson conditions \cite{robinson1982generalized}, are satisfied (see for example \cite{damla2008linear}). 

\subsubsection{Fully-corrective CG}
The Fully-corrective variant of CG (FC-CG) also involves polyhedral $\setX$, and aims to reduce the number of calls to the linear oracle, by replacing them with a more accurate minimization over a convex-hull of some subset $\mathcal{A}^k\subseteq\mathcal{A}$. The heart of the method is a correction routine, which updates the correction atoms $\mathcal{A}^{k}$ and iterate $x^{k}$, and satisfy the following:
\begin{align*}
S^k&\subseteq\mathcal{A}^k\\
f(x^k)&\leq \min_{t\in [0,1]} f((1-t)x^{k-1}+t p^k)\\
\epsilon &\geq \max_{s\in S^k} \inner{\nabla f(x^k),s-x^k}
\end{align*} 
where $p^k=\mathcal{L}_\setX(\nabla f(x^{k-1}))$, and $\epsilon$ is a given accuracy parameter.
The FC-CG was known by various names depending on the updating scheme of $\mathcal{A}^k$ and $x^k$ \cite{holloway1974extension,von1977simplicial}, and was unified and analyzed to show linear convergence in \cite{lacoste2015global}. The convergence analysis of FC-CG is similar to that of AW-CG, and is based on the correction routine guaranteeing that the forward step is larger than the away-step computed in the previous iteration. 

In order to apply FC-CG one must choose a correction routine, and the linear convergence analysis does not take into account the computational cost of this routine. One choice of a correction routine is to apply AS-CG on the subset $\mathcal{A}^k=S^{k-1}\cup\{p^k\}$ until the conditions are satisfied. This correction routine is wise only if efficient linear oracles $\mathcal{L}_{\mathcal{A}^k}$ can be constructed for all $k$ such that their low computational cost balances the routine's iteration complexity.

\subsubsection{Enhanced LO based CG}
A variant of CG which is based on an enhanced linear minimization oracle, was suggested by Garber and Hazan \cite{garber2016LLOO}.
In this variant, the linear oracle $\mathcal{L}_{\setX}(c)$ is replaced by a \emph{local oracle} $\mathcal{L}_{\setX,\rho}(c,x,\delta)$ with some constant $\rho\geq 1$, which takes an additional radius input $\delta$ and returns a point $p\in\setX$ satisfying
\begin{align*}
\norm{p-x}&\leq \rho \delta \\
\inner{p,y}&\leq \min_{u\in\setX:\norm{u-x}\leq \delta}\inner{u,y}. 
\end{align*} 
Thus, the only deviation from the CG algorithm is that $p^k$ is obtained by applying $\mathcal{L}_{X,\rho}(\nabla f(x^k),x^k,\delta_k)$ for a suitably chosen sequence $(\delta_{k})_{k}$. The linear convergence for the case where the smooth part $f$ is strongly convex, is obtained by a specific update of $\delta_k$ at each step of the algorithm. This update depends on the Lipschitz constant $L_f$, the strong convexity constant of $f$, and the parameter $\rho$. Moreover, despite the fact that LLOO-CG can theoretically be applied to any set $\setX$, constructing a general LLOO is challenging. In \cite{garber2016LLOO}, the authors suggest an LLOO with $\rho=\sqrt{n}$ when the set $\setX$ is the unit simplex, and generalize it for convex polytopes with $\rho=\sqrt{n}\tilde{\rho}$ where $\tilde{\rho}$ depends on some geometric properties the polytope which may generally not tractably computed. Thus, while the strong convexity and geometric properties of the problem are only used for the analysis of the AW-CG and FC-CG, the associated parameters are explicitly used in the execution of LLOO-CG. The difficulty of accurately estimating the strong convexity and the geometric parameters renders the LLOO-CG less applicable in practice.   

\subsubsection{CG with gradient sliding}
\label{S:S-CG}
Each iteration of CG requires one call to the linear minimization oracle and one gradient evaluation. Coupled with our knowledge about the iteration complexity of CG, this fact implies that CG requires $O(1/\eps)$ gradient evaluations of the objective function. This is suboptimal, when compared with the $O(1/\sqrt{\eps})$ gradient evaluations for smooth convex optimization, as we will see in Section \ref{sec:accelerated}. While it is known that within the linear minimization oracle, the order estimate $O(1/\eps)$ for the number of calls of the LO is unimprovable, in this section we review a method based on the linear minimization oracle which can skip the computation of gradients from time to time. This improves the complexity of LO-based methods and leads us to the \emph{conditional gradient sliding} (S-CG) algorithm introduced by Lan and Zhou \cite{lan2016conditional}. S-CG is a numerical optimization method which runs in epochs and overall contains some similarities with accelerated methods, to be thoroughly surveyed in Section \ref{sec:accelerated}. S-CG has been described in the context of the smooth convex programming problem for which $r=0$.  

\begin{algorithm}{\bf{The conditional gradient sliding methods} (S-CG)}\\
	{\bf Input:} A linear oracle $\mathcal{L}_{\setX}$ a starting point $x^0\in\setX$.\\
	\qquad $(\beta_{k})_{k},(\gamma_{k})_{k}$ parameter sequence such that 
	\begin{align*}
	\gamma_{1}=1,\;L_{f}\gamma_{k}\leq \beta_{k},\\
	\frac{\beta_{k}\gamma_{k}}{\Gamma_{k}}\geq \frac{\beta_{k-1}\gamma_{k-1}}{\Gamma_{k-1}},
	\end{align*}
	where 
	\begin{equation}
	\Gamma_{k}=\left\{\begin{array}{ll}
	 1 & \text{if }k=1,\\
	 \Gamma_{k-1}(1-\gamma_{k}) & \text{if }k\geq 2.
	 \end{array}
	 \right.
	 \end{equation}
	 
	{\bf General step:} For $k=1,2,\ldots$\\
	\qquad Compute 
	\begin{align*}
	z^{k}&=(1-\gamma_{k})y^{k-1}+\gamma_{k}x^{k-1},\\
	x^{k}&=\text{CndG}(\nabla f(z^{k}),x^{k-1},\beta_{k},\eta_{k}),\\
	y^{k}&=(1-\gamma_{k})y^{k-1}+\gamma_{k}x^{k}.
	\end{align*}
\end{algorithm}
Similarly to accelerated methods, S-CG keeps track of three sequentially updated sequences. The update of the sequence $(x^{k})$ is stated in terms of a procedure $\text{CndG}$, which describes an inner loop of conditional gradient steps. This subroutine aims at approximately solving for the proximal step 
\begin{align*}
\min_{x\in X} f(z^{k})+\inner{\nabla f(z^{k}),x-z^k}+\frac{\beta_k}{2}\norm{x-x^{k-1}}^2
\end{align*}
up to an accuracy of $\eta_k$. As will become clear later, the S-CG can thus be thought of as an approximate version of the accelerated scheme presented in Section~\ref{sec:acc_grad}.

\begin{algorithm}{\bf{The procedure} $\text{CndG}(g,u,\beta,\eta)$}\\
	{\bf Input: } $u_{1}=u,t=1$.\\
	{\bf Output: } point $u^{+}=\text{CndG}(g,u,\beta,\eta).$ \\
	{\bf General step:} Let $v_{t}=\argmax_{x\in\setX}\inner{g+\beta(u_{t}-u),u_{t}-x}$\\
	\qquad If $V_{g,u,\beta}(u_{t})=\inner{g+\beta(u_{t}-u),u_{t}-v_{t}}\leq \eta$, set $u^{+}=u_{t}$; \\
	\qquad else, set $u_{t+1}=(1-\alpha_{t})u_{t}+\alpha_{t}v_{t}$, where 
	$$
	\alpha_{t}=\min\left\{1,\frac{\inner{\beta(u-u_{t})-g,v_{t}-u_{t}}}{\beta\norm{v_{t}-u_{t}}^{2}}\right\}.
	$$
	\qquad Set $t\leftarrow t+1$. Repeat General step. 
\end{algorithm}

The main performance guarantee of the algorithm S-CG is summarized in the following theorem:
\begin{theorem}
For all $k\geq 1$ and $u\in\setX$, we have 
\begin{equation}
f(y^{k})-f(u)\leq \frac{\beta\gamma_{k} \Omega^{2}}{2}+\Gamma_{k}\sum_{i=1}^{k}\frac{\eta_{i}\gamma_{i}}{\Gamma_{i}},
\end{equation}
where $\Omega\equiv\Omega_{\frac{1}{2}\norm{\cdot}}(\setX)$. The number of calls of the linear minimization oracle is bounded by $\lceil\frac{6 \beta_{k}\Omega^{2}}{\eta_{k}}\rceil$. In particular, if the parameter sequences in S-CG are chosen as
$$
\beta_{k}=\frac{3L_{f}}{k+1},\gamma_{k}=\frac{3}{k+2},\eta_{k}=\frac{L_{f}\Omega^{2}}{k(k+1)},
$$
then 
$$
f(y^{k})-f(u)\leq \frac{15L_{f}\Omega^{2}}{2(k+1)(k+2)}.
$$
As a consequence, the total number of calls of the function gradients and the LO oracle is bounded by $O\left(\sqrt{\frac{L_{f}\Omega^{2}}{\eps}}\right)$, and $O(L_{f}\Omega^{2}/\eps)$, respectively. 
\end{theorem}

\section{Accelerated Methods}
\label{sec:accelerated}
%

In previous sections we focused on simple first-order methods with sublinear convergence guarantees in the convex case, and linear convergence in the strongly convex case. Towards the end of the discussion in Section \ref{sec:MD}, we pointed out the possibility to accelerate simple iterative schemes via suitably defined extrapolation steps. In this last section of the survey, we are focusing on such \emph{accelerated methods}. The idea of acceleration dates back to 1980's. The rationale for this research direction is the desire to understand the computational boundaries of solving optimization problems. Of particular interest has been the unconstrained smooth, and strongly convex optimization problem. This would be covered by our generic model \eqref{eq:Opt} by setting $r=0,\setX=\setV=\Rn$ and $f$ strongly convex with parameter $\mu_{f}>0$ and $L_{f}$-smooth. The standard approach to quantify the computational hardness of optimization problems is through the oracle model. Upon receiving a query point $x$, the oracle reports the corresponding function value $f(x)$, and in first-order models, the function gradient $\nabla f(x)$ as well. In their seminal work, Nemirovski and Yudin \cite{NY83} showed that for any first-oder optimization algorithm, there exists an $L_f$-smooth (with some $L_f>0$) and convex function $f:\Rn\to\R$ such that the number of queries required to obtain an $\eps$-optimal solution $x^{\ast}$ which satisfies 
$$
f(x^{\ast})<\min_{x}f(x) +\eps,
$$
is at least of the order of $\min\{n,\sqrt{L_f/\mu_{f}}\}\ln(1/\eps)$ if $\mu_{f}>0$ and $\min\{n\ln(1/\eps),\sqrt{L_f/\eps}\}$, if $\mu_{f}=0$. This bound, obtained by information-theoretical arguments, turned out to be tight. 
Nemirovski \cite{nemirovski1982orth} proposed a method achieving the optimal rate $O(1/k^2)$ via a combination of standard gradient steps with the classical center of gravity method, which required additional small-dimensional minimization, see also a recent paper \cite{nesterov2020primal-dual}. 
Nesterov \cite{Nes83} proposed an optimal method with explicit step-sizes, which is now known as Nesterov's accelerated gradient method. Mainly driven by applications in imaging and machine learning, the idea of acceleration turned out to be very productive in the last 20 years. During this time span it has been extended to composite optimization \cite{Nes13,BecTeb09}, general proximal setups \cite{Nes05,Nes18}, stochastic optimization problems \cite{devolder2011stochastic,lan2012optimal,ghadimi2012optimal,ghadimi2013optimal,dvurechensky2016stochastic,gasnikov2016stochasticInter,dvurechensky2018decentralize}, optimization with inexact oracle \cite{aspremont2008smooth,devolder2014first,dvurechensky2016stochastic,gasnikov2016stochasticInter,cohen2018acceleration,gasnikov2019fast,kamzolov2020universal,stonyakin2020inexact}, variance reduction methods \cite{frostig2015un-regularizing,lin2015universal,zhang2015stochastic,allen2016katyusha,lan2017optimal,ivanova2020oracle}, alternating minimization methods \cite{diakonikolas2018alternating,guminov2019acceleratedAM}, random coordinate descent \cite{nesterov2012efficiency,lee2013efficient,fercoq2015accelerated,lin2014accelerated,nesterov2017efficiency,gasnikov2016accelerated,allen2016even,shalev-shwartz2014accelerated,dvurechensky2017randomized,diakonikolas2018alternating} and other randomized methods such as randomized derivative-free methods \cite{nesterov2017random,dvurechensky2017randomized,gorbunov2018accelerated,vorontsova2019accelerated2} and randomized directional search \cite{dvurechensky2017randomized,vorontsova2019accelerated1,dvurechensky2020accelerated}, second-order methods \cite{Nes08} and even high-order methods \cite{Bae09,nesterov2019implementable,gasnikov2019near}. 

\subsection{Accelerated Gradient Method}\label{sec:acc_grad}
In this section we consider one of the multiple variants of an Accelerated Gradient Method. This variant is close to the accelerated proximal method in \cite{tseng2008accelerated}, which has been very influential to the field. Another very influential version of the accelerated method, especially in applications, is the FISTA algorithm \cite{BecTeb09}, which is excellently described in \cite{Beck17}. The version we present here is inspired by the Method of Similar Triangles \cite{gasnikov2018universal,Nes18} and is obtained via the change of the Dual Averaging step (see Section \ref{sec:DA}) to the Bregman Proximal Gradient step. 
In our presentation of the accelerated method, we consider a particular choice of the the control sequences, {\it i.e.}, numerical sequences $\alpha_{k}$, $A_{k}$ from \cite{dvurechensky2020stable,dvurechensky2018computational}. A more general way of constructing such sequences can be found in \cite{lan2020first}, see also the constants used in the S-CG method described at the end of Section~\ref{sec:CG}. Moreover, the version we present here, is very flexible and allows one to obtain accelerated methods for many settings. As a particular example, below in Section \ref{sec:Universal}, we show how a slight modification of this method allows one to obtain universal accelerated gradient method.

Our aim is to solve the composite model problem \eqref{eq:Opt} within a general Bregman proximal setup, formulated in Section \ref{S:Bregman_setup}.  Let $\setX\subseteq\setV$ be a closed convex set in a finite-dimensional real vector space $\setV$ with primal-dual pairing $\inner{\cdot,\cdot}$ and general norm $\norm{\cdot}$. We are given a DGF $h \in \scrH_{1}(\setX)$. The scaling of the strong convexity parameter to the value 1 actually is without loss of generality, modulo a constant rescaling of the employed DGF. Recall the Bregman divergence $D_{h}(u,x)=h(u)-h(x)-\inner{\nabla h(x),u-x}\geq \frac{1}{2}\norm{u-x}^{2}$ for all $x\in\setX^{\circ},u\in\setX$
\begin{algorithm}{\bf{The Accelerated Bregman Proximal Gradient Method} (A-BPGM)}
\\
	{\bf Input:} pick $x^{0}=u^{0}=y^{0}\in\dom(r)\cap\setX^{\circ}$, set $A_0=0$\\
	{\bf General step:} For $k=0,1,\ldots$ do:\\
	\qquad Find $\alpha_{k+1}$ from quadratic equation $A_k+\alpha_{k+1}=L_{f}\alpha_{k+1}^2$. Set $A_{k+1}=A_k+\alpha_{k+1}$.\\
	\qquad Set $y^{k+1} = \frac{\alpha_{k+1}}{A_{k+1}}u^{k}+\frac{A_{k}}{A_{k+1}}x^{k}$.\\
	\qquad Set \vspace{-28pt}
	\begin{align*}
	\hspace{-28pt} u^{k+1}&= \scrP^{h}_{\alpha_{k+1} r}
		(u^{k},\alpha_{k+1}\nabla f(y^{k+1}))\\
	\hspace{-28pt}&=\argmin_{x\in\setX}\left\{\alpha_{k+1}\left( f(y^{k+1}) + \inner{\nabla f(y^{k+1}),x-y^{k+1}} +r(x) \right) + D_{h}(x,u^{k}) \right\}.\end{align*}
	\qquad Set $x^{k+1} = \frac{\alpha_{k+1}}{A_{k+1}}u^{k+1}+\frac{A_{k}}{A_{k+1}}x^{k}$.\\
\end{algorithm}
\begin{figure}
\centering
\begin{tikzpicture}[scale=0.8]
	\filldraw (0,0) circle (1pt);
    \draw[-] (0,0) node[left] {$x^{k}$} -- (-3,6) node[left] {$u^{k}$};
	\draw (-2,4) node[left] {$\frac{A_{k}}{A_{k+1}}$};
	\filldraw (-1,2) node[left] {$y^{k+1}$} circle (1pt);
	\draw (-0.5,1) node[left] {$\frac{\alpha_{k+1}}{A_{k+1}}$};
	\filldraw (-3,6) circle (1pt);
	\draw[->] (-3,6) -- (6,6) node[above] {$u^{k+1}=u^{k}-\alpha_{k+1}\nabla f(y^{k+1})$};
	\filldraw (6,6) circle (1pt);
	\filldraw (2,2) node[right] {$x^{k+1}= y^{k+1}-\frac{1}{L_f}\nabla f(y^{k+1})$} circle (1pt);
	\draw[-] (6,6) -- (0,0);
	\draw[->] (-1,2) -- (2,2);
	%
%
%
		%
		%
		
    \end{tikzpicture}
    \caption{Illustration of the three sequences of the A-BPGM in the unconstrained case $\setX=\R^n$, $r=0$, $h=\frac{1}{2}\norm{x}_2^2$. In this simple case it is easy to see that $u^{k+1}=u^{k}-\alpha_{k+1}\nabla f(y^{k+1})$, and the sequence $u^k$ accumulates the previous gradient, while helping to keep momentum. Also by the similarity of the triangles, $x^{k+1}=y^{k+1}-\alpha_{k+1}\nabla f(y^{k+1})\cdot \frac{\alpha_{k+1}}{A_{k+1}} = y^{k+1}-\frac{1}{L_f}\nabla f(y^{k+1})$, i.e. $y^k$ is the sequence obtained by gradient descent steps. Finally, the sequence $x^{k}$ is a convex combination of the momentum step and the gradient step. The illustration is inspired by personal communication with Yu. Nesterov on the Method of Similar Triangles \cite{gasnikov2018universal,Nes18}.}
\end{figure}
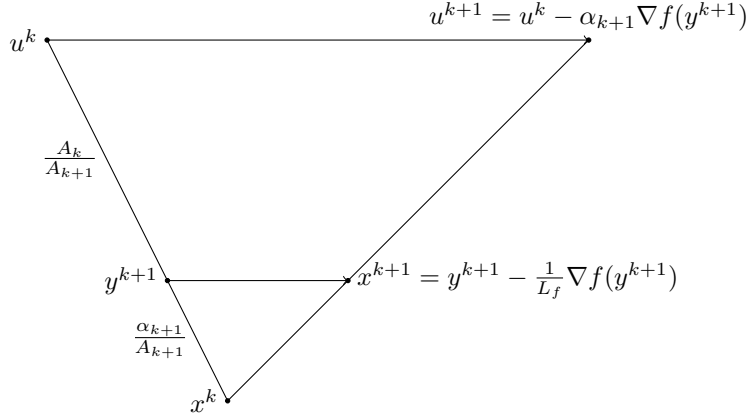

We start the analysis applying the descent Lemma property \eqref{eq:DLclassical} which holds for any two points due to $L_f$-smoothness:
\begin{align}
\label{eq:AGD_Proof_1}
\Psi(x^{k+1})=f(x^{k+1})+r(x^{k+1}) \leq f(y^{k+1})+\inner{\nabla f(y^{k+1}),x^{k+1}-y^{k+1}}+\frac{L_{f}}{2}\norm{x^{k+1}-y^{k+1}}^{2}+r(x^{k+1}).
\end{align}
Let us next consider the squared norm term. Using the definition of $x^{k+1},y^{k+1}$ and the quadratic equation for $\alpha_{k+1}$, as well as strong convexity of the Bregman divergence, i.e. \eqref{eq:Dlower}, we obtain
\begin{align}
\label{eq:AGD_Proof_2}
\frac{L_{f}}{2}\norm{x^{k+1}-y^{k+1}}^{2}&=\frac{L_{f}}{2}\norm{\frac{\alpha_{k+1}}{A_{k+1}}u^{k+1}+\frac{A_{k}}{A_{k+1}}x^{k}-\left(\frac{\alpha_{k+1}}{A_{k+1}}u^{k}+\frac{A_{k}}{A_{k+1}}x^{k}\right)}^{2} \notag \\
&= \frac{L_{f}\alpha_{k+1}^2}{2A_{k+1}^2} \norm{u^{k+1}-u^{k}}^{2} = \frac{1}{2A_{k+1}} \norm{u^{k+1}-u^{k}}^{2} \leq \frac{1}{A_{k+1}} D_h( u^{k+1},u^{k}).
\end{align}
Next, we consider the remaining terms in the r.h.s. of \eqref{eq:AGD_Proof_1}. Substituting $x^{k+1}$ and using $A_{k+1}=A_k+\alpha_{k+1}$, we obtain
\begin{align}
\label{eq:AGD_Proof_3}
f(y^{k+1})&+\inner{\nabla f(y^{k+1}),x^{k+1}-y^{k+1}}+r(x^{k+1}) \notag \\
&= \left(\frac{\alpha_{k+1}}{A_{k+1}}+ \frac{A_{k}}{A_{k+1}}\right)f(y^{k+1}) 
+\inner{\nabla f(y^{k+1}),\frac{\alpha_{k+1}}{A_{k+1}}u^{k+1}+\frac{A_{k}}{A_{k+1}}x^{k}-\left(\frac{\alpha_{k+1}}{A_{k+1}}+ \frac{A_{k}}{A_{k+1}}\right)y^{k+1}}\\
& + r\left( \frac{\alpha_{k+1}}{A_{k+1}}u^{k+1}+\frac{A_{k}}{A_{k+1}}x^{k}\right) \notag\\
&\leq \frac{A_{k}}{A_{k+1}}\left(f(y^{k+1})+\inner{\nabla f(y^{k+1}),x^{k}-y^{k+1}} + r(x^{k})\right) \\
&+ \frac{\alpha_{k+1}}{A_{k+1}}\left(f(y^{k+1})+\inner{\nabla f(y^{k+1}),u^{k+1}-y^{k+1}} + r(u^{k+1})\right) \notag \\
& \leq \frac{A_{k}}{A_{k+1}} \left(f(x^{k})+ r(x^{k})\right)+ \frac{\alpha_{k+1}}{A_{k+1}}\left(f(y^{k+1})+\inner{\nabla f(y^{k+1}),u^{k+1}-y^{k+1}} + r(u^{k+1})\right) \notag \\
& = \frac{A_{k}}{A_{k+1}} \Psi(x^{k})+ \frac{\alpha_{k+1}}{A_{k+1}}\left(f(y^{k+1})+\inner{\nabla f(y^{k+1}),u^{k+1}-y^{k+1}} + r(u^{k+1})\right),
\end{align}
where in the first inequality used the convexity of $r$, and in the second inequality we used the convexity of $f$. Now we plug \eqref{eq:AGD_Proof_2} and \eqref{eq:AGD_Proof_3} into \eqref{eq:AGD_Proof_1} to obtain
\begin{align}
\label{eq:AGD_Proof_4}
\Psi(x^{k+1}) &\leq \frac{A_{k}}{A_{k+1}} \Psi(x^{k}) + \frac{\alpha_{k+1}}{A_{k+1}} \left(f(y^{k+1})+\inner{\nabla f(y^{k+1}),u^{k+1}-y^{k+1}} + r(u^{k+1})\right) + \frac{1}{A_{k+1}} D_h( u^{k+1},u^{k})\notag \\
& = \frac{A_{k}}{A_{k+1}} \Psi(x^{k}) + \frac{1}{A_{k+1}} \left[\alpha_{k+1}\left(f(y^{k+1})+\inner{\nabla f(y^{k+1}),u^{k+1}-y^{k+1}} + r(u^{k+1})\right) + D_h( u^{k+1},u^{k}) \right].
\end{align}
Given the definition of $u^{k+1}$ as a Prox-Mapping, we can apply \eqref{eq:r} by substituting $x^{+}=u^{k+1}$, $x=u^{k}$, $\gamma = \alpha_{k+1}$. In this way, we obtain, for any $u \in \setX$,
\begin{align}
\label{eq:AGD_Proof_5}
\Psi(x^{k+1}) &\leq \frac{A_{k}}{A_{k+1}} \Psi(x^{k}) + \frac{1}{A_{k+1}} \left(\alpha_{k+1}(f(y^{k+1})+\inner{\nabla f(y^{k+1}),u^{k+1}-y^{k+1}} + r(u^{k+1})) + D_h( u^{k+1},u^{k}) \right) \notag \\
& \stackrel{\eqref{eq:r}}{\leq}  \frac{A_{k}}{A_{k+1}} \Psi(x^{k}) + \frac{1}{A_{k+1}} \left(\alpha_{k+1}(f(y^{k+1})+\inner{\nabla f(y^{k+1}),u-y^{k+1}} + r(u)) + D_h( u,u^{k}) - D_h( u,u^{k+1})  \right) \notag \\
& \leq \frac{A_{k}}{A_{k+1}} \Psi(x^{k}) + \frac{\alpha_{k+1}}{A_{k+1}} (f(u) + r(u)) + \frac{1}{A_{k+1}}D_h( u,u^{k})-\frac{1}{A_{k+1}}D_h( u,u^{k+1})  \notag \\
& = \frac{A_{k}}{A_{k+1}} \Psi(x^{k}) + \frac{\alpha_{k+1}}{A_{k+1}} \Psi(u) + \frac{1}{A_{k+1}}D_h( u,u^{k})-\frac{1}{A_{k+1}}D_h( u,u^{k+1}),
\end{align}
where we also used convexity of $f$. Multiplying both sides of the last inequality by $A_{k+1}$, summing these inequalities from $k=0$ to $k=N-1$, and using that $A_{N}-A_0=\sum_{k=0}^{N-1}\alpha_{k+1}$, we obtain
\begin{align}
\label{eq:AGD_Proof_6}
A_N \Psi(x^{N}) \leq A_0\Psi(x^{0}) + (A_{N}-A_0)\Psi(u) + D_h( u,u^{0}) - D_h( u,u^{N}).
\end{align}
Since $A_0=0$, we can choose $u=x^{\ast} \in\argmin\{D_h( u,u^{0})\vert u\in\setX^{\ast}\}\subseteq \setX^{\ast}$ and $D_h( x^{\ast},u^{N})\geq0$, so that, for all $N\geq 1$,
\begin{align}
\label{eq:AGD_Proof_7}
\Psi(x^{N}) - \Psi_{\min}(\setX) &\leq \frac{D_h( x^{\ast},u^{0})}{A_{N}}, \quad D_h( x^{\ast},u^{N}) \leq D_h( x^{\ast},u^{0}).
\end{align}
So, we see from the second inequality that the Bregman distance between the iterates $u^{N}$ and the solution $x^{\ast}$ is non-increasing. 
Then, from the inequality $D_h( x^{\ast},u^{N}) \geq \frac{1}{2}\norm{x^{\ast}-u^{N}}^2$ it follows that $\| x^{\ast}-u^{N}\|$ is bounded for any $N$, which leads to the existence of a subsequence converging to $x^{\ast}$ by the continuity of $\Psi$.
To obtain the convergence rate in terms of the objective residual it remains to estimate the sequence $A_N$ from below.

We prove by induction that $A_k \geq \frac{(k+1)^2}{4L_{f}}$. For $k=1$ this inequality holds as equality since $A_0=0$, and, hence, $A_1=\alpha_1=\frac{1}{L_{f}}$. Let us prove the induction step. From the quadratic equation $A_k+\alpha_{k+1}=L_{f}\alpha_{k+1}^2$, we have
\begin{align}
\label{eq:AGD_Proof_8}
\alpha_{k+1}=\frac{1}{2L_{f}}+\sqrt{\frac{1}{4L_{f}^2}+\frac{A_k}{L_{f}}}\geq \frac{1}{2L}+\sqrt{\frac{A_k}{L_{f}}}\geq \frac{1}{2L_{f}} +\frac{k+1}{2L_{f}} = \frac{k+2}{2L_{f}}.
\end{align}
\begin{align}
\label{eq:AGD_Proof_9}
A_{k+1}=A_k+\alpha_{k+1}\geq \frac{(k+1)^2}{4L_{f}} + \frac{k+2}{2L_{f}} = \frac{k^2+2k+1+2k+4}{4L_{f}} \geq \frac{(k+2)^2}{4L_{f}}.
\end{align}
Thus, combining \eqref{eq:AGD_Proof_9} with \eqref{eq:AGD_Proof_7}, we obtain that the A-BPGM has optimal convergence rate:
\begin{align}
\label{eq:AGD_conv_rate}
\Psi(x^{N}) - \Psi_{\min}(\setX) &\leq \frac{4L_{f}D_h( x^{\ast},u^{0})}{(N+1)^2}.
\end{align}

As it was mentioned above, accelerated gradient method in the form of A-BPGM can serve as a template meta-algorithm for many accelerated algorithms. The examples of accelerated methods which have a close form include primal-dual accelerated methods \cite{tseng2008accelerated,dvurechensky2018computational,lin2019efficient}, random coordinate descent and other randomized algorithms \cite{fercoq2015accelerated,dvurechensky2017randomized,diakonikolas2018alternating}, methods for stochastic optimization \cite{lan2012optimal,dvurechensky2018decentralize}, methods with inexact oracle  \cite{cohen2018acceleration} and inexact model of the objective \cite{gasnikov2019fast,stonyakin2020inexact}. Moreover, only using this one-projection version it was possible to obtain accelerated gradient methods with inexact model of the objective \cite{gasnikov2019fast}, accelerated decentralized distributed algorithms for stochastic convex optimization \cite{gorbunov2019optimal}, and accelerated method for stochastic optimization with heavy-tailed noise \cite{gorbunov2020stochastic}. The key to the last two results is the proof that the sequence generated by the one-projection accelerated gradient method is bounded with large probability, which, to our knowledge, is not possible to prove for other types of accelerated methods applied to stochastic optimization problems.

\subsubsection{Linear Convergence}
Under additional assumptions, we can use the scheme A-BPGM to obtain a linear convergence rate, or, in other words, logarithmic in the desired accuracy complexity bound. One such possible assumption is that $\Psi(x)$ satisfies a quadratic error bound condition for some $\mu>0$:
\begin{equation}
\label{eq:quadr_err_bound}
\Psi(x) -  \Psi_{\min}(\setX) \geq \frac{\mu}{2}\|x-x^{\ast}\|^2.
\end{equation}
This is a weaker assumption than the assumption that $\Psi(x)$ is $\mu$-strongly convex with $\mu>0$. For a review of different additional conditions which allow to obtain linear convergence rate we refer the reader to \cite{necoara2019linear,bolte2017from}. The linear convergence rate can be obtained under quadratic error bound condition by a widely used restart technique, which dates back to \cite{Nes83,nemirovskii1985optimal}, and was extended in the past 20 years to many settings including problems with non-quadratic error bound condition \cite{juditsky2014deterministic,roulet2017sharpness}, stochastic optimization problems \cite{juditsky2014deterministic,ghadimi2013optimal,dvurechensky2016stochastic,gasnikov2016stochasticInter,bayandina2018mirror}, methods with inexact oracle \cite{dvurechensky2016stochastic,gasnikov2016stochasticInter}, randomized methods \cite{allen2016optimal,fercoq2020restarting}, conditional gradient \cite{lan2013complexity,kerdreux2019restarting}, variational inequalities and saddle-point problems \cite{stonyakin2018generalized,stonyakin2020inexact}, methods for constrained optimization problems \cite{bayandina2018mirror}. 

To apply the restart technique, we make several additional assumptions. First, without loss of generality, we assume that $0\in \setX$, $0 = \arg \min_{x \in \setX} h(x)$ and $h(0)=0$. Second, we assume that we are given a starting point $x^{0} \in \setX$ and a number $R_0 >0$ such that $\| x^0 - x^{\ast} \|^2 \leq R_0^2$. Finally, we make the assumption that $h$ is bounded on the unit ball \cite{juditsky2014deterministic} in the following sense. Assume that $x^{\ast}$ is some fixed point and $x$ is such that $\| x-x^{\ast} \|^2 \leq R^2$, then
\begin{equation}
\label{eq:h_bounded}
    h\Big( \frac{x-x^{\ast}}{R} \Big)\leq \frac{\Omega}{2},
\end{equation}
where $\Omega$ is some known number. For example, in the Euclidean setup $\Omega=1$, and other examples are given in \cite[Section 2.3]{juditsky2014deterministic}, where typically $\Omega=O(\ln n)$. 

\begin{algorithm}{\bf{The Restarted Accelerated Bregman Proximal Gradient Method} (R-A-BPGM)}
\\
	{\bf Input:} $z^{0}\in\dom(r)\cap\setX^{\circ}$ such that $\| z^0 - x^{\ast} \|^2 \leq R_0^2$, $\Omega,L_{f},\mu$.\\
	{\bf General step:} For $p=0,1,\ldots$ do:\\
	\qquad Make $N=\left\lceil 2\sqrt{\frac{\Omega L_f}{\mu}}\right\rceil-1$ steps of A-BPGM with starting point $x^{0}=z^{p}$ and proximal setup given by distance-generating function $h_p(x)=R_p^2 h\left(\frac{x-z^{p}}{R_p}\right)$, where $R_p:=R_{p-1}/2=R_0\cdot 2^{-p}$. \\
	\qquad Set $z^{p+1}=x^N$.\\
\end{algorithm}
We next use the above assumptions to show the accelerated logarithmic complexity of R-A-BPGM, i.e. that the number of Bregman proximal steps to find a point $\hat{x}$ such that $f(\hat{x})-f(x^{\ast})\leq \varepsilon$ is proportional to $\sqrt{L_f/\mu}\log_2(1/\varepsilon)$ instead of $(L_f/\mu)\log_2(1/\varepsilon)$ for the BPGM under the error bound condition.
The idea of the proof is to show by induction that, for all $p\geq 0$, $\|z^p-x^{\ast}\|^2\leq R_p^2$. For $p=0$ this holds by the assumption on $z^0$ and $R_0$. So, next we prove an induction step from $p-1$ to $p$.
Using the definition of $h_{p-1}$, assumptions about $h$, and the inductive assumption, we have
\begin{equation}
\label{eq:AGD_SC_Proof_1}
	D_{h_{p-1}}(x^{\ast},z^{p-1}) 
	\leq  h_{p-1}(x^{\ast})=R_{p-1}^2h\Big( \frac{z^{p-1}-x^{\ast}}{R_{p-1}} \Big)\stackrel{\eqref{eq:h_bounded}}{\leq}\frac{\Omega R_{p-1}^2}{2}.
\end{equation}
Thus, applying the error bound condition \eqref{eq:quadr_err_bound}, the bound \eqref{eq:AGD_conv_rate} and our choice of the number of steps $N$, we obtain
\begin{align*}
\frac{\mu}{2}\|z^p-x^{\ast}\|^2 & \stackrel{\eqref{eq:quadr_err_bound}}{\leq}\Psi(z^p) - \Psi_{\min}(\setX) =  \Psi(x^{N}) - \Psi_{\min}(\setX)
\stackrel{\eqref{eq:AGD_conv_rate}}{\leq} \frac{L_{f}D_{h_{p-1}}( x^{\ast},z^{p-1})}{(N+1)^2} 
\stackrel{\eqref{eq:AGD_SC_Proof_1}}{\leq} \frac{L_{f}\Omega R_{p-1}^2}{2 (N+1)^2}\\
&\leq \frac{ \mu R_{p-1}^2}{8} = \frac{ \mu R_{p}^2}{2}.
\end{align*}
So, we obtain that $\|z^p-x^{\ast}\| \leq R_p=R_0 \cdot 2^{-p}$ and $\Psi(z^p) - \Psi_{\min}(\setX) \leq \frac{ \mu R_0^2 \cdot 2^{-2p}}{2}$. To estimate the total number of basic steps of A-BPGM to achieve $\Psi(z^p) - \Psi_{\min}(\setX) \leq \varepsilon$, we need to multiply the sufficient number of restarts $\hat{p} = \left\lceil  \frac{1}{2}\log_2 \frac{\mu R_0^2}{2\varepsilon}\right\rceil$ by the number of A-BPGM steps $N$ in each restart. This leads to the complexity estimate 
$O\left(\sqrt{\frac{\Omega L_f}{ \mu}}\log_2 \frac{\mu R_0^2}{\varepsilon} \right)$ which is optimal \cite{NY83,Nes18} for first-order methods applied to smooth strongly convex optimization problems. 

A possible drawback of the restart scheme is that one has to know an estimate $R_0$ for $\| z^0 - x^{\ast} \|$. It is possible to avoid this by directly incorporating the parameter $\mu$ into the steps of A-BPGM, see e.g. \cite{devolder2013exactness,Nes18,lan2020first,stonyakin2020inexact}. Yet, in this case, a stronger assumption that $\Psi(x)$ is strongly convex or relatively strongly convex \cite{LuFreNes18} is used. The second drawback of both approaches: restart technique and direct incorporation of $\mu$ into the steps, is that they require to know the value of the parameter $\mu$. This is in contrast to non-accelerated BPGM, which using the same step-size as in the non-strongly convex case automatically has linear convergence rate and complexity $O\left({\frac{L_f}{\mu}}\log_2 \frac{\mu R_0^2}{\varepsilon} \right)$, see e.g. \cite{stonyakin2019gradient,stonyakin2020inexact}. Several recipes on how to restart accelerated methods with only rough estimates of the parameter $\mu$ are proposed in \cite{fercoq2020restarting}.

\subsection{Smooth minimization of non-smooth functions}
An important observation made during the last 20 years of development of first-order methods for convex programming is that there is a large gap between the optimal convergence rate for black-box non-smooth optimization problems, i.e. $O(1/\sqrt{N})$ and the optimal convergence rate for black-box smooth optimization problems, i.e. $O(1/N^2)$. For the second observation, let us make a thought experiment. Assume that we minimize a smooth function by $N$ steps of A-BPGM, i.e. solve problem \eqref{eq:Opt} with $r=0$. Then in each iteration we observe first-order information $(f(y^{k+1}), \nabla f(y^{k+1}))$ and can construct a \textit{non-smooth} piecewise linear approximation of $f$ as $g(x)=\max_{k=1,...,N}\{f(y^{k+1}) + \inner{\nabla f(y^{k+1}),x-y^{k+1}}$. If we now make $N$ steps of A-BPGM with the same starting point to minimize $g(x)$, and choose the appropriate subgradients of $g(\cdot)$, the steps will be absolutely the same as when we minimized $f(x)$, and we will be able to minimize a \textit{non-smooth} function $g$ with much faster rate $1/N^2$ than the lower bound $1/\sqrt{N}$. This leads to an idea of trying to find a sufficiently wide class of non-smooth functions which can be efficiently minimized by A-BPGM. 

To do this, one needs to look into the black-box and use the structure of a non-smooth problem to obtain faster convergence rates. The result is known as \emph{Nesterov's smoothing technique} \cite{Nes05}, a powerful tool we are about to describe now. 

Consider the model problem \eqref{eq:Opt}, with the added assumption that the non-smooth part admits a Fenchel representation of the form 
\begin{equation}
r(x)=\max_{w\in\setW}\{\inner{\bA x,w}-\kappa(w)\}.
\end{equation}
Here, $\setW\subseteq\setE$ is a compact convex subset of a finite-dimensional real vector space $\setE$, and $\kappa:\setW\to\R$ is a continuous convex function on $\setW$. $\bA$ is a linear operator from $\setV$ to $\setE^{\ast}$. This additional structure of the problem gives rise to a min-max formulation of \eqref{eq:Opt}, given by 
\begin{equation}
    \label{eq:Sm-ng-Problem}
  \min_{x\in \setX}\max_{w\in\setW} \{f(x) + \inner{\bA x,w} - \kappa(w) \}.  
\end{equation}

The main idea of Nesterov is based on the observation that the function $r$ can be well approximated by a class of smooth convex functions, defined as follows. Let $h_{w}\in\scrH_{1}(\setW)$ with a nonrestrictive assumptions that $\min_{w \in \setW}h_w(w)=0$, and for some $\tau>0$, define the function 
\begin{equation}
    \label{eq:Sm-ng-Smoothed-func}
\Psi_{\tau}(x):= f(x) + \max_{ w \in \setW} \{ \inner{\bA x,w} - \kappa(w) - \tau h_w(w) \}. 
\end{equation}
We denote by $\widehat{w}_{\tau}(x)$ the optimal solution of the maximization problem for a fixed $x$. The main technical lemma, which leads to the main result is as follows.
\begin{proposition}[\cite{Nes05}]
\label{Th:Sm-ng-Diff-f-mu}
The function $\Psi_{\tau}(x)$ is well defined, convex and continuously differentiable at any $x\in \setX$ with $\nabla \Psi_{\tau}(x) = \nabla f(x) + \bA^{\ast} \widehat{w}_{\tau}(x)$. Moreover, $\nabla \Psi_{\tau}(x)$ is Lipschitz continuous with constant $L_{\tau} = L_f + \frac{\|\bA\|_{\setV,\setE}^2}{\tau}$.
\end{proposition}
Here the adjoint operator $\bA^{\ast}$ is defined by equality $\inner{\bA x,w}_{\setE}=\inner{\bA^{\ast}w,x}_{\setV}$ and the norm of the operator $\|\bA\|_{\setV,\setE}$ is defined by $\|\bA\|_{\setV,\setE} = \max_{x,w}\{\inner{\bA x,w}:\|x\|_{\setV} = 1, \|w\|_{\setE} = 1 \}$.
Since $\setW$ is bounded, $\Psi_{\tau}(x)$ is a uniform approximation for the function $\Psi$, namely, for all $x \in \setX$,
\begin{equation}
\label{eq:Sm-ng-Unif-Appr}
    \Psi_{\tau}(x) \leq \Psi(x) \leq \Psi_{\tau}(x) + \tau D_{\setW},
\end{equation}
where $D_{\setW}:= \max \{h_w(w) \vert w \in \setW \}$, assumed to be a finite number. Then, the idea is to choose $\tau$ sufficiently small and apply accelerated gradient method to minimize $\Psi_{\tau}(x)$ on $\setX$ with a DGF $h_{x}\in\scrH_{1}(\setX)$. Doing this, and assuming that $D_{\setX}=\max\{h_{x}(u)\vert u\in\setX\}<\infty$, we can apply the result \eqref{eq:AGD_conv_rate} to $\Psi_{\tau}(x)$ and, using \eqref{eq:Sm-ng-Unif-Appr}, to obtain
\begin{align*}
 0 \leq \Psi(x^N) - \Psi_{\min}(\setX)  & \leq  \Psi_{\tau}(x^N) + \tau D_{\setW} - \Psi_{\tau}(x^*) \leq  \Psi_{\tau}(x^N) + \tau D_{\setW} - \Psi_{\tau}(x_{\tau}^*) \leq  \tau D_{\setW} + \frac{4L_{\tau}D_{\setX}}{(N+1)^2}
\\ 
& = \tau D_{\setW} + \frac{4\|\bA\|_{\setV,\setE}^2D_{\setX}}{\tau(N+1)^2}+ \frac{4L_fD_{\setX}}{(N+1)^2}.
\end{align*}
Choosing $\tau$ to minimize the r.h.s., i.e. $\tau = \frac{2\|\bA\|_{\setV,\setE}}{N+1} \sqrt{\frac{D_{\setX}}{D_{\setW}}}$, we obtain
\begin{equation}
    0 \leq \Psi(x^N) - \Psi_{\min}(\setX) \leq \frac{4\|\bA\|_{\setV,\setE} \sqrt{D_{\setX}D_{\setW}}}{N+1} + \frac{4L_fD_{\setX}}{(N+1)^2}.
\end{equation}

A more careful analysis in the proof of \cite[Theorem 3]{Nes05}, allows also to obtain an approximate solution to the conjugate problem 
\begin{equation}
  \max_{ w \in \setW}  \{ \psi(w) := - \kappa(w) + \min_{x\in \setX} \left( \inner{\bA x,w} + f(x)  \right)\}. 
\end{equation}
In each iteration of A-BPGM, the optimizer needs to calculate $\nabla \Psi_{\tau}(y^{k+1})$, which requires to calculate $\widehat{w}_{\tau}(y^{k+1})$. This information is aggregated to obtain the vector $\widehat{w}^{N} = \sum_{k=0}^{N-1} \frac{\alpha_{k+1}}{A_{k+1}}\widehat{w}_{\tau}(y^{k+1})$ and is used to obtain the following  primal-dual result 
\begin{equation}
\label{eq:Sm-ng-final-Conve-Rate}
    0 \leq \Psi(x^N) - \Psi_{\min}(\setX) \leq \Psi(x^N) - \psi(\widehat{w}^{N})  \leq \frac{4\|\bA\|_{\setV,\setE} \sqrt{D_{\setX}D_{\setW}}}{N+1} + \frac{4L_fD_{\setX}}{(N+1)^2}.
\end{equation}
In both cases using the special structure of the problem it is possible to obtain convergence rate $O(1/N)$ for non-smooth optimization, which is better than the lower bound $O(1/\sqrt{N})$ for general non-smooth optimization problems.

We illustrate the smoothing technique by two examples of piecewise-linear minimization.

\begin{example}[Uniform fit]
Consider the problem of finding a uniform fit of some signal $b\in\setE$, given linear observations $\bA x$. where $\bA:\setV\to\setE$ is a bounded linear operator. This problem amount to minimize the non-smooth function $\norm{\bA x-b}_{\infty}$. Of course, this problem can be equivalently formulated as an LP, however in case where the dimensionality of the parameter vector $x$ is large, such a direct approach could turn out to be not very practical. Adopting the just introduced smoothing technology, the representation \eqref{eq:Sm-ng-Problem} can be obtained using the definition of the dual norm $\norm{\cdot}_{1}$, i.e. $\norm{\bA x-b}_{\infty} = \max_{w:\norm{w}_1 \leq 1} \inner{\bA x-b,w}$. Yet, a better representation is obtained using the unit simplex $\setW=\{w\in\R^{2m}_{+}\vert \sum_{i=1}^{2m}w_{i}=1\}$, matrix $\hat{\bA}=[\bA;-\bA]$, and vector $\hat{b}=[b;-b]$. For the set $\setW$, a natural Bregman setup is the norm $\norm{w}_{\setE}=\norm{w}_1$ and the Boltzmann-Shannon entropy $h_w(w)=\ln 2m + \sum_{i=1}^m w_i\ln w_i$. This gives
\[
\Psi_{\tau}(x) = \max_{ w \in \setW} \{ \inner{\hat{\bA}x-\hat{b},w}- \tau h_w(w) \} = \tau \ln \left(\frac{1}{2m}\sum_{i=1}^m \exp\left(\frac{\inner{a_i,x}-b_i}{\tau}\right) + \exp\left(-\frac{\inner{a_i,x}-b_i}{\tau}\right) \right), 
\]
which is recognized as a softmax function.
\end{example}
\begin{example}[$\ell_{1}$-fit]
In compressed sensing \cite{CanRomTao06,Don06,CanTao07} one encounters the problem to minimize the $\ell_{1}$ norm of the residual vector $\bA x-b$ over a given closed convex set $\setX$. While it is well-known that this problem can in principle again be reformulated as an LP, the typical high-dimensionality of such problems makes this direct approach often not practicable. Adopting the smoothing technology, it is natural to choose $\setW=\{w\in\R^{m}\vert \|w\|_{\infty}\leq 1\}$ and $h_w(w)=\frac{1}{2}\sum_{i=1}^m\|a_i\|_{\setE,\ast}w_i^2$, which gives
\[
\Psi_{\tau}(x) = \max_{ w \in \setW} \{ \inner{{\bA}x-{b},w}- \tau h_w(w) \} =  \sum_{i=1}^m\|a_i\|_{x,\ast} \psi_{\tau}\left(\frac{|\inner{a_i,x}-b_i|}{\|a_i\|_{\setE,\ast}}\right),
\]
where $\psi_{\tau}(t)$ is the Huber function equal to $t^2/(2\tau)$ for $0 \leq t \leq \tau$ and $t-\tau/2$ if $t \geq \tau$.

For the particular case of smoothing the absolute value function $|x|$, Figure \ref{Fig:smoothing} gives the plot of the original function, its softmax smoothing and Huber smoothing, both with $\tau=1$. Potentially, other ways of smoothing a non-smooth function can be applied, see \cite{beck2012smoothing} for a general framework.
\end{example}

\begin{minipage}[t]{0.48\linewidth}
		\includegraphics[width=0.9\linewidth]{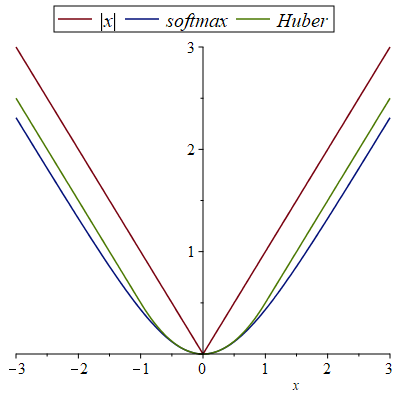}
	\captionof{figure}{Absolute value function $|x|$, its softmax smoothing and Huber smoothing, both with $\tau=1$.}
	\label{Fig:smoothing}
\end{minipage}
\vspace{2em}
\begin{minipage}[t]{0.48\linewidth}
		\includegraphics[width=0.9\linewidth]{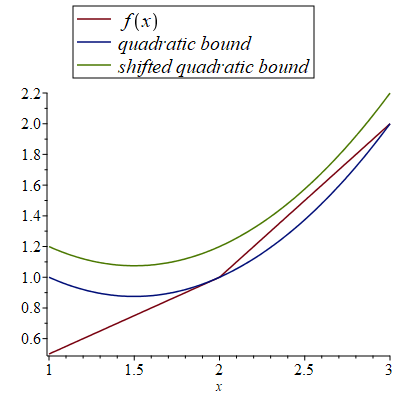}
	\captionof{figure}{Non-smooth function $f(x) = \max\{x-1,x/2\}$, a quadratic function constructed using the first-order information at the point $x=2$, and a shifted quadratic function constructed using the first-order information at the point $x=2$. As one can see, adding a shift allows to obtain an upper quadratic bound for the objective, which is then minimized to obtain a new test point.}
	\label{Fig:universal}
\end{minipage}

\paragraph{Closing Remarks}
Let us make several remarks on the related literature. A close approach is proposed in \cite{Nem04}, where the problem \eqref{eq:Sm-ng-Problem} is considered directly as a min-max saddle-point problem. These classes of equilibrium problems are typically solved via tools from monotone variational inequalities, whose performance is typically worse than the performance of optimization algorithms. In particular, contrasting the above rate estimate with the one reported in \cite{Nem04}, one observes that the bound in \cite{Nem04} has a similar to \eqref{eq:Sm-ng-final-Conve-Rate}  structure, yet with the second term being non-accelerated, i.e. proportional to $1/N$. This approach was generalized to obtain an accelerated method for a special class of variational inequalities in \cite{CheLanOuy17}, where an optimal iteration complexity $O(L/\sqrt{\eps})$ to reach an $\eps$-close solution is reported. In the original paper \cite{Nes05}, the smoothing parameter is fixed and requires to know the parameters of the problem in advance. This has been improved in \cite{nesterov2005excessive}, where an adaptive version of the smoothing techniques is proposed. This framework was extended in \cite{tran-dinh2014constrained,tran-dinh2015smooth,alacaoglu2017smooth,tran-dinh2020adaptive} for structured composite optimization problems in the form \eqref{eq:structured} and a related primal-dual representation \eqref{eq:saddle}. A related line of works studies minimization of strongly convex functions under linear constraints. Similarly to \eqref{eq:Sm-ng-Smoothed-func} the objective in the Lagrange dual problem has Lipschitz gradient, yet the challenge is that the feasible set in the dual problem is not bounded. Despite that it is possible to obtain accelerated primal-dual methods  
\cite{tran-dinh2014constrained,tran-dinh2015smooth,chernov2016fast,dvurechensky2016primal-dual,anikin2017dual,dvurechensky2018computational,dvurechensky2020stable,guminov2019accelerated,guminov2019acceleratedAM,kroshnin2019complexity,nesterov2020primal-dual,ivanova2020composite}.
In particular, this allows to obtain improved complexity bounds for different types of optimal transport problems 
\cite{dvurechensky2018computational,lin2019efficient,kroshnin2019complexity,dvurechensky2018decentralize,uribe2018distributed,2019arXiv191000152L,2020arXiv200204783L,tupitsa2020multimarginal,krawtschenko2020distributed}.

\subsection{Universal Accelerated Method}
\label{sec:Universal}
As it was discussed in the previous subsection, there is a gap in the convergence rate between the class of non-smooth convex optimization problems and the class of smooth convex optimization problems. In this subsection, we present a unifying framework \cite{Nes15} for these two classes which allows to obtain uniformly optimal complexity bounds for both classes by a single method without the need to know whether the objective is smooth or non-smooth. To do that, consider the Problem \eqref{eq:Opt} with $f$ which belongs to the class of functions with H\"older-continuous subgradients, i.e. for some $L_{\nu} >0 $ and $\nu \in [0,1]$ it holds that $\norm{\nabla f(x) - \nabla f(y)}_* \leq L_{\nu} \norm{x-y}^{\nu}$ for all $x,y \in \dom f$. If $\nu=1$, we recover the $L_f$-smoothness condition \eqref{eq:fsmooth}. If $\nu=0$ we have that $f$ has bounded variation of the subgradient, which is essentially equivalent to the bounded subgradient Assumption \ref{ass:BoundedSubgrad}. The main observation \cite{devolder2014first,Nes15} is that this H\"older condition allows to prove an inexact version of the "descent Lemma" inequality \eqref{eq:DLclassical}. More precisely \cite[Lemma 2]{Nes15}, for any $x, y \in \dom f$ and any $\delta >0$, 
\begin{equation}
\label{eq:Holder_upper_bound}
f(y)\leq f(x)+\inner{\nabla f(x),y-x}+\frac{L_{\nu}}{1+\nu}\norm{y-x}^{1+\nu} \leq f(x)+\inner{\nabla f(x),y-x}+\frac{L}{2}\norm{y-x}^{2} +\delta,
\end{equation}
where 
\begin{equation}
\label{eq:L_delta}
L \geq L(\delta) := \left(\frac{1-\nu}{1+\nu}\frac{1}{\delta} \right)^{\frac{1-\nu}{1+\nu}} L_{\nu}^{\frac{2}{1+\nu}}
\end{equation}
with the convention that $0^0=1$. We illustrate this by Figure~\ref{Fig:universal} where we plot a quadratic bound in the r.h.s. of \eqref{eq:Holder_upper_bound} with $\delta=0$ and a shifted quadratic bound in the r.h.s. of \eqref{eq:Holder_upper_bound} with some $\delta >0$. The first quadratic bound can not be an upper bound for $f(y)$ for any $L >0$, and the positive shift allows to construct an upper bound. Thus, it is sufficient to equip the A-BPGM with a backtracking line-search to obtain a universal method.

\begin{algorithm}{\bf{The Universal Accelerated Bregman Proximal Gradient Method} (U-A-BPGM)}
\\
	{\bf Input:} Pick $x^{0}=u^{0}=y^{0}\in\dom(r)\cap\setX^{\circ}$, $\varepsilon >0$, $0< L_0 < L(\varepsilon/2)$, set $A_0=0$\\
	{\bf General step:} For $k=0,1,\ldots$ do:\\
	\qquad Find the smallest integer $i_k \geq 0$ such that if one defines $\alpha_{k+1}$ from quadratic equation $A_k+\alpha_{k+1}=2^{i_k-1}L_k\alpha_{k+1}^2$, sets $A_{k+1}=A_k+\alpha_{k+1}$, \\
	\qquad sets $y^{k+1} = \frac{\alpha_{k+1}}{A_{k+1}}u^{k}+\frac{A_{k}}{A_{k+1}}x^{k}$,\\
	\qquad sets $u^{k+1}=\argmin_{x\in\setX}\left\{\alpha_{k+1}\left( f(y^{k+1}) + \inner{\nabla f(y^{k+1}),x-y^{k+1}} +r(x) \right) + D_{h}(x,u^{k}) \right\},$ \\
	\qquad sets $x^{k+1} = \frac{\alpha_{k+1}}{A_{k+1}}u^{k+1}+\frac{A_{k}}{A_{k+1}}x^{k}$,\\
	\qquad then it holds that $f(x^{k+1})\leq f(y^{k+1})+\inner{\nabla f(y^{k+1}),x^{k+1}-y^{k+1}}+\frac{2^{i_k-1}L_k}{2}\norm{x^{k+1}-y^{k+1}}^{2} +\frac{\varepsilon \alpha_{k+1}}{2A_{k+1}}$.\\
	\qquad Set $L_{k+1}=2^{i_k-1}L_k$ and go to the next iterate $k$.
\end{algorithm}
We first observe that for sufficiently large $i_k$,  $2^{i_k-1}L_k \geq L\left(\frac{\varepsilon \alpha_{k+1}}{2A_{k+1}}\right)$, see \cite[p.396]{Nes15}. This means that the process of finding $i_k$ is finite since the condition which is checked for each $i_k$ is essentially \eqref{eq:Holder_upper_bound} with $\delta= \frac{\varepsilon \alpha_{k+1}}{2A_{k+1}}$. Further, the convergence proof follows the same steps as the proof of the convergence rate for A-BPGM. The first thing which is changed is equation \eqref{eq:AGD_Proof_1}, where now the inexact descent Lemma is used instead of the exact one. The only difference is that $L_f$ is changed to its local approximation $L_{k+1}$ and add the error term $\frac{\varepsilon \alpha_{k+1}}{2A_{k+1}}$ appears in the r.h.s. In \eqref{eq:AGD_Proof_2} the new quadratic equation with $L_{k+1}$ is used and the inequality remains the same. This eventually leads to \eqref{eq:AGD_Proof_5} with the only change being an additive error term $\frac{\varepsilon \alpha_{k+1}}{2A_{k+1}}$ in the r.h.s. Finally, this leads to the bound  
\[
\Psi(x^{N}) - \Psi_{\min}(\setX) \leq \frac{D_h(u^{\ast},u^0)}{A_{N}} + \frac{\varepsilon}{2}.
\]
After some algebraic manipulation, Nesterov \cite[p.397]{Nes15} obtains an inequality $A_N \geq \frac{N^{\frac{1+3\nu}{1+\nu}}\varepsilon^{\frac{1-\nu}{1+\nu}}}{2^{\frac{2+4\nu}{1+\nu}}L_{\nu}^{\frac{2}{1+\nu}}}$. Substituting, we obtain
\[
\Psi(x^{N}) - \Psi_{\min}(\setX) \leq \frac{2^{\frac{2+4\nu}{1+\nu}}D_h(u^{\ast},u^0)L_{\nu}^{\frac{2}{1+\nu}}}{N^{\frac{1+3\nu}{1+\nu}}\varepsilon^{\frac{1-\nu}{1+\nu}}} + \frac{\varepsilon}{2}.
\]
Since the method does not require to know $\nu$ and $L_{\nu}$, the iteration complexity to achieve accuracy $\varepsilon$ is
\[
N=O\left( \inf_{\nu \in [0,1]} \left(\frac{L_{\nu}}{\varepsilon}\right)^{\frac{2}{1+3\nu}} \left(D_h(u^{\ast},u^0)\right)^{\frac{1+\nu}{1+3\nu}} \right).
\]
It is easy to see that the oracle complexity, i.e. the number of proximal operations, is approximately the same. Indeed, the number of oracle calls for each $k$ is $2(i_k+1)$. Further, $L_{k+1}=2^{i_k-1}L_k$, which means that the total number of the oracle calls up to iteration $N$ is 
$\sum_{k=0}^{N-1} 2(i_k+1) = \sum_{k=0}^{N-1} 2(2 \log_2\frac{L_{k+1}}{L_k}) = 4N + 2\log_2\frac{L_{N}}{L_0}$, i.e. is, up to a logarithmic term, four times larger than $N$. The obtained oracle complexity coincides up to a constant factor with the lower bound  \cite{NY83} for first-order methods applied to minimization of functions with H\"older-continuous gradients. In the particular case $\nu=0$, we obtain the complexity $O\left(\frac{L_0^2D_h(u^{\ast},u^0)}{\varepsilon^2} \right)$, which corresponds to the convergence rate $1/\sqrt{k}$, which is typical for general non-smooth minimization. In the opposite case of smooth minimization corresponding to $\nu=1$, we obtain the complexity $O\left(\sqrt{\frac{L_1D_h(u^{\ast},u^0)}{\varepsilon}} \right)$, which corresponds to the optimal convergence rate $1/k^2$. The same idea can be used to obtain universal version of the BPGM method \cite{Nes15}. One can also use the strong convexity assumption to obtain faster convergence rate of the U-A-BPGM either by restarts \cite{roulet2017sharpness,kamzolov2020universal}, or by incorporating the strong convexity parameter in the steps \cite{stonyakin2020inexact}. The same backtracking line-search can be applied in a much simpler way if one knows that $f$ is $L_f$-smooth with some unknown Lipschitz constant or to achieve acceleration in practice caused by a pessimistic estimate for $L_f$ \cite{Nes13,tran-dinh2015smooth,dvurechensky2016primal-dual,dvurechensky2018computational,malitsky2018first-order,dvinskikh2020line-search,dvurechensky2020stable}. The idea is to use standard exact "descent Lemma" inequality in each step of the accelerated method.

The idea of universal methods turned out to be very productive and several extensions has been proposed in the literature including universal primal-dual method for composite optimization \cite{baimurzina2019universal}, universal primal-dual method \cite{yurtsever2015universal} for problems with linear constraints and problems in the form \eqref{eq:structured}, universal method for convex and non-convex optimization \cite{ghadimi2019generalized}, a universal primal-dual hybrid of accelerated gradient method with conjugate gradient method using additional one-dimensional minimization \cite{nesterov2020primal-dual}. Extensions are also known for first-order methods for variational inequalities and saddle-point problems \cite{stonyakin2018generalized}. The above-described method is not the only way to obtain adaptive and universal methods for smooth and non-smooth optimization problems. An alternative way which uses the norm of the current (sub)gradient to define the step-size  was initiated probably by \cite{Pol87} and became very popular in stochastic optimization for machine learning after the paper \cite{duchi2011adaptive}. On this avenue it was possible to obtain for $\nu\in\{0,1\}$  universal accelerated optimization method \cite{levy2018online} and universal methods for variational inequalities and saddle-point problems \cite{bach2019universal,antonakopoulos2020adaptive}.


\subsection{Connection between Accelerated method and Conditional Gradient}
In this subsection we describe how a variant of conditional gradient method can be obtained as a particular case of A-BPGM with inexact Bregman Proximal step.
Since we consider conditional gradient method it is natural to assume that the set $\setX$ is bounded with $\max_{x,u \in \setX}D_h(x,u) \leq D_{\setX}$.
We follow the idea of \cite{ben-tal2020lectures} where the main observation of is that the Prox-Mapping in A-BPGM can be calculated inexactly by applying the generalized linear oracle given in Definition \ref{Def:gen_lin_or}. The idea is very similar to the idea of the conditional gradient sliding described in Section \ref{S:S-CG} with the difference that here we implement an approximate Bregman Proximal step using only \textit{one} step of the generalized conditional gradient method. The resulting algorithm is listed below with the only difference with A-BPGM being the change of the Bregman Proximal step $u^{k+1}= \scrP_{\alpha_{k+1} r}(u^{k},\alpha_{k+1}\nabla f(y^{k+1}))$ to the step $u^{k+1}=\mathcal{L}_{\setX,\alpha_{k+1}r}(\alpha_{k+1}\nabla f(y^{k+1}))$ given by generalized linear oracle.  
\begin{algorithm}{\bf{Conditional Gradient Method by A-BPGM with Approximate Bregman Proximal Step}}
\\
	{\bf Input:} pick $x^{0}=u^{0}=y^{0}\in\dom(r)\cap\setX^{\circ}$, set $A_0=0$\\
	{\bf General step:} For $k=0,1,\ldots$ do:\\
	\qquad Find $\alpha_{k+1}$ from quadratic equation $A_k+\alpha_{k+1}=L_{f}\alpha_{k+1}^2$. Set $A_{k+1}=A_k+\alpha_{k+1}$.\\
	\qquad Set $y^{k+1} = \frac{\alpha_{k+1}}{A_{k+1}}u^{k}+\frac{A_{k}}{A_{k+1}}x^{k}$.\\
	\qquad Set (Approximate Bregman proximal step by generalized linear oracle) $u^{k+1}=\argmin_{x\in\setX}\left\{\alpha_{k+1}\left( f(y^{k+1}) + \inner{\nabla f(y^{k+1}),x-y^{k+1}} +r(x) \right) \right\} = \mathcal{L}_{\setX,\alpha_{k+1}r}(\alpha_{k+1}\nabla f(y^{k+1}))$.\\
	\qquad Set $x^{k+1} = \frac{\alpha_{k+1}}{A_{k+1}}u^{k+1}+\frac{A_{k}}{A_{k+1}}x^{k}$.\\
\end{algorithm}

Since the difference between such conditional gradient method and A-BPGM is in one simple change of the step for $u^{k+1}$, to obtain the convergence rate of the former, it is sufficient to track, what changes such approximate Bregman Proximal step entails in the convergence rate proof for A-BPGM.
In other words, we need to understand what happens with the proof for A-BPGM if the Bregman Proximal step is made inexactly by applying the generalized linear oracle.
The first important difference is that we need an inexact version of inequality \eqref{eq:r}, which was used in the convergence proof of A-BPGM and which the result of the exact Bregman Proximal step. 
To obtain its inexact version, let us denote 
\[
\varphi(x)=\alpha_{k+1}\left( f(y^{k+1}) + \inner{\nabla f(y^{k+1}),x-y^{k+1}} +r(x) \right).
\] 
Then generalized linear oracle actually minimizes this function on the set $\setX$ to obtain $u^{k+1}$. Thus, by the optimality condition, we have that there exists $\xi \in \partial \varphi(u^{k+1})$ such that $\inner{\xi,u^{k+1}-x}\leq 0$ for all $x \in \setX$. 
Now we remind that the Bregman Proximal step in A-BPGM minimizes $\varphi(x)+D_h(x,u^{k})$.
These observations allow to estimate the inexactness of the Bregman Proximal step implemented via generalized linear oracle.
Indeed, for $u^{k+1}=\mathcal{L}_{\setX,\alpha_{k+1}r}(\alpha_{k+1}\nabla f(y^{k+1}))$
\begin{align}
\label{eq:AGD_CG_Proof_1}
\inner{\xi+\nabla h(u^{k+1})-\nabla h(u^{k}),u^{k+1}-x} \leq \inner{\nabla h(u^{k+1})-\nabla h(u^{k}),u^{k+1}-x} \notag\\
= -D_{h}(x,u^{k}) + D_{h}(x,u^{k+1})+ D_{h}(u^{k+1},u^{k} ) \leq 2D_{\setX},
\end{align}
where we used three-point identity in Lemma \ref{Lm:3-point}. This inequality provides inexact version of the optimality condition  \eqref{eq:sub_prob_opt_cond} in the problem $\min_{x \in \setX} \{\varphi(x)+D_h(x,u^{k}) \}$, i.e. \eqref{eq:sub_prob_opt_cond_inexact} with $\Delta=2D_{\setX}$. This in order leads to \eqref{eq:r_inexact} with $\Delta=2D_{\setX}$, which is the desired inexact version of \eqref{eq:r}.

Let us now see, how this affects the convergence rate proof of A-BPGM. Inequality \eqref{eq:r} was used in the analysis only in \eqref{eq:AGD_Proof_5}. This means that the change of \eqref{eq:r} to \eqref{eq:r_inexact} with $\Delta=2D_{\setX}$ leads to an additive term $\frac{2D_{\setX}}{A_{k+1}}$ in the r.h.s. of \eqref{eq:AGD_Proof_5}:
\begin{align}
\label{eq:AGD_CG_Proof_2}
\Psi(x^{k+1}) &\leq \frac{A_{k}}{A_{k+1}} \Psi(x^{k}) + \frac{\alpha_{k+1}}{A_{k+1}} \Psi(u) + \frac{1}{A_{k+1}}D_h( u,u^{k})-\frac{1}{A_{k+1}}D_h( u,u^{k+1}) +\frac{2D_{\setX}}{A_{k+1}}, \quad u \in \setX.
\end{align}
Multiplying both sides of the last inequality by $A_{k+1}$, summing these inequalities from $k=0$ to $k=N-1$, and using that $A_{N}-A_0=\sum_{k=0}^{N-1}\alpha_{k+1}$, we obtain
\begin{align}
\label{eq:AGD_CG_Proof_3}
A_N \Psi(x^{N}) \leq A_0\Psi(x^{0}) + (A_{N}-A_0)\Psi(u) + D_h( u,u^{0}) - D_h( u,u^{N}) + 2ND_{\setX}.
\end{align}
Since $A_0=0$, we can choose $u=x^{\ast} \in\argmin\{D_h( u,u^{0})\vert u\in\setX^{\ast}\}$, so that, for all $N\geq 1$,
\[
\Psi(x^{N}) - \Psi_{\min}(\setX) \leq  \frac{D_h( x^{\ast},u^{0})}{A_{N}} + \frac{2ND_{\setX}}{A_{N}} \leq \frac{D_{\setX}}{A_{N}} + \frac{2ND_{\setX}}{A_{N}},
\]
which, given the lower bound $A_N \geq \frac{(N+1)^2}{4L_f}$ leads to the final result for the convergence rate of this inexact A-BPGM implemented via generalized linear oracle:
\[
\Psi(x^{N}) - \Psi_{\min}(\setX) \leq \frac{4L_{f}D_{\setX}}{(N+1)^2} + \frac{8L_{f}D_{\setX}}{N+1}.
\]
Thus, we obtain a variant of conditional gradient method with the same convergence rate $1/N$ as for the standard conditional gradient method. Using the same approach, but with U-A-BPGM as the basis method, one can obtain a universal version of conditional gradient method \cite{stonyakin2020inexact} for minimizing objectives with H\"older-continuous gradient. The bounds in this case a similar to the ones obtained in a more direct universal method in \cite{nesterov2018complexity}. Similar bounds were also recently obtained in \cite{zhao2020analysis}.

\section{Conclusion}
\label{sec:conclusion}

We close this survey, with a very important fact which Nesterov writes in the introduction of his important textbook \cite{Nes18}: \emph{in general, optimization problems are unsolvable.} Convex programming stands out from this general fact, since it describes a significantly large class of model problems, with important practical applications, for which general solution techniques have been developed within the mathematical framework of interior-point techniques. However, modern optimization problems are large-scale in nature, which renders these polynomial time methods impractical. First-order methods have become the gold standard in balancing cheap iterations with low solution accuracy, and many theoretical and practical advances having been made in the last 20 years. 

Despite the fact that convex optimization is approaching the state of being a primitive similar to linear algebra techniques, we foresee that the development of first-order methods has not come to a halt yet. In connection with stochastic inputs, the combination of acceleration techniques with other performance boosting tricks, like variance reduction, incremental techniques, as well as distributed optimization, still promises to produce some new innovations. On the other hand, there is also still much room for improvement of algorithms for optimization problems which do not admit a prox-friendly geometry. Distributed optimization, in particular in the context of federated learning is now a very active area of research, see  \cite{SurveyFederated} for a recent review of federated learning and \cite{gorbunov2020recent} for a recent review of distributed optimization. Another important focus in the research in optimization methods is now on numerical methods for non-convex optimization motivated by training of deep neural networks, see \cite{sun2019optimization,danilova2020recent} for a recent review. A number of open questions remain in the theory of first-order methods for variational inequalities and saddle-point problems, mainly in the case of variational inequalities with non-monotone operators. In particular, recently the authors of \cite{cohen2020relative} observed a connection between extragradient methods for monotone variational inequalities and accelerated first-order methods.  Thus, as we emphasize in this survey, new connections, that are still continuously being discovered between different methods and different formulations, can lead to new understanding and developments in this lively field of first-order methods.

\section*{Acknowledgements} 
The authors are grateful to Yu. Nesterov and A. Gasnikov for fruitful discussions.
M. Staudigl thanks the COST Action CA16228 (European Network for Game Theory), the FMJH Program PGMO and from the support of EDF (Project "Privacy preserving algorithms for distributed control of energy markets") for its support. 

\section*{References}


\end{document}